\documentclass{amsart}
\usepackage{amssymb, amsmath,amsthm }
\usepackage{hyperref}
\usepackage{stmaryrd}
\usepackage[all]{xy}
\usepackage{dsfont}

\newtheorem{prop}{Proposition}[section]

\newtheorem{lem}[prop]{Lemma}
\newtheorem{thm}[prop]{Theorem}

\newtheorem{cor}[prop]{Corollary}

\theoremstyle{remark}
\newtheorem{remar}[prop]{Remark}
\theoremstyle{definition}
\newtheorem{defi}[prop]{Definition}

\DeclareMathAlphabet{\mathpzc}{OT1}{pzc}{m}{it}
\DeclareMathOperator{\Aut}{Aut}
\DeclareMathOperator{\End}{End}
\DeclareMathOperator{\Hom}{Hom}

\DeclareMathOperator{\Ind}{Ind}

\DeclareMathOperator{\Sym}{Sym}

\DeclareMathOperator{\GL}{GL}
\DeclareMathOperator{\SL}{SL}
\DeclareMathOperator{\Ker}{Ker}

\DeclareMathOperator{\WD}{WD}
\DeclareMathOperator{\Gal}{Gal}

\DeclareMathOperator{\soc}{soc}

\DeclareMathOperator{\tr}{tr}

\DeclareMathOperator{\Irr}{Irr}

\DeclareMathOperator{\Spec}{Spec}
\DeclareMathOperator{\mSpec}{m-Spec}

\DeclareMathOperator{\id}{id}
\DeclareMathOperator{\Mod}{Mod}

\DeclareMathOperator{\Sp}{Sp}

\DeclareMathOperator{\Ext}{Ext}

\DeclareMathOperator{\dt}{det}
\DeclareMathOperator{\Ban}{Ban}

\DeclareMathOperator{\dualcat}{\mathfrak C}

\DeclareMathOperator{\Ord}{Ord}

\newcommand{\Indu}[3]{\Ind_{#1}^{#2}{#3}}

\newcommand{\Q}{\mathbb{Q}}
\newcommand{\Qp}{\mathbb {Q}_p}
\newcommand{\Zp}{\mathbb{Z}_p}
\newcommand{\Qpbar}{\overline{\mathbb{Q}}_p}

\newcommand{\NN}{\mathbb N}
\newcommand{\Eins}{\mathbf 1}

\newcommand{\ZZ}{\mathbb Z}
\newcommand{\VV}{\mathbf V}
\newcommand{\CC}{\mathbb C}

\newcommand{\QQ}{\mathbb Q}

\newcommand{\RR}{\mathbb R}
 
\newcommand{\mm}{\mathfrak m}

\newcommand{\st}{\mathrm{st}}

\newcommand{\wP}{\widetilde{P}}
\newcommand{\wE}{\widetilde{E}}

\newcommand{\OO}{\mathcal O}

\newcommand{\TT}{\mathbb T}

\newcommand{\gal}{G_{\Qp}}

\newcommand{\ad}{\mathrm {ad}}

\DeclareMathOperator{\wtimes}{\widehat{\otimes}}

\newcommand{\cV}{\check{\mathbf{V}}}
\DeclareMathOperator{\Ad}{Ad}

\newcommand{\BB}{\mathfrak B}

\newcommand{\nn}{\mathfrak n}

\newcommand{\md}{\mathrm m}

\newcommand{\pp}{\mathfrak p}

\newcommand{\br}[1]{\llbracket #1\rrbracket}
\newcommand{\qq}{\mathfrak{q}}

\newcommand{\sm}{\mathrm{sm}}

\newcommand{\pro}{\mathrm{pro}}

\newcommand{\adm}{\mathrm{adm}}
\newcommand{\alg}{\mathrm{alg}}
\newcommand{\cont}{\mathrm{cont}}

\newcommand{\ladm}{\mathrm{l.adm}}

\newcommand{\wB}{\mathrm{B}}

\newcommand{\rhobar}{\bar{\rho}}
\newcommand{\lfin}{\mathrm{l.fin}}

\newcommand{\univ}{\mathrm{univ}}
\newcommand{\wt}{\mathbf w} 
\newcommand{\cris}{\mathrm{cr}}
\DeclareMathOperator{\LLL}{LL}

\DeclareMathOperator{\Spf}{Spf}
\newcommand{\AfF}{\mathbb{A}_F^f}
\newcommand{\inv}{\mathrm{inv}}
\newcommand{\Nm}{\mathrm{Nm}}
\newcommand{\cyc}{\mathrm{cyc}}
\newcommand{\Frob}{\mathrm{Frob}}
\newcommand{\sigmabar}{\bar{\sigma}}
\newcommand{\psibar}{\bar{\psi}}
\newcommand{\tsigma}{\tilde{\sigma}}
\newcommand{\tEins}{\tilde{\Eins}}
\newcommand{\rhobarss}{\rhobar^{\mathrm{ss}}}
\newcommand{\ps}{\mathrm{ps}}
\newcommand{\CNL}{\mathrm{CNL}}
\newcommand{\gm}{(\widehat{\mathbb G}_m[2])^t}

\title[On $2$-dimensional $2$-adic Galois representations]{On $2$-dimensional $2$-adic Galois representations of local and global fields}
\author{Vytautas Pa\v{s}k\={u}nas}
\date{\today.}
\begin{document} 

\maketitle
\begin{abstract} We describe the generic blocks in  the category of smooth locally admissible mod $2$ representations  of $\GL_2(\QQ_2)$. 
As an application we  obtain new cases of Breuil--M\'ezard and Fontaine--Mazur conjectures for $2$-dimensional $2$-adic Galois representations. 
\end{abstract}
\tableofcontents

\section{Introduction}\label{intro}

Let $p$ be a prime and let $L$ be a finite extension of $\QQ_p$ with the ring of integers $\OO$ and uniformizer $\varpi$. We prove the following modularity lifting theorem. 

\begin{thm}\label{intro1} Assume that $p=2$. Let $F$ be a totally real field where $2$ is totally split, let $S$ be a finite set of places of $F$ containing all the places above 
$2$ and all the infinite places and let 
$$\rho:G_{F,S}\rightarrow \GL_2(\OO)$$
be a continuous representation of the Galois group of the maximal extension of $F$ unramified outside $S$. Suppose that 
\begin{itemize}
\item[(i)] $\rhobar: G_{F,S}\overset{\rho}{\rightarrow} \GL_2(\OO)\rightarrow \GL_2(k)$ is modular with non-solvable image;
\item[(ii)] For $v\mid 2$, $\rho|_{G_{F_v}}$ is potentially semi-stable with distinct Hodge--Tate weights;
\item[(iii)] $\det \rho$ is totally odd; 
\item[(iv)] for $v\mid 2$, $\rhobar|_{G_{F_v}}\not\cong \bigl ( \begin{smallmatrix} \chi & \ast\\ 0 & \chi\end{smallmatrix}\bigr )$ for any character $\chi: G_{F_v}\rightarrow k^{\times}$.
\end{itemize}
Then $\rho$ is modular. 
\end{thm}
Kisin  \cite{kisinfm}, Emerton \cite{lg} have proved an analogous theorem for $p>2$.  Our proof follows the  strategy of Kisin. 
We patch automorphic forms on definite quaternion algebras and deduce the theorem from a weak form of Breuil--M\'ezard conjecture, which 
we prove for all $p$ under some technical assumptions on the residual representation of $G_{\Qp}$, see Theorems \ref{weak_bm}, \ref{split_bm}, which force us to assume (iv) in the Theorem. 

The Breuil--M\'ezard conjecture is proved by employing a formalism developed in \cite{mybm}, where analogous result is proved 
under the assumption that  $p\ge 5$ and the residual representation has scalar endomorphisms.   We can prove the result for primes $2$ and $3$ by better understanding  the smooth representation theory of $G:=\GL_2(\Qp)$ in characteristic $p$: in the local part of the paper we extend 
the results of \cite{cmf} to the generic blocks, when $p$ is $2$ and $3$, which we will now describe.

Let $\Mod^{\sm}_G(\OO)$ be the category of smooth $G$-representation on $\OO$-torsion modules. We fix a continuous 
character $\psi: \Qp^{\times}\rightarrow \OO^{\times}$ and let $\Mod^{\mathrm{\ladm}}_{G, \psi}(\OO)$  be the full subcategory of $\Mod^{\sm}_G(\OO)$, consisting of representations, on which the centre of $G$ acts by the character $\psi$ and 
which are equal to the union of their admissible subrepresentations. The categories $\Mod^{\sm}_G(\OO)$ and $\Mod^{\mathrm{l.adm}}_{G, \psi}(\OO)$ are abelian, see \cite[Prop.2.2.18]{ord1}. A finitely generated smooth admissible 
representation of $G$ with a central character is of finite length, \cite[Thm.2.3.8]{ord1}. This makes $\Mod^{\mathrm{l.adm}}_{G, \psi}(\OO)$ into a locally finite category. Gabriel \cite{gabriel} has proved that a locally finite category decomposes into a direct product of indecomposable subcategories as follows. 

Let $\Irr_G^{\adm}$ be the set of irreducible  representation in $\Mod^{\mathrm{l.adm}}_{G, \psi}(\OO)$. We define an equivalence relation $\sim$ on the  set of irreducible  representation $\Irr_G^{\adm}$ in $\Mod^{\mathrm{l.adm}}_{G, \psi}(\OO)$: $\pi\sim \tau$, if there exists a sequence in  $\Irr_G^{\adm}$,
$\pi=\pi_1, \pi_2, \ldots, \pi_n=\tau$, such that for each $i$ one of the following holds: 1) $\pi_i\cong \pi_{i+1}$;
2) $\Ext^1_G(\pi_i, \pi_{i+1})\neq 0$; 3) $\Ext^1_G(\pi_{i+1}, \pi_{i})\neq 0$. We have a canonical decomposition:
 \begin{equation}\label{blocksdecompose}
 \Mod^{\mathrm{l.adm}}_{G, \psi}(\OO)\cong \prod_{\BB\in \Irr_G^{\adm}/\sim} \Mod^{\mathrm{l.adm}}_{G,\psi}(\OO)[\BB],
\end{equation}
where $\Mod^{\mathrm{l.adm}}_{G,\psi}(\OO)[\BB]$ is the full subcategory of  $\Mod^{\mathrm{l.adm}}_{G, \psi}(\OO)$ consisting of representations, with all irreducible subquotients in $\BB$. A block is an  equivalence class of $\sim$. 

For a block $\BB$ let $\pi_{\BB}=\oplus_{\pi\in \BB} \pi$, let $\pi_{\BB}\hookrightarrow J_{\BB}$ be an injective envelope of $\pi_{\BB}$ and let $E_{\BB}:=\End_G(J_{\BB})$. 
Then $J_{\BB}$ is an injective generator for $\Mod^{\mathrm{l.adm}}_{G,\psi}(\OO)[\BB]$, $E_{\BB}$ is  a pseudo-compact ring and the functor $\kappa\mapsto \Hom_G(\kappa, J_{\BB})$ induces an anti-equivalence of
categories between   $\Mod^{\mathrm{l.adm}}_{G,\psi}(\OO)[\BB]$ and the category of right pseudo-compact $E_{\BB}$-modules. 
The inverse functor is given by $\md \mapsto (\md\wtimes_{E_{\BB}} J_{\BB}^{\vee})^{\vee}$, where $\vee$ denotes the Pontryagin dual, see
\cite[\S IV.4]{gabriel}. The main result of \cite{cmf} computes the rings $E_{\BB}$ for each block $\BB$ and describes the Galois representation of $G_{\Qp}$ obtained by applying the Colmez's functor to  $J_{\BB}$ under the assumption $p\ge 5$ or $p\ge 3$ depending on the block $\BB$. 

If $\pi\in \Irr_G^{\adm}$ then one may show that after extending scalars $\pi$ is isomorphic to a finite direct sum of absolutely irreducible 
representations of $G$. It has been proved in \cite{blocks} that the blocks containing an absolutely irreducible representation are given by the following:
\begin{itemize} 
\item[(i)] $\BB=\{\pi\}$ with $\pi$ supersingular;
\item[(ii)] $\BB=\{(\Indu{B}{G}{\chi_1\otimes\chi_2\omega^{-1}})_{\sm},
(\Indu{B}{G}{\chi_2\otimes\chi_1\omega^{-1}})_{\sm}\}$ with $\chi_2 \chi_1^{-1}\neq \omega^{\pm 1}, \Eins$;
\item[(iii)] $p>2$ and $\BB=\{ (\Indu{B}{G}{\chi\otimes\chi\omega^{-1}})_{\sm}\}$;
\item[(iv)] $p=2$ and $\BB=\{\Eins,  \Sp\}\otimes\chi\circ \det$;
\item[(v)] $p\ge 5$ and $\BB=\{\Eins, \Sp, (\Indu{B}{G}{\omega\otimes \omega^{-1}})_{\sm}\}\otimes \chi\circ \det$;
\item[(vi)] $p=3$ and $\BB=\{\Eins, \Sp, \omega\circ \det, \Sp\otimes \omega\circ \det\}\otimes \chi\circ \det$;
\end{itemize}
where $\chi, \chi_1, \chi_2: \Qp^{\times}\rightarrow k^{\times}$ are smooth characters, $\omega: 
\Qp^{\times}\rightarrow k^{\times}$ is the character $\omega(x)= x|x| \pmod{\varpi}$ and we view $\chi_1\otimes \chi_2$ as a character 
of the subgroup of upper-triangular matrices $B$ in $G$, which sends $\bigl(\begin{smallmatrix} a & b \\ 0 & d\end{smallmatrix} \bigr)$
to $\chi_1(a) \chi_2(d)$. An absolutely irreducible representation $\pi$ is supersingular if it is not a  subquotient of a principal series 
representation, they have been classified by Breuil in \cite{breuil1}, and $\Sp$ denotes the Steinberg representation.  

To each block above one may attach a semi-simple $2$-dimensional $k$-representation $\rhobarss$ of $G_{\Qp}$: in case (i) 
$\rhobarss$ is absolutely irreducible, and such that Colmez's functor $\VV$, see \S \ref{colmez_functor}, maps $\pi$ to $\rhobarss$, 
in case (ii) $\rhobarss=\chi_1\oplus \chi_2$, in cases (iii) and (iv) $\rhobarss= \chi\oplus \chi$, in cases (v) and (vi) $\rhobarss=\chi\oplus 
\chi\omega$, where we consider characters of $G_{\Qp}$ as characters of $\Qp^{\times}$ via local class field theory, normalized so that 
uniformizers correspond to geometric Frobenii. We note that the determinant of $\rhobarss$ is equal to $\psi \varepsilon$ modulo $\varpi$, where $\varepsilon$ is the $p$-adic cyclotomic character and $\omega$ is its reduction modulo $\varpi$.

\begin{thm}\label{intro_V} If $\BB=\{\pi\}$ with $\pi$ supersingular (so that $\rhobarss$ is irreducible) then $E_{\BB}$ is naturally isomorphic to the quotient  of the universal deformation ring of $\rhobarss$ parameterizing deformations with determinant $\psi \varepsilon$ and 
$\VV(J_{\BB})^{\vee}(\psi \varepsilon)$  is a tautological deformation of $\rhobarss$ to $E_{\BB}$.
\end{thm}

We also obtain an analogous result for block in (ii), see Theorem \ref{block_split}. Let $R^{\ps}$ be the deformation ring 
parameterizing all the $2$-dimensional determinants, in the sense of \cite{che_det}, lifting $(\tr \rhobarss, \det \rhobarss)$, and let 
$R^{\ps, \psi}$ be the quotient of $R^{\ps}$ parameterizing those, which have determinant $\psi \varepsilon$. 

\begin{thm}\label{intro_centre} Assume that the block $\BB$ is given by (i) or (ii) above. Then the centre of the category 
$\Mod^{\ladm}_{G, \psi}(\OO)[\BB]$ is naturally isomorphic to $R^{\ps, \psi}$.
\end{thm}

We view this Theorem as an analogue of Bernstein-centre for this category. Theorems \ref{intro_V}, \ref{intro_centre} are new if $p=2$ and 
if $p=3$ and $\BB=\{\pi\}$ with $\pi$ supersingular. Together with the results of \cite{cmf} this covers all the blocks except for 
those in (iv) and (vi) above.

One also has a decomposition similar to \eqref{blocksdecompose} for the category  $\Ban^{\adm}_{G, \psi}(L)$ of admissible unitary $L$-Banach space representations of $G$, on which the center of $G$ acts by $\psi$, see \S \ref{section_blocks}. An admissible unitary $L$-Banach space representation $\Pi$ lies in $\Ban^{\adm}_{G, \psi}(L)[\BB]$ if and only if all the irreducible subquotients of the reduction modulo 
$\varpi$ of a unit ball in $\Pi$ modulo $\varpi$ lie in $\BB$.  An irreducible $\Pi$ is \textit{ordinary} if it is a subquotient of  a unitary parabolic induction of a unitary character. Otherwise
it is called \textit{non-ordinary}.

\begin{cor}\label{padicLL} Assume that the block $\BB$ is given by (i) or (ii) above.  Colmez's Montreal functor $\Pi\mapsto \cV(\Pi)$ induces a bijection between the isomorphism classes of 
\begin{itemize}
\item absolutely irreducible non-ordinary $\Pi\in \Ban^{\adm}_{G, \psi}(L)[\BB]$;
\item  absolutely irreducible $\tilde{\rho}:\gal\rightarrow \GL_2(L)$, such that $\det \tilde{\rho}=\psi\varepsilon$ and the semi-simplification of the reduction modulo $\varpi$ of a $\gal$-invariant  $\OO$-lattice in $\tilde{\rho}$ is isomorphic to $\rhobarss$.
\end{itemize}
\end{cor}

A stronger result, avoiding the assumption on $\BB$,  is proved in \cite{PCD}. However, our proof of Corollary \ref{padicLL} avoids the hard $p$-adic functional analysis, which is used to construct representations of $\GL_2(\Qp)$ out of a $2$-dimensional representations of $G_{\Qp}$ via the theory of $(\varphi, \Gamma)$-modules by Colmez in \cite{colmez}, which plays the key role in \cite{PCD}.

It might be possible,  given the global part of this paper, and the results of \cite{versal}, which computes  various deformation rings, when $p=2$, to prove Theorem \ref{intro1} by repeating the arguments of Kisin in \cite{kisinfm}. We have not checked this. However, our original goal was to prove Theorems \ref{intro_V} and \ref{intro_centre}; Theorem \ref{intro1} came out as a bonus at the end. 

\subsection{Outline of the paper.} The paper has two largely independent parts: a local one and a global one. We will review each of them 
individually by carefully explaining, which arguments are new. 
\subsubsection{Local part} To fix ideas let assume that $\BB=\{\pi\}$ with $\pi$ supersingular. Let $\rhobar=\VV(\pi)$, let  $R_{\rhobar}$ 
be the universal deformation ring of $\rhobar$ and $R^{\psi}_{\rhobar}$ be the quotient of $R_{\rhobar}$ parameterizing deformations with determinant $\psi\varepsilon$. We follow the strategy  outlined 
in \cite[\S 5.8]{cmf}. We show that $J_{\BB}^{\vee}$ is the universal deformation of $\pi^{\vee}$ and  $E_{\BB}$ is the universal deformation ring by verifying that hypothesis (H0)-(H5), made in \cite[\S3]{cmf}, hold. In \cite[\S 3.3]{cmf} we have developed a criterion to check that the ring $E_{\BB}$ is commutative. To apply this criterion, one needs the ring $R^{\psi}_{\rhobar}$ to be formally smooth and to control the 
image of some $\Ext^1$-group in some $\Ext^2$-group. The first condition does not hold if $p=2$ and if  $p=3$ and $\rhobar \cong \rhobar \otimes\omega$. Even if $p=3$ and $\rhobar\not\cong \rhobar \otimes \omega$, so that the ring is formally smooth, to check the second 
condition is a computational nightmare. In \cite{PCD} we found a different characteristic $0$ argument to get around this. The key input is the 
result of Berger--Breuil \cite{bb}, which says that if a locally algebraic principal series representation admits a $G$-invariant norm, then its completion is irreducible, and $\pi$ occurs in the reduction modulo $\varpi$ with multiplicity one. We deduce from \cite[Cor. 2.22]{PCD} that 
the ring $E_{\BB}$ is commutative. The argument of Kisin \cite{kisin} shows that $\VV(J_{\BB})^{\vee}(\psi \varepsilon)$ is a deformation of $\rhobar$ to $E_{\BB}$ and we have surjections $R_{\rhobar}\twoheadrightarrow E_{\BB}\twoheadrightarrow R^{\psi}_{\rhobar}$.  

To prove Theorem \ref{intro_V} we have to show that the surjection $\varphi: E_{\BB}\twoheadrightarrow R^{\psi}_{\rhobar}$ is an isomorphism. The proof of this claim is new and is carried out in \S\ref{main_local}.
Corollary \ref{padicLL} is then a formal consequence of this isomorphism. If $p\ge 5$ then $R^{\psi}_{\rhobar}$ is formally smooth  and the claim is proved by a calculation on tangent spaces in \cite{cmf}.
This does not hold if $p=2$ or $p=3$ and $\rhobar\cong \rhobar\otimes \omega$. 
  We also note that even if we admit the main result of \cite{PCD} (which we don't), we would only get that $\varphi$ induces a bijection on maximal spectra of the generic fibres of the rings. 
From this one could deduce that the map induces an isomorphism between the maximal reduced quotient of $E_{\BB}$ and $R^{\psi}_{\rhobar}$
and it is not at all clear that $E_{\BB}$ is reduced. However, by using techniques of \cite{mybm} we can show that certain quotients $E_{\BB}/\mathfrak a$ are reduced and identify them with crystabeline deformation rings of $\rhobar$ via $\varphi$. Again the argument uses the results of Berger--Breuil \cite{bb} in a crucial way.  Further, we show that 
the intersection of all such ideals in $E_{\BB}$ is zero, which allows us to conclude. 
A similar  argument using density  appears in \cite[\S 2.4]{PCD}, however we have to work a bit more here, because we fix a central character, see \S \ref{capture}. Theorem \ref{intro_V} implies immediately 
that $\det \cV(\Pi)=\psi\varepsilon$ for all $\Pi\in \Ban^{\adm}_{G, \psi}(L)[\BB]$. This  is proved directly in \cite{PCD} without any restriction on $\BB$, and is the most technical part of that paper.  

Once we have Theorem \ref{intro_V}, the Breuil--M\'ezard conjecture is proved the same way as in \cite{mybm}, see \S \ref{section_bm}. If $\BB$ is the block containing  two generic principal series representations, so that $\rhobarss=\chi_1\oplus \chi_2$, with $\chi_1\chi_2^{-1}\neq \Eins, \omega^{\pm 1}$ then we prove the Breuil--M\'ezard conjecture for both non-split extensions 
$\bigl ( \begin{smallmatrix} \chi_1 & \ast \\ 0 & \chi_2\end{smallmatrix}\bigr)$ and $\bigl ( \begin{smallmatrix} \chi_1 & 0 \\ \ast  & \chi_2\end{smallmatrix}\bigr)$ and deduce the conjecture in the split case from this in a companion paper \cite{versal}, following an 
idea of Hu--Tan \cite{hu_tan}. We formulate and prove the Breuil--M\'ezard conjecture in the language of cycles, as introduced by 
Emerton--Gee in \cite{eg}. All our arguments are local, except that if the inertial type extends to an irreducible representation of the 
Weil group $W_{\Qp}$ of $\Qp$, the description of locally algebraic vectors in the Banach space representations relies on a global input of Emerton \cite[\S 7.4]{lg}. Dospinescu's  \cite{dospinescu} results on locally algebraic vectors in extensions of Banach space representations of $G$ are also crucial in this case. 
 
\subsubsection{Global part} As already explained an analogue of Theorem \ref{intro1} has been proved by Kisin if $p>2$. 
Moreover, if $p=2$ and  $\rho|_{G_{F_v}}$ is semi-stable with Hodge--Tate weights  $(0,1)$ for all $v\mid 2$ then the 
Theorem has been proved by Khare--Wintenberger \cite{KW} and Kisin \cite{kisin_serre2} in their work on Serre's conjecture.  We use their
results as an input in our proof.

The strategy of the proof is the same as in \cite{kisinfm}.  By base change arguments, which are the same as  in \cite{KW}, \cite{kisin_moduli}, \cite{kisin_serre2}, see \S \ref{modularity},  we reduce ourselves to a situation, where the ramification of $\rho$ and $\rhobar$ outside $2$ 
is minimal and $\rhobar$ comes from an automorphic form on a definite quaternion algebra.  
We patch automorphic forms on definite quaternion algebras and deduce the theorem from a weak form of Breuil--M\'ezard conjecture, which is proved in the local part of the paper. The assumption (iv) in Theorem \ref{intro1} comes from the local part of the paper. 

Let us explain some differences with \cite{kisinfm}. If $p>2$ then the patched ring is formally smooth over a completed tensor product 
of local deformation rings. This implies that the patched ring is reduced, equidimensional, $\OO$-flat and its Hilbert--Samuel multiplicity 
is equal to the product of Hilbert--Samuel multiplicities of the local deformation rings.  For $p=2$ we modify the patching argument 
used in \cite{kisinfm} following  Khare--Wintenberger \cite{KW}. This gives us two patched rings and the passage between them and the completed tensor
product of local rings is not as straightforward as before. To overcome this we use an idea, which appears in Errata to \cite{kisinfm} published in \cite{gee_kisin}.  If $\rho_f$ is a Galois representation associated to a Hilbert modular form lifting $\rhobar$ and $v$ is a place of $F$ above $p$ then one knows from  \cite{blasius}, \cite{KM74}, \cite{Sai09}  that the Weil--Deligne representation associated to $\rho|_{G_{F_v}}$ is pure. Kisin shows that this implies that the point on the generic fibre of the potentially semi-stable deformation ring, defined by $\rho_f|_{G_{F_v}}$, cannot lie on the intersection of two irreducible components, and hence is regular. Using this we show that the localization of patched rings at the prime ideal defined by $\rho_f$ is regular, and we are in a position to use Auslander--Buchsbaum theorem, see
Lemma \ref{nail}, Proposition \ref{equivalent_cond_new}.
As explained in \cite{gee_kisin}, this observation enables us to deal with cases, when the patched module is not generically free of rank $1$
over the patched ring, which was the case in the original paper \cite{kisinfm}. In particular, we don't add any Hecke operators at places above $2$ and we don't use \cite[Lem 4.11]{ddt}.  

As a part of his proof Kisin uses the description by Gee \cite{gee_prescribed} of Serre weights for $\rhobar$, which is available only for $p>2$. 
We determine Serre weights for $\rhobar$, when $p=2$ in \S \ref{small_weights} under the assumption (iv) of Theorem \ref{intro1}. As in \cite{gee_prescribed} the main input is a modularity lifting theorem, which in our case is the theorem proved by Khare--Wintenberger \cite{KW} and Kisin \cite{kisin_serre2}. We do this by a characteristic $0$ argument, where Gee argues in characteristic $p$, see \S \ref{small_weights}.

The modularity lifting theorems for $p=2$ proved by Kisin in \cite{kisin_serre2} and more recently by Thorne in  \cite{thorne}, do not require $2$ to split completely in the totally real field $F$, but they put a more restrictive hypothesis on $\rho|_{G_{F_v}}$ for $v\mid 2$. Kisin assumes that $\rho|_{G_{F_v}}$ for all $v\mid 2$ is potentially crystalline with Hodge--Tate weights equal $(0,1)$ for every embedding $F_v\hookrightarrow \overline{\mathbb Q}_2$ and $F_v= \QQ_2$ if  $\rho|_{G_{F_v}}$ is ordinary. Thorne removes this last assumption, but requires instead that $\rhobar|_{G_{F_v}}$ has to be non-trivial for at least one $v\mid \infty$.
We need $2$ to split completely in $F$ in order to apply the results on the $p$-adic Langlands correspondence, which is currently available only for $\GL_2(\Qp)$.   

\subsubsection{Acknowledgements} The local part  was written at about the same time as \cite{PCD}, and it would not have happened  if not for the competitive nature of the correspondence I have had with Gabriel Dospinescu at the time. I thank Ga\"{e}tan Chenevier for the correspondence on $2$-dimensional $2$-adic determinants. 
The global part owes a great deal to the work 
of Khare--Winteberger \cite{KW} and Kisin \cite{kisinfm}, \cite{kisin_serre2}, \cite{kisin_moduli} as will be obvious to the reader. I thank Mark Kisin and Jean-Pierre 
Wintenberger for answering my questions about their  work. I would like to especially thank Toby Gee, for his explanations of 
Taylor--Wiles--Kisin patching method and for pointing out   the right places in the literature to me. I thank Lennart Gehrmann, Jochen Heinloth and Shu Sasaki 
for a number of stimulating discussions. I thank Gabriel Dospinescu, Matthew Emerton, Toby Gee, Jack Thorne for their comments on the earlier draft.  I thank Patrick Allen for pointing out an error in the earlier draft, and for subsequent correspondence, which led to a fix. I thank the referees for their careful reading of the paper.

\section{Local part}

\subsection{Capture}\label{capture}
Let $K$ be a pro-finite group with an open pro-$p$ group. Let $\OO\br{K}$ be the completed group algebra, and 
let $\Mod^{\pro}_K(\OO)$ be the category of compact linear-topological $\OO\br{K}$-modules.  Let $\psi:Z(K)\rightarrow \OO^{\times}$ be a continuous character. 
We let $\Mod^{\pro}_{K, \psi}(\OO)$ be the full subcategory of $\Mod^{\pro}_K(\OO)$, such that $M\in\Mod^{\pro}_K(\OO)$ lies in $\Mod^{\pro}_{K, \psi}(\OO)$ if and only if 
$Z(K)$ acts on $M$ via $\psi^{-1}$. Let $\{V_i\}_{i\in I}$ be a family of  continuous representations 
of $K$ on finite dimensional $L$-vector spaces, and let $M\in \Mod^{\mathrm{pro}}_K(\OO)$. 

\begin{defi}\label{def_capture}
We say 
that $\{V_i\}_{i\in I}$ \textit{captures  $M$} if the smallest quotient
$M\twoheadrightarrow Q$, such that
 $\Hom_{\OO\br{K}}^{\cont}(Q, V_i^*)\cong \Hom_{\OO\br{K}}^{\cont}(M, V_i^*)$ for all $i\in I$ is equal to  $M$. 
\end{defi}

We let $c:=\bigl( \begin{smallmatrix} -1 & 0 \\ 0 & -1\end{smallmatrix}\bigr )$ and note that the centre of $\SL_2(\Zp)$ is equal to 
$\{1, c\}$.

\begin{lem}\label{torsion_free_quotient} If $K=\SL_2(\Zp)$ then $\OO\br{K}/(c-1)$ and $\OO\br{K}/(c+1)$ are $\OO$-torsion free.
\end{lem}
\begin{proof} If $K_n$ is an open normal subgroup of $K$, such that the image of $c$ in $K/K_n$ is non-trivial,
  then $\OO[K/K_n]$ is a free $\OO[Z]$-module, where $Z$ is the centre of $K$. This implies that 
$\OO[K/K_n]/(c\pm 1)$ is a free $\OO$-module and by passing to the limit we obtain the assertion.
\end{proof}

\begin{lem}\label{partition} Let $K=\SL_2(\Zp)$, $Z$ the centre of $K$  and let $\{V_i\}_{i\in I}$ be a family which captures $\OO\br{K}$, such that each $V_i$ has a central character. Let $I^+$, $I^-$ be subsets of $I$ consisting of $i$, such that
 $c$ acts on $V_i$ by $1$ and by $-1$ respectively. Let $\psi: Z \rightarrow L^{\times}$ be a character. If $\psi(c)=1$ then $I^{+}$ captures every projective object in $\Mod^{\pro}_{K, \psi}(\OO)$.
If $\psi(c)=-1$ then $I^{-}$ captures every projective object in $\Mod^{\pro}_{K, \psi}(\OO)$.
\end{lem}
\begin{proof}  If $M\in \Mod^{\pro}_K(\OO)$ is $\OO$-torsion free then $I$ captures $M$ if and only if the image of the evaluation map 
$\bigoplus_{i\in I} V_i \otimes \Hom_K(V_i, \Pi)\rightarrow \Pi$ is dense, where $\Pi=\Hom^{\cont}_{\OO}(M, L)$ is the Banach space representation of $K$ 
with the topology induced by the supremum norm, \cite[Lem.2.10]{PCD}. Let $\Pi=\Hom^{\cont}_{\OO}(\OO\br{K}, L)$ and $\Pi^{\pm}:=\Hom^{\cont}_{\OO}(\OO\br{K}/(c\pm 1),  L)$. Since $\Pi=\Pi^+ \oplus \Pi^-$, and 
 $\{V_i\}$ captures $\OO\br{K}$, we deduce that the image of the evaluation map $\bigoplus_{i\in I} V_i \otimes \Hom_K(V_i, \Pi^{\pm})\rightarrow \Pi^{\pm}$ is dense. If $i\in I^{+}$ then $c$ acts trivially on $V_i$ and so 
$\Hom_K(V_i, \Pi^-)=0$. This implies that the image of $\bigoplus_{i\in I^{+}} V_i \otimes \Hom_K(V_i, \Pi^{+})\rightarrow \Pi^{+}$ is dense.  Using Lemma \ref{torsion_free_quotient}  we deduce that $I^{+}$ captures $\OO\br{K}/(c-1)$. 
A similar argument shows that $I^-$ captures $\OO\br{K}/(c+1)$. Every projective object in $\Mod^{\pro}_{K, \psi}(\OO)$ can be realized as a direct summand of a product of some copies of $\OO\br{K}/(c-\psi(c))$, which implies 
the assertion, see the proof of \cite[Lem. 2.11]{PCD}.
\end{proof}

\begin{lem}\label{what_i_want} Let $K=\SL_2(\Zp)$, $Z$ the centre of $K$, $\psi: Z\rightarrow L^{\times}$ a character, $V$ a continuous representation of $K$ on a finite dimensional $L$-vector space with a central character $\psi_V$. 
If $\psi(c)=\psi_V(c)$ then $\{V\otimes \Sym^{2a} L^2\}_{a\in \NN}$ captures every projective object in  $\Mod^{\pro}_{K, \psi}(\OO)$. If $\psi(c)=-\psi_V(c)$ then $\{V\otimes \Sym^{2a+1} L^2\}_{a\in \NN}$
 captures every projective object in  $\Mod^{\pro}_{K, \psi}(\OO)$.
\end{lem} 
\begin{proof} Proposition 2.12 in \cite{PCD} implies that $\{\Sym^a L^2\}_{a\in \NN}$ captures $\OO\br{K}$. We leave it as exercise for the reader to check that this implies that $\{V\otimes \Sym^a L^2\}_{a\in \NN}$
also captures $\OO\br{K}$.  The assertion follows from Lemma \ref{partition}.
\end{proof} 
 
 \begin{lem}\label{equal_kernels} Let $M\in \Mod^{\pro}_{\GL_2(\Zp), \psi}(\OO)$ and let $V$ be a continuous representation of $K$ on a finite dimensional $L$-vector space with a central character $\psi$. Then 
$$\bigcap_{\phi} \Ker \phi= \bigcap_{\xi, \eta} \Ker \xi,$$
where the first intersection is taken over all $\phi\in \Hom_{\OO\br{\SL_2(\Zp)}}^{\cont}(M, V^*)$ and the second intersection is taken over all  characters  $\eta: \Zp^{\times}\rightarrow L^{\times}$ with $\eta^2=1$
and all $\xi\in  \Hom_{\OO\br{\GL_2(\Zp)}}^{\cont}(M, (V\otimes \eta\circ \det)^*)$.
\end{lem}
\begin{proof} Let $Z$ be the center of $\GL_2(\Zp)$. Determinant induces the isomorphism 
 $\GL_2(\Zp)/ Z \SL_2(\Zp)\cong \Zp^{\times}/(\Zp^{\times})^2$, which is a cyclic group of order $2$ if $p\neq 2$, and a product of cyclic groups of order 
$2$ if $p=2$. Hence, $\Indu{Z\SL_2(\Zp)}{\GL_2(\Zp)}{\Eins}\cong \bigoplus \eta\circ \det$, where the sum is taken over all characters $\eta$ with $\eta^2=1$. The isomorphism
\begin{displaymath}
\begin{split} 
\Hom_{\OO\br{\SL_2(\Zp)}}^{\cont}(M, V^*)&\cong \Hom_{\OO\br{Z\SL_2(\Zp)}}^{\cont}(M, V^*)\\ &\cong  \Hom_{\OO\br{\GL_2(\Zp)}}^{\cont}(M, V^*\otimes \Indu{Z\SL_2(\Zp)}{\GL_2(\Zp)}{\Eins})\\
&\cong \bigoplus_{\eta} \Hom_{\OO\br{\GL_2(\Zp)}}^{\cont}(M, V^*\otimes \eta\circ \det)
\end{split}
\end{displaymath}
implies the assertion.
\end{proof}

\begin{lem}\label{last} Let $M\in \Mod^{\pro}_{\GL_2(\Zp), \psi}(\OO)$ and let $\{V_i\}_{i\in I}$ be a family of  continuous representations of $K$ on  finite dimensional $L$-vector spaces with a central character $\psi$. 
If $\{V_i|_{\SL_2(\Zp)}\}_{i\in I}$ captures $M|_{\SL_2(\Zp)}$ then $\{V_i\otimes\eta\circ\det\}_{i\in I, \eta}$ captures $M$, where $\eta$ runs over all characters $\eta: \Zp^{\times}\rightarrow L^{\times}$ with $\eta^2=1$.
\end{lem}
\begin{proof} The assertion follows from Lemma \ref{equal_kernels} and \cite[Lem.2.7]{PCD}.
\end{proof}

\begin{prop}\label{capture_zeta} Let $K=\GL_2(\Zp)$, $Z$ the centre of $K$ and $\psi: Z\rightarrow L^{\times}$ a continuous character. There is a smooth irreducible representation $\tau$ of $K$, 
which is a  type for a Bernstein component, which contains a principal series representation, but  does not contain a special series representation, such that 
$$\{\tau\otimes \Sym^{a} L^2 \otimes \eta'\circ \det\}_{a\in \NN, \eta'}$$ captures every projective object in $\Mod^{\pro}_{K, \psi}(\OO)$, where for each $a\in \NN$, $\eta'$ runs over all continuous characters $\eta': \Zp^{\times}\rightarrow L^{\times}$ such that $\tau\otimes \Sym^{a} L^2 \otimes \eta'\circ \det$ has central character $\psi$.
\end{prop}

\begin{proof} If $p\neq 2$ (resp.~$p=2$) then $1+p\Zp$ (resp.~$1+4\mathbb{Z}_2$) is a free pro-$p$ group of rank $1$. Using this one may show that there is a smooth, non-trivial  character $\chi: \Zp^{\times}\rightarrow L^{\times}$, and 
a continuous character $\eta_0: \Zp^{\times}\rightarrow L^{\times}$, such that $\psi=\chi \eta_0^2$. Let $e$ be the smallest integer such that $\chi$ is trivial on $1+p^e \Zp$. Let $J=\bigl(\begin{smallmatrix} \Zp^{\times} & \Zp \\ p^e \Zp & \Zp^{\times}\end{smallmatrix}\bigr)$, and let $\chi\otimes \Eins: J \rightarrow L^{\times}$  be the character, which sends $\bigl(\begin{smallmatrix} a & b \\ c & d\end{smallmatrix} \bigr)\mapsto \chi(a)$. 
The representation $\tau:=\Indu{J}{K}{(\chi\otimes \Eins)}$ is irreducible  and is a type. More precisely, for an irreducible smooth $\overline{L}$-representation $\pi$ of $G=\GL_2(\Qp)$,  $\Hom_K(\tau, \pi)\neq 0$ if and only if 
$\pi\cong (\Indu{B}{G}{\psi_1\otimes \psi_2})_{\sm}$, where $B$ is a Borel subgroup and $\psi_1|_{\Zp^{\times}}=\chi$ and $\psi_2|_{\Zp^{\times}}=\Eins$, see \cite[A.2.2]{henniart}. The central character of $\tau$ is equal to $\chi$. We claim that the family
$\{ \tau \otimes \Sym^{2a} L^2 \otimes (\det)^{-a} \otimes \eta \eta_0 \circ \det\}_{a\in \NN, \eta}$, where $\eta$ runs over all the characters with $\eta^2=1$, captures every projective object in $\Mod^{\pro}_{K, \psi}(\OO)$.  If $M\in \Mod^{\pro}_{K, \psi}(\OO)$ is projective then 
$M|_{\SL_2(\Zp)}$ is projective in $\Mod^{\pro}_{\SL_2(\Zp), \psi}(\OO)$, \cite[Prop. 2.1.11]{ord2}. Lemma \ref{what_i_want} implies that the family captures $M|_{\SL_2(\Zp)}$.
Since each representation in the family has central character equal to $\chi \eta_0^2=\psi$, the claim follows from Lemma \ref{last}. Since the family of representations appearing in the  claim 
is a subfamily of the representations appearing in the Proposition, the claim implies the Proposition.
\end{proof}

\subsection{The image of Colmez's Montreal functor}\label{imageCMF} 

Let $G=\GL_2(\Qp)$, $K=\GL_2(\Zp)$, $B$ the subgroup of upper-triangular matrices in $G$, let $T$ be the 
subgroup of diagonal matrices and let $Z$ be the centre of $G$.
We make no assumption on the prime $p$. We fix  a continuous character $\psi: Z\rightarrow \OO^{\times}$. 

Let $\Mod^{\pro}_G(\OO)$ be the category of profinite augmented representations of $G$, \cite[2.1.6]{ord1}. Pontryagin 
duality $\pi\mapsto \pi^{\vee}:=\Hom^{\cont}_{\OO}(\pi, L/\OO)$ induces an anti-equivalence of categories between 
$\Mod^{\sm}_G(\OO)$ and   $\Mod^{\pro}_G(\OO)$, \cite[(2.2.8)]{ord1}. Let $\Mod^{\ladm}_{G}(\OO)$ be the full subcategory 
of $\Mod^{\sm}_G(\OO)$ consisting of locally admissible, \cite[(2.2.17)]{ord1}, representations of $G$ and let 
$\Mod^{\ladm}_{G, \psi}(\OO)$ be the full subcategory of $\Mod^{\ladm}_{G}(\OO)$ consisting of those representations  on which $Z$ acts by the character $\psi$. Let $\dualcat(\OO)$ be the full subcategory of $\Mod^{\pro}_G(\OO)$ anti-equivalent to $\Mod^{\ladm}_{G, \psi}(\OO)$ via the Pontryagin duality. For  $\pi_1, \pi_2$ in  $\Mod^{\ladm}_{G, \psi}(\OO)$ we let $\Ext^i_{G, \psi}(\pi_1, \pi_2)$
be the Yoneda $\Ext$ group computed in   $\Mod^{\ladm}_{G, \psi}(\OO)$.

Let $\pi\in \Mod^{\ladm}_{G, \psi}(\OO)$ be  absolutely irreducible and  either supersingular, \cite{bl}, \cite{breuil1},  or principal series representations $\pi\cong (\Indu{B}{G}{\chi_1\otimes\chi_2\omega^{-1}})_{\sm}$, for  some smooth characters $\chi_1, \chi_2: \Qp^{\times}\rightarrow k^{\times}$, such that $\chi_1\chi_2^{-1} \neq \omega^{\pm 1}, \Eins$. This hypothesis ensures that $\pi':=(\Indu{B}{G}{\chi_2\otimes\chi_1\omega^{-1}})_{\sm}$ is also absolutely irreducible and $\pi\not\cong \pi'$.  Let $P\twoheadrightarrow \pi^{\vee}$ be 
a projective envelope of $\pi^{\vee}$ in $\dualcat(\OO)$ and let $E=\End_{\dualcat(\OO)}(P)$. Then $E$ is naturally a topological ring with a unique maximal ideal and 
residue field $k=\End_{\dualcat(\OO)}(\pi^{\vee})$, see \cite[\S 2]{cmf}.

\begin{prop}\label{residual} If $\pi$ is supersingular then $k\wtimes_E P\cong \pi^{\vee}$. If $\pi$ is principal series then $k\wtimes_E P\cong \kappa^{\vee}$, where $\kappa$ is the 
unique non-split extension $0\rightarrow \pi\rightarrow \kappa\rightarrow \pi'\rightarrow 0$.
\end{prop}
\begin{proof} In both cases, $(k\wtimes_{E} P)^{\vee}$ is the unique representation in $\Mod^{\ladm}_{G, \psi}(\OO)$, which is maximal with respect to the following 
conditions: 1) $\soc_G (k\wtimes_E P)^{\vee}\cong \pi$; 2) $\pi$ occurs in $(k\wtimes_E P)^{\vee}$ with multiplicity one, see \cite[Rem.1.13]{cmf}. For if $\tau$ in 
$\Mod^{\ladm}_{G, \psi}(\OO)$ satisfies both conditions, then 1) and  \cite[Lem. 2.10]{cmf} imply that the natural map $\Hom_{\dualcat(\OO)}(P, \tau^{\vee})\wtimes_E P\rightarrow \tau^{\vee}$ is surjective, and 2) and the exactness of $\Hom_{\dualcat(\OO)}(P, \ast)$ imply that $\Hom_{\dualcat(\OO)}(P, \tau^{\vee})\cong \Hom_{\dualcat(\OO)}(P, \pi^{\vee})\cong k$. Hence, dually we obtain 
an injection $\tau\hookrightarrow (k\wtimes_E P)^{\vee}$.

Let $\pi_1$ be an irreducible representation in $\Mod^{\ladm}_{G, \psi}(\OO)$, such that $\Ext^1_{G, \psi}(\pi_1, \pi)\neq 0$. 
It follows from Corollary 1.2 in \cite{blocks}, that if $\pi$ is supersingular then $\pi_1\cong \pi$ and hence $(k\wtimes_E P)^{\vee}\cong \pi$, if $\pi$ is principal series as above then $\pi_1\cong \pi$ or $\pi_1\cong \pi'$. We will now explain 
how to modify the arguments of \cite[\S8]{cmf} so that they also work for $p=2$. The main point being that Emerton's functor of ordinary parts works for all $p$. Proposition 4.3.15 2) of \cite{ord2} implies that $\Ext^1_{G, \psi}(\pi', \pi)$ is one dimensional. Let $\kappa$ be the unique non-split extension $0\rightarrow \pi\rightarrow \kappa\rightarrow \pi'\rightarrow 0$. We claim that
$\Ext^n_{G, \psi}(\pi', \kappa)=0$  for all $n\ge 0$. The claim for $n=1$ implies that $(k\wtimes_E P)^{\vee}\cong \kappa$. 
It is proved in \cite[Cor. 3.12]{eff} that the $\delta$-functor $H^{\bullet}\Ord_B$, defined in \cite[(3.3.1)]{ord2}, is effaceable in $\Mod^{\ladm}_{G, \psi}(\OO)$. Hence it coincides with the derived functor $\RR^{\bullet}\Ord_B$. An open compact subgroup $N_0$ of the unipotent radical of $B$ is isomorphic to $\Zp$, and hence $H^i(N_0, \ast)$ vanishes for $i\ge 2$. This implies that 
$\RR^i\Ord_B=H^i\Ord_B=0$ for $i\ge 2$.  The proof of 
\cite[Lem. 8.1]{cmf} does not use the assumption $p>2$ and gives that 
\begin{equation}\label{compute_ord} 
\Ord_B \kappa\cong \Ord_B \pi \cong \RR^1\Ord_B \pi' \cong \RR^1\Ord_B \kappa\cong \chi_2\omega^{-1}\otimes\chi_1.
\end{equation}

Our assumption on $\chi_1$ and $\chi_2$ implies that 
$\chi_1\omega^{-1}\otimes \chi_2$ and $\chi_2\omega^{-1}\otimes\chi_1$ are distinct characters of $T$. It follows \cite[Lem 4.3.10]{ord2} that all the $\Ext$-groups between them 
vanish.  Since $\pi'\cong (\Indu{\overline{B}}{G}{\chi_1\omega^{-1}\otimes \chi_2})_{\sm}$, where   $\overline{B}$ is the subgroup of lower-triangular matrices in $G$, all the terms in Emerton's spectral sequence, \cite[(3.7.4)]{ord2},  converging to $\Ext^n_{G, \psi}(\pi', \kappa)$, are zero. Hence, $\Ext^n_{G, \psi}(\pi_2, \kappa)=0$ for all $n\ge 0$. 
 Let us also note that the $5$-term exact sequence associated to the spectral sequence implies that $\Ext^1_{G, \psi}(\pi, \kappa)$ is finite dimensional.
\end{proof}

\begin{prop}\label{hyp_ok} If $\pi$ is supersingular then let $S=Q=\pi^{\vee}$. If $\pi$ is principal series then let $S=\pi^{\vee}$ and $Q=\kappa^{\vee}$. Then $S$ and $Q$ satisfy the hypotheses 
(H0)-(H5) of \cite[\S 3]{cmf}.
\end{prop} 
\begin{proof} If $\pi$ is supersingular then there are no other irreducible representations in the block of $\pi$ and hence the only hypothesis to check is (H4), which is equivalent to 
the finite dimensionality of $\Ext^1_{G, \psi}(\pi, \pi)$. This follows from Proposition 9.1 in \cite{ext}.
If $\pi$ is principal series then the assertion follows from the $\Ext$-group calculations made in the proof 
of Proposition \ref{residual}.
\end{proof}

The proposition enables us to apply the formalism developed in \cite[\S 3]{cmf}. Corollary  3.12 of \cite{cmf} implies:

\begin{prop}\label{PE_flat} The functor $\wtimes_E P$ is an exact functor from the category of pseudocompact right $E$-modules to $\dualcat(\OO)$.
\end{prop}    

If $\md$ is a pseudocompact right  $E$-module then $\Hom_{\dualcat(\OO)}(P, \md\wtimes_E P)\cong \md$ by \cite[Lem.2.9]{cmf}. This implies that the functor is fully faithful, so that 
\begin{equation}\label{full1}
\Hom_{E}^{\cont}(\md_1, \md_2)\cong \Hom_{\dualcat(\OO)}(\md_1\wtimes_E P, \md_2\wtimes_E P).
\end{equation}

\begin{prop}\label{commutative} $E$ is commutative. 
\end{prop}
\begin{proof} Let $\widetilde{\dualcat}(\OO)$ be the the full subcategory of $\Mod^{\pro}_G(\OO)$, which is  anti-equivalent to $\Mod^{\ladm}_{G}(\OO)$ via the Pontryagin duality. Let $\wP$ be a projective envelope
of $\pi^{\vee}$ in $\widetilde{\dualcat}(\OO)$, let $\wE:=\End_{\widetilde{\dualcat}(\OO)}(\wP)$ and let $\mathfrak a$ be the closed two-sided ideal of $\wE$ generated by the elements $z-\psi^{-1}(z)$, for all $z$ in the center of $G$.
We may consider $\dualcat(\OO)$ as a full subcategory of $\widetilde{\dualcat}(\OO)$. Since the center of $G$ acts on $\wP/\mathfrak a \wP$ by $\psi^{-1}$,  $\wP/\mathfrak a \wP$ lies in $\dualcat(\OO)$. The functor 
$\Hom_{\dualcat(\OO)}(\wP/\mathfrak a \wP, \ast)$ is exact, since
 \begin{equation}\label{homs_are_equal}
 \Hom_{\dualcat(\OO)}(\wP/\mathfrak a \wP, M)=\Hom_{\widetilde{\dualcat}(\OO)}(\wP, M)
\end{equation}
 for all $M\in \dualcat(\OO)$, and $\wP$ is projective. 
Hence, $\wP/\mathfrak a \wP$ is projective in $\dualcat(\OO)$. Its $G$-cosocle is isomorphic to $\pi^{\vee}$, since the same is true of $\wP$. Hence, $\wP/\mathfrak a \wP$ is a projective envelope of $\pi^{\vee}$ in $\dualcat(\OO)$. Since projective envelopes are unique up to isomorphism,  $\wP/\mathfrak a \wP$ is isomorphic to $P$. Since $\mathfrak a$ is generated by central elements, any $\phi\in \wE$ maps $\mathfrak a \wP$ to itself. This yields a ring homomorphism 
$\wE\ \rightarrow \End_{\dualcat(\OO)}(\wP/\mathfrak a \wP)\cong E$. Projectivity of $\wP$ and \eqref{homs_are_equal} applied with $M=\wP/\mathfrak a \wP$ implies that the homomorphism is surjective and
induces an isomorphism $\wE/\mathfrak a\cong \End_{\dualcat(\OO)}(\wP/\mathfrak a \wP)$. Since $\wE$ is commutative, \cite[Cor.2.22]{PCD}, we deduce that $E$ is commutative. 
\end{proof}

\begin{prop} $E$ is a complete local noetherian commutative $\OO$-algebra  with residue field $k$.
\end{prop} 
\begin{proof} Proposition \ref{commutative} asserts that $E$ is commutative. Corollary 3.11 of \cite{cmf} implies that the natural topology on $E$, see \cite[\S 2]{cmf}, 
coincides with the topology defined by the maximal ideal $\mm$, which implies that $E$ is complete for the $\mm$-adic topology. It follows from Lemma 3.7, 
Proposition 3.8 (iii) of \cite{cmf} that $\mm/(\mm^2 + (\varpi))$ is a finite dimensional $k$-vector space. Since $E$ is commutative, we deduce that $E$ is noetherian.
\end{proof}

\begin{prop} Let $Q=\pi^{\vee}$ if $\pi$ is supersingular and let $Q=\kappa^{\vee}$ if $\pi$ is principal series. The ring $E$ represents the universal deformation problem of $Q$ in $\dualcat(\OO)$, 
and $P$ is the universal deformation of $Q$.
\end{prop} 
\begin{proof} Since  $E$ is commutative by Proposition \ref{commutative} and the hypotheses (H0)-(H5) of \cite[\S 3]{cmf} are satisfied by Proposition \ref{hyp_ok}, the assertion follows from \cite[Cor.3.27]{cmf}. 
\end{proof}

\subsubsection{Colmez's Montreal functor}\label{colmez_functor}
This subsection is essentially the same as \cite[\S 5.7]{cmf}. Let $G_{\Qp}$ be the absolute Galois group of $\Qp$. 
We will consider $\psi$ as a character of $G_{\Qp}$ via local class field theory, normalized so that the uniformizers correspond to
geometric Frobenii. Let $\varepsilon: G_{\Qp}\rightarrow \OO^{\times}$ be the $p$-adic cyclotomic character. Simirlarly, we will identify $\varepsilon$ with the character of $\Qp^{\times}$, which  maps $x$ to $x|x|$.

In \cite{colmez}, Colmez has defined an exact and covariant functor $\VV$ from the category of smooth, finite length representations of $G$ on $\OO$-torsion modules with 
 a central character to the category of continuous finite length representations of $G_{\Qp}$ on $\OO$-torsion modules. This functor enables us to make the connection 
 between the $\GL_2(\Qp)$ and $G_{\Qp}$ worlds. We modify Colmez's functor to obtain an exact covariant functor $\cV: \dualcat(\OO)\rightarrow \Mod^{\pro}_{\gal}(\OO)$ as follows. 
 Let  $M$ be in $\dualcat(\OO)$, if it is of finite length then $\cV(M):=\VV(M^{\vee})^{\vee}(\varepsilon \psi)$, 
where $\vee$ denotes the Pontryagin dual and $\varepsilon$ is the cyclotomic character.  In general, we may write 
$M\cong \varprojlim M_i$, where the limit is taken over all quotients of finite length in $\dualcat(\OO)$ and  we define
$\cV(M):=\varprojlim \cV(M_i)$. Let  $\pi\in \Mod^{\lfin}_{G, \psi}(k)$ be absolutely irreducible, then $\pi^{\vee}$ is an object of $\dualcat(\OO)$, and 
if $\pi$ is supersingular in the sense of \cite{bl}, then $\cV(\pi^{\vee})\cong \VV(\pi)$ is an absolutely irreducible continuous representations of $G_{\Qp}$ associated to $\pi$ by Breuil  in \cite{breuil1}. 
If $\pi\cong(\Indu{B}{G}{\chi_1\otimes\chi_2 \omega^{-1}})_{\sm}$ then $\cV(\pi^{\vee})\cong \chi_1$. If $\pi\cong \chi\circ \det$ then $\cV(\pi^{\vee})=0$ and if $\pi\cong \Sp\otimes \chi\circ\det$, where 
$\Sp$ is the Steinberg representation, then $\cV(\pi^{\vee})\cong \chi$.  Since $\cV$ is exact we obtain an exact sequence of $\gal$-representations:
\begin{equation}\label{cvkappa}
0\rightarrow \chi_2\rightarrow \cV(\kappa^{\vee}) \rightarrow \chi_1\rightarrow 0
\end{equation}
The sequence is non-split by \cite[VII.4.13 iii)]{colmez}. If $\md$ is a pseudocompact right $E$-module then there exists  a natural isomorphism of $\gal$-representations:
\begin{equation}\label{wtimes_cv}
\cV(\md\wtimes_E P)\cong \md\wtimes_E\cV(P) 
\end{equation}
by \cite[Lem.5.53]{cmf}. It follows from \eqref{wtimes_cv} and Proposition \ref{PE_flat} that $\cV(P)$ is a deformation of $\rho:=\cV(k\wtimes_E P)$ to $E$. 
If $\pi$ is supersingular then $\rho$ is an absolutely irreducible $2$-dimensional representation of $\gal$, and if $\pi$ is principal series then $\rho$ is a non-split extension of distinct characters, see \eqref{cvkappa}. In both cases, $\End_{\gal}(\rho)=k$ and so the universal deformation problem of $\rho$ is represented by a  complete local noetherian $\OO$-algebra $R$. Let $R^{\psi}$ be the quotient of $R$ parameterizing deformation of $\rho$ with determinant equal to $\psi \varepsilon$.

\begin{prop}\label{kisin_plus_e} The functor $\cV$ induces  surjective homomorphisms $R\twoheadrightarrow E$,  $\varphi: E\twoheadrightarrow R^{\psi}$.
\end{prop}
\begin{proof} This is proved in the same way as \cite[Prop.5.56, \S 5.8]{cmf} following \cite{kisin}. For the first surjection it is enough to prove that 
$\cV$ induces an injection: 
$$\Ext^1_{\dualcat(\OO)}(Q, Q)\hookrightarrow \Ext^1_{G_{\Qp}}(\rho, \rho).$$
This follows from \cite[VII.5.2]{colmez}. To prove the second surjection, we observe that $R^{\psi}$ is reduced and $\OO$-torsion free: if $p\ge 5$ then $R^{\psi}$ 
is formally smooth over $\OO$, if $p=3$ then the  assertion follows from  results of B\"ockle \cite{bockle}, if $p=2$ then the assertion follows from the results of Chenevier  \cite[Prop. 4.1]{che2}. Thus it is enough to show that every closed point of 
$\Spec R^{\psi}[1/p]$ is contained in $\Spec E$. This is equivalent to showing that for every deformation 
$\tilde{\rho}$ of $\rho$ with determinant $\psi\varepsilon$ there is a Banach space representation $\Pi$ lifting $Q^{\vee}$ with central  character $\psi$, such that $\cV(\Pi)\cong \tilde{\rho}$. This follows from  \cite[Thm.10.1]{surjectivity}.
\end{proof}

\subsubsection{Banach space representations} Let $\Ban^{\adm}_{G, \psi}(L)$ be the category of admissible unitary $L$-Banach 
space representations, \cite[\S 3]{iw}, on which $Z$ acts by the character $\psi$. If $\Pi\in \Ban^{\adm}_{G, \psi}(L)$ then we let 
\begin{equation}
\cV(\Pi):=\cV(\Theta^d)\otimes_{\OO} L,
\end{equation}
 where $\Theta$ is any open bounded $G$-invariant lattice in $\Pi$. So that $\cV$ is exact and contravariant on $ \Ban^{\adm}_{G, \psi}(L)$.
 
 \begin{remar} One of the  reasons why we use $\cV$ instead of $\mathbf{V}$ is that this allows us to define $\cV(\Pi)$, without making the assumption that the reduction of $\Pi$ modulo $\varpi$ has finite length as a $G$-representation. 
 \end{remar}
 
  If $\md$ is an $E[1/p]$-module of finite length then we let \begin{equation}
 \Pi(\md):=\Hom_{\OO}^{\cont}(\md^0\wtimes_E P, L),
 \end{equation}
  where $\md^0$ is any  $E$-stable $\OO$-lattice in $\md$. Then $\Pi(\md)$ is an admissible unitary $L$-Banach space representation of $G$ by \cite[Lem.2.21]{mybm} with the topology given by the supremum norm. Since   the functor $\wtimes_E P$ is exact by Proposition \ref{PE_flat}, the functor $\md\mapsto \Pi(\md)$ is exact and contravariant. Moreover, it is fully faithful, as
  \begin{equation}\label{full2}
  \begin{split}
  \Hom_G(\Pi(\md_1), \Pi(\md_2))&\cong \Hom_{\dualcat(\OO)}(\md_2^0\wtimes_E P, \md_1^0\wtimes_E P)_L\\
  &\cong \Hom_{E[1/p]}(\md_2, \md_1),
  \end{split}
  \end{equation}
  where the first isomorphism follows from \cite[Thm.2.3]{iw} and the second from \eqref{full1}.
  \begin{lem}\label{Ext_Banach} Let $\md$ be an $E[1/p]$-module of finite length and let $\Pi\in \Ban_{G, \psi}^{\adm}(L)$ be such that $\pi$ does not occur as a subquotient in the reduction of 
  an open bounded $G$-invariant lattice in $\Pi$ modulo $\varpi$. Then $\Ext^1_{G}(\Pi, \Pi(\md))$ computed in $\Ban^{\adm}_{G, \psi}(L)$ is zero. 
  \end{lem}
  \begin{proof} If $\Theta$ is an open bounded $G$-invariant lattice in $\wB\in \Ban^{\adm}_{G, \psi}(L)$ then we define $\md(\wB):=\Hom_{\dualcat(\OO)}(P, \Theta^d)_L$. Proposition 4.17 
  in \cite{cmf} implies that $\md(\wB)$ is a finitely generated $E[1/p]$-module. The functor $\wB\mapsto \md(\wB)$ is exact by \cite[Lem.4.9]{cmf}. The evaluation map
   $\Hom_{\dualcat(\OO)}(P, \Theta^d)\wtimes_E P\rightarrow \Theta^d$ induces a continuous $G$-equivariant map $\wB\rightarrow \Pi(\md(\wB))$. If $\md$ is an $E[1/p]$-module of finite length
   and $\wB\cong \Pi(\md)$ then $\md(\wB)\cong \md$ and the map is $\wB\rightarrow \Pi(\md(\wB))$ is an isomorphism by  \cite[Lem. 4.28]{cmf}. Moreover, $\md(\wB)=0$ if and only if $\pi$ does not 
   occur as a subquotient of $\Theta/(\varpi)$ by \cite[Prop. 2.1 (ii)]{PCD}. Hence, if we have an exact sequence $0\rightarrow \Pi(\md)\rightarrow \wB\rightarrow \Pi\rightarrow 0$ then 
   by applying the functor $\md$ to it, we obtain an isomorphism $\md\cong \md(\Pi(\md))\cong \md(\wB)$ and hence an isomorphism $\Pi(\md)\cong \Pi(\md(\wB))$. The map 
   $\wB\rightarrow \Pi(\md(\wB))$ splits the exact sequence.
  \end{proof}
  The proof of \cite[Lem.4.3]{mybm} shows that we have a natural isomorphism of $\gal$-representations:
  \begin{equation} 
  \cV(\Pi(\md))\cong \md \otimes_{E} \cV(P).
  \end{equation}
 Let us point out a special case of this isomorphism.  If $\nn$ is a maximal ideal of $E[1/p]$ then its residue field $\kappa(\nn)$ is a finite extension of $L$. Let $\OO_{\kappa(\nn)}$ be the ring 
 of integers in $\kappa(\nn)$ and let $\varpi_{\kappa(\nn)}$ be the uniformizer. Then $\Theta: =\Hom^{\cont}_{\OO}(\OO_{\kappa(\nn)} \wtimes_E P, \OO)$  is an open bounded $G$-invariant lattice in 
 $\Pi(\kappa(\nn))$. The evaluation map induces an isomorphism $\Theta^d\cong \OO_{\kappa(\nn)}\wtimes_E P$. Since $E$ is noetherian, $\OO_{\kappa(\nn)}$ is a finitely presented $E$-module
 and thus the usual and completed tensor products coincide. We obtain 
 \begin{equation}\label{cv_Pi_kappa_n}
 \cV(\Theta^d)\cong \OO_{\kappa(\nn)}\otimes_E \cV(P), \quad \cV(\Pi(\kappa(\nn)))\cong \kappa(\nn)\otimes_E \cV(P)
 \end{equation}
Since the residue field of $\OO_{\kappa(\nn)}$ is $k$, we have 
\begin{equation}\label{reduce_Theta}
\Theta/(\varpi_{\kappa(\nn)})\cong \Hom_k^{\cont}(k\wtimes_E P, k)\cong (k\wtimes_E P)^{\vee}.
\end{equation} 

Recall, \cite[\S 4]{cmf},  that $\Pi\in \Ban^{\adm}_{G, \psi}(L)$ is \textit{irreducible} if it does not have a nontrivial closed $G$-invariant subspace. 
It is \textit{absolutely irreducible} if $\Pi\otimes_L L'$ is irreducible in $\Ban^{\adm}_{G, \psi}(L')$ for every finite field extension 
$L'/L$. An irreducible $\Pi$ is \textit{ordinary} if it is a subquotient of  a unitary parabolic induction of a unitary character. Otherwise
it is called \textit{non-ordinary}.

\begin{prop}\label{specialize} If $\nn$ is a maximal ideal of $E[1/p]$ then the  $\kappa(\nn)$-Banach space representation $\Pi(\kappa(\nn))$ is either absolutely irreducible 
non-ordinary or $$\pi\cong (\Indu{B}{G}{\chi_1\otimes\chi_2\omega^{-1}})_{\sm}$$ and  (after possibly replacing $\kappa(\nn)$ by a finite extension) there exists a non-split extension 
\begin{equation}\label{reducible_case}
0\rightarrow \bigl ( \Indu{B}{G}{\delta_1\otimes\delta_2\varepsilon^{-1}})_{\cont}\rightarrow \Pi(\kappa(\nn))\rightarrow \bigl ( \Indu{B}{G}{\delta_2\otimes\delta_1\varepsilon^{-1}})_{\cont}\rightarrow 0,
\end{equation}
where $\delta_1, \delta_2: \Qp^{\times}\rightarrow \kappa(\nn)^{\times}$ are unitary characters congruent to $\chi_1$ and $\chi_2$ respectively, such that $\delta_1\delta_2=\psi\varepsilon$.
\end{prop}
\begin{proof} It follows from \eqref{cv_Pi_kappa_n} that $\dim_{\kappa(\nn)} \cV(\Pi(\kappa(\nn)))=2$. Since $\cV$ applied to a parabolic induction of a unitary character is a one dimensional 
representation of $\gal$, we deduce that if $\Pi(\kappa(\nn))$ is absolutely  irreducible then it cannot be ordinary.

 If $\pi$ is supersingular then \eqref{reduce_Theta} implies that $\Theta/(\varpi_{\kappa(\nn)})\cong \pi$, which is absolutely irreducible. This implies that $\Pi(\kappa(\nn))$ is absolutely irreducible. 
 If $\pi$ is principal series then  $\Theta/(\varpi_{\kappa(\nn)})$ is of length $2$ and both irreducible subquotients are absolutely irreducible. Hence, $\Pi(\kappa(\nn))$ is either irreducible, or 
 of length $2$. Let us assume that $\Pi(\kappa(\nn))$ is not absolutely irreducible. Then after possibly  replacing $\kappa(\nn)$ by a finite extensions we have an exact sequence of admissible 
 $\kappa(\nn)$-Banach space representations $0\rightarrow \Pi_1\rightarrow \Pi(\kappa(\nn))\rightarrow \Pi_2\rightarrow 0$. This sequence is non-split, since otherwise 
 $\cV(\Pi(\kappa(\nn)))$ would be a direct sum of two one dimensional representations, which would 
  contradict \cite[Lem. 4.5 (iii)]{mybm}.  Let $\Theta_1:=\Theta \cap \Pi_1$ and let $\Theta_2$ be the image of $\Theta$ in $\Pi_2$. Since we are dealing with admissible representations, $\Theta_2$ is  a bounded $\OO$-lattice in $\Pi_2$. Lemma 5.5 of \cite{comp} says that we have  exact sequences of $\OO_{\kappa(\nn)}$-modules:  
 \begin{equation}\label{exact_Theta}
 0\rightarrow \Theta_1\rightarrow \Theta \rightarrow \Theta_2\rightarrow 0,
 \end{equation}
 \begin{equation} \label{reduce_mod_p}
 0\rightarrow \Theta_1/(\varpi_{\kappa(\nn)})\rightarrow \Theta/(\varpi_{\kappa(\nn)}) \rightarrow \Theta_2/(\varpi_{\kappa(\nn)})\rightarrow 0.
 \end{equation}
It follows from \eqref{reduce_Theta} that the exact sequence of $G$-representations in \eqref{reduce_mod_p} is the unique non-split extension $0\rightarrow \pi\rightarrow \kappa\rightarrow \pi'\rightarrow 0$. Proposition 4.2.14 of \cite{ord2} applied with $A=\OO_{\kappa(\nn)}/(\varpi_{\kappa(\nn)}^n)$ for all $n\ge 1$ implies that 
$$\Pi_1\cong ( \Indu{B}{G}{\delta_1\otimes\delta_2\varepsilon^{-1}})_{\cont}, \quad \Pi_2\cong( \Indu{B}{G}{\delta'_2\otimes\delta'_1\varepsilon^{-1}})_{\cont},$$
where $\delta_1, \delta_2, \delta_1', \delta_2':\Qp^{\times}\rightarrow \kappa(\nn)^{\times}$ are unitary characters with $\delta_1, \delta_1'$ congruent to $\chi_1$ and $\delta_2$, $\delta_2'$ congruent to $\chi_2$ modulo $\varpi_{\kappa(\nn)}$. We reduce  \eqref{exact_Theta} modulo $\varpi_{\kappa(\nn)}^n$ to obtain an exact sequence to which we apply $\Ord_B$. This gives us an injection
$\Ord_B(\Theta_2/(\varpi_{\kappa(\nn)}^n))\hookrightarrow \RR^1\Ord_B (\Theta_2/(\varpi_{\kappa(\nn)}^n))$. Since both are free $\OO_{\kappa(\nn)}/(\varpi_{\kappa(\nn)}^n)$-modules
of rank $1$, the injection is an isomorphism. This implies that $\delta_1$ is congruent to $\delta_1'$ and $\delta_2$ is congruent $\delta_2'$ modulo $\varpi_{\kappa(\nn)}^n$ for all $n\ge 1$.
Hence, $\delta_1=\delta_1'$ and $\delta_2=\delta_2'$.
\end{proof}

\subsubsection{Main local result}\label{main_local} We will prove that the surjection $\varphi: E \twoheadrightarrow R^{\psi}$ in Proposition \ref{kisin_plus_e} is an isomorphism. The argument combines the first part of the paper with methods of \cite{mybm}. The argument in \cite{cmf} used to prove this statement when $p\ge 5$, uses the fact that the rings $R^{\psi}$ are formally smooth, when $p\ge 5$. This does not hold in general, when $p=2$ or $3$ and  even when the ring is formally smooth and $p=3$ the computations just get too complicated.

Let $V$ be a continuous representation of $K$ with a central character $\psi$ of the form $\tau \otimes \Sym^a L^2 \otimes \eta\circ \det$, 
where $\eta: \Zp^{\times}\rightarrow L^{\times}$ is a continuous character, and $\tau$ is a type for a Bernstein component, which contains a principal series representation, but  does not contain a special series representation.

\begin{prop}\label{tiresome} If $\nn$ is a maximal ideal of $E[1/p]$ then the following hold:
\begin{itemize} 
\item[(i)] $\dim_{\kappa(\nn)}\Hom_K(V, \Pi(\kappa(\nn)))\le 1$;
\item[(ii)]$\dim_{\kappa(\nn)} \Hom_K(V, \Pi(E_{\nn}/\nn^2))\le 2$.
\end{itemize}
Moreover, if $\Hom_K(V, \Pi(\kappa(\nn)))\neq 0$ then $\det \cV(\Pi(\kappa(\nn)))= \psi\varepsilon$.
\end{prop}
\begin{proof} If $\md$ is an $E[1/p]$-module of finite length and $L'$ is a finite extension of $L$, then $\Pi(\md\otimes_L L')\cong \Pi(\md)\otimes_L L'$ and 
$\Hom_K(V, \Pi(\md))\otimes_L L'\cong \Hom_K(V, \Pi(\md)\otimes_L L')$. This implies that  it is enough to prove the assertions after replacing $\kappa(\nn)$ by a finite extension.
In particular, we may assume that $\Pi(\kappa(\nn))$ is either absolutely irreducible 
or a non-split extension as in Proposition \ref{specialize}.  Since $\cV$ 
 is compatible with twisting by characters, to prove the Proposition it is enough to assume that $\eta$ is trivial, so that $V$ is a locally algebraic representation of $K$. 
 
  Since $\tau$ is a type and $\Pi(\kappa(\nn))$ is admissible, $\Hom_K(V, \Pi(\kappa(\nn)))\neq 0$
 if and only if  (after possibly replacing $\kappa(\nn)$ by a finite extension) $\Pi(\kappa(\nn))$ contains a subrepresentation of the form $\Psi \otimes \Sym^a L^2$, 
 where $\Psi$ is an absolutely irreducible  smooth principal series representation in the Bernstein component described by 
 $\tau$, see the proof of \cite[Thm 7.2]{comp}. Let $\Pi$ be the universal unitary completion of $\Psi \otimes \Sym^a L^2$. Then $\Pi$ is absolutely irreducible, 
 by \cite[5.3.4]{bb}, \cite[2.2.1]{be}. 
 
 If $\Pi(\kappa(\nn))$ is absolutely irreducible, we deduce that $\Pi(\kappa(\nn))\cong \Pi$. Since $\Pi$ in \cite{bb} is constructed out of a $(\varphi, \Gamma)$-module 
 of a $2$-dimensional crystabeline representation of $\gal$ with determinant $\psi\varepsilon$, applying $\cV$ undoes this construction to obtain the Galois representation we started with.
 In particular, $\det \cV(\Pi(\kappa(\nn)))=\psi\varepsilon$. Moreover, it follows from \cite[Thm. VI.6.50]{colmez} that the locally algebraic vectors in $\Pi(\kappa(\nn))$ are isomorphic to 
 $\Psi\otimes \Sym^a L^2$, which implies that 
 \begin{equation}
\dim_{\kappa(\nn)} \Hom_K(V, \Pi(\kappa(\nn)))=\dim_{\kappa(\nn)} \Hom_K(V, \Psi\otimes \Sym^a L^2)=1
 \end{equation}
giving part (i). 

If $\Pi(\kappa(\nn))$ is reducible, then using the fact that \eqref{reducible_case} is non-split we deduce that $\Pi$ is the unique irreducible subrepresentation of $\Pi(\kappa(\nn))$. 
It follows from \cite[Lem.12.5]{cmf}\footnote{The assumption $p\ge 5$ in \cite[\S12]{cmf} is only invoked in the proof of  Theorem 12.7 by appealing to Theorem 11.4. All the other arguments
in that section work for all primes $p$.} that the locally algebraic vectors in $\Pi$ are isomorphic to $\Psi\otimes\Sym^a L^2$ and the locally algebraic vectors in $\Pi(\kappa(\nn))/\Pi$ are 
zero. Thus locally algebraic vectors in $\Pi(\kappa(\nn))$ are isomorphic to $\Psi \otimes \Sym^a L^2$ and so part (i) holds. Moreover, applying $\cV$ to \eqref{reducible_case} 
we obtain an exact sequence: $0\rightarrow \delta_2 \rightarrow \cV(\Pi(\kappa(\nn)))\rightarrow \delta_1\rightarrow 0$. Hence, $\det \cV(\Pi(\kappa(\nn)))= \delta_1\delta_2= \psi\varepsilon$.

The exact sequence $0\rightarrow \nn/\nn^2\rightarrow E_{\nn}/\nn^2\rightarrow \kappa(\nn)\rightarrow 0$ of $E[1/p]$-modules gives rise to an exact sequence of admissible Banach space representations of $G$: $$0\rightarrow \Pi(\kappa(\nn))\rightarrow \Pi(E_{\nn}/\nn^2)\rightarrow \Pi(\kappa(\nn))^{\oplus d}\rightarrow 0,$$ where $d=\dim_{\kappa(\nn)} \nn/\nn^2$. We claim that $ 
\Hom_G(\Pi, \Pi(E_{\nn}/\nn^2))$ is one dimensional as $\kappa(\nn)$-vector space. Given the claim we can deduce part (ii) by the same argument as in \cite[Cor.4.21]{mybm}. To show the claim 
let $\Pi':=\Pi(\kappa(\nn))/\Pi$. If $\Pi'$ is zero then the assertion follows from \eqref{full2}. If $\Pi'$ is non-zero then the reduction of the unit ball modulo $\varpi_{\kappa(\nn)}$ is isomorphic 
to $\pi'$.   Since \eqref{reducible_case} is non-split we obtain $\Hom_G(\Pi', \Pi(\kappa(\nn)))=0$, and Lemma \ref{Ext_Banach} implies that $\Ext^1_{G}(\Pi', \Pi(\kappa(\nn)))=0$. 
Hence, $\Hom_G(\Pi(\kappa(\nn)), \Pi(E_{\nn}/\nn^2))\cong \Hom_G(\Pi, \Pi(E_{\nn}/\nn^2))$ and the claim follows from \eqref{full2}.
\end{proof}

Let $\Theta$ be a $K$-invariant 
$\OO$-lattice in $V$ and let $M(\Theta):= \Hom_{\OO\br{K}}^{\cont}(P, \Theta^d)^d$, where $(\ast)^d:=\Hom_{\OO}(\ast, \OO)$. 
It follows from Proposition \ref{residual} that $(k\wtimes_E P)^{\vee}$ is an admissible representation of $G$, dually this implies that $k\wtimes_E P$ is a finitely generated $\OO\br{K}$-module. Hence, \cite[Prop.2.15]{mybm} implies that $M(\Theta)$ is a
finitely generated $E$-module. We will denote by $\mSpec$ the set of maximal ideals of a commutative ring.

\begin{prop}\label{twist_and_shout} Let $\mathfrak a$ be the $E$-annihilator of $M(\Theta)$. Then $E/\mathfrak a$ is reduced and $\OO$-torsion free. Moreover, $\mSpec (E/\mathfrak a)[1/p]$ 
is contained in the image of $\mSpec R^{\psi}[1/p]$ under $\varphi^{\sharp}: \Spec R^{\psi} \rightarrow \Spec E$. 
\end{prop}
\begin{proof} Theorem 5.2 in \cite{mybm} implies that 
there is a $P$-regular  $x\in E$, such that $P/xP$ is a finitely generated $\OO\br{K}$-module, which is projective in $\Mod^{\pro}_{K, \psi}(\OO)$. It follows from \cite[Lem.2.33]{mybm} that 
$M(\Theta)$ is Cohen-Macaulay as a module over $E$ and its Krull dimension is equal to $2$. If $\md$ is an $E[1/p]$-module of finite length then 
\begin{equation}\label{equal_dimensions} 
\dim_L \Hom_K(V, \Pi(\md))=\dim_L  \md \otimes_E M(\Theta)
\end{equation}
by \cite[Prop.2.22]{mybm}. Proposition \ref{tiresome} together with \cite[Prop.2.32]{mybm} imply that $E/\mathfrak a$ is reduced. It is $\OO$-torsion free, since $M(\Theta)$ is 
$\OO$-torsion free. Let $\nn$ be a maximal ideal of $E[1/p]$. Since $E$ is a quotient of $R$, $\nn$ lies in the image of 
$\mSpec R^{\psi}[1/p]$ if and only if $\det \kappa(\nn)\otimes_E \cV(P)=\psi\varepsilon$. Proposition \ref{tiresome}, \eqref{cv_Pi_kappa_n}  and \eqref{equal_dimensions} imply that this holds for
all the maximal ideals of $(E/\mathfrak a)[1/p]$.
\end{proof}

\begin{cor}\label{the_biz} The surjection $\varphi: E\twoheadrightarrow R^{\psi}$, given by Proposition \ref{kisin_plus_e}, induces
an isomorphism $E/\mathfrak a \cong R^{\psi}/\varphi(\mathfrak a)$.
\end{cor}
\begin{proof} Since $(E/\mathfrak a)[1/p]$ and $(R^{\psi}/\varphi(\mathfrak a))[1/p]$ are Jacobson, Proposition \ref{twist_and_shout} implies that $\varphi$ induces an isomorphism between 
$E/\mathfrak a$ and the image of $R^{\psi}$ in the maximal reduced quotient of $(R^{\psi}/\varphi(\mathfrak a))[1/p]$. This implies that the surjection $E/\mathfrak a \twoheadrightarrow  R^{\psi}/\varphi(\mathfrak a)$ is injective, and 
hence an isomorphism.
\end{proof}

\begin{lem}\label{same_ann} The $E$-annihilators of $\Hom^{\cont}_K(P, V^*)$ and $M(\Theta)$ are equal.
\end{lem}
\begin{proof} One inclusion is trivial, the other follows from \cite[(11)]{mybm}, which says that $\Hom^{\cont}_K(P, V^*)$
is naturally isomorphic to $\Hom^{\cont}_{\OO} (M(\Theta), L)$.
\end{proof} 

\begin{thm}\label{main} The functor $\cV$ induces an isomorphism $\varphi: E \overset{\cong}{\rightarrow} R^{\psi}$. Moreover, $\cV(P)$ 
is the universal deformation of $\rho$ with determinant $\psi\varepsilon$.
\end{thm}
\begin{proof} It follows from  Corollary \ref{the_biz} and Lemma \ref{same_ann} that the kernel of $\varphi$ is contained in the 
$E$-annihilator of $\Hom_K^{\cont}(P, V^*)$. It follows from Proposition \ref{capture_zeta} that the intersection of the annihilators as $V$ varies is zero. Hence, $\varphi$ is injective, and hence an isomorphism by Proposition \ref{kisin_plus_e}. The second part is a formal consequence of the first part. \end{proof} 

\subsubsection{Blocks}\label{section_blocks} As explained in the introduction the category $\Mod^{\ladm}_{G, \psi}(\OO)$ decomposes into a product of subcategories:
 \begin{equation}\label{blocksdecompose1}
 \Mod^{\mathrm{l.adm}}_{G, \psi}(\OO)\cong \prod_{\BB\in \Irr_G^{\adm}/\sim} \Mod^{\mathrm{l.adm}}_{G,\psi}(\OO)[\BB],
\end{equation}
where $\Mod^{\mathrm{l.adm}}_{G,\psi}(\OO)[\BB]$ is the full subcategory of  $\Mod^{\mathrm{l.adm}}_{G, \psi}(\OO)$ consisting of representations with all irreducible subquotients in $\BB$. Dually we obtain a decomposition:
\begin{equation} 
\dualcat(\OO)\cong \prod_{\BB\in \Irr_G^{\adm}/\sim} \dualcat(\OO)[\BB],
\end{equation}
where $M\in \dualcat(\OO)$ lies in $\dualcat(\OO)[\BB]$ if and only if $M^{\vee}$ lies in  $\Mod^{\mathrm{l.adm}}_{G,\psi}(\OO)[\BB]$.

For a block $\BB$ let $\pi_{\BB}=\oplus_{\pi\in \BB} \pi$, let $\pi_{\BB}\hookrightarrow J_{\BB}$ be an injective envelope of $\pi_{\BB}$. Then 
$P_{\BB}:= (J_{\BB})^{\vee}$ is a projective envelope of $(\pi_{\BB})^{\vee}$ in $\dualcat(\OO)$. Moreover, $J_{\BB}$ is an injective generator of $\Mod^{\mathrm{l.adm}}_{G,\psi}(\OO)[\BB]$ and $P_{\BB}$ is a projective generator of $\dualcat(\OO)[\BB]$. The ring $E_{\BB}:=\End_{\dualcat(\OO)}(P_{\BB})$ caries a natural topology 
with respect to which it is a pseudo-compact ring, see \cite[Prop. IV.13]{gabriel}. Moreover, the functor 
$$ M\mapsto \Hom_{\dualcat(\OO)}(P_{\BB}, M)$$
induces an equivalence of categories between $\dualcat(\OO)[\BB]$ and the category of right pseudo-compact $E_{\BB}$-modules, see
 Corollaire  1 after \cite[Thm. IV.4]{gabriel}. The inverse functor is given by $\md \mapsto \md\wtimes_{E_{\BB}} P_{\BB}$, as follows from Lemmas
 2.9, 2.10 in \cite{cmf}. Moreover, the centre of the category of $\dualcat(\OO)[\BB]$, which by definition is the ring of the natural transformations 
 of the identity functor, is naturally isomorphic to the centre of the ring $E_{\BB}$, see Corollaire  5 after  \cite[Thm. IV.4]{gabriel}.
 
 Let us prove Theorem \ref{intro_V} stated in the introduction. If $\BB$ is a block containing a supersingular representation $\pi$ then 
 $\BB=\{\pi\}$ and so $\pi_{\BB}=\pi$, $P_{\BB}$ is   a projective envelope of $\pi^{\vee}$ and $E_{\BB}$ coincides with the ring denoted by 
 $E$ in the previous section. Theorem \ref{main} implies that $E_{\BB}$ is naturally isomorphic to  $R^{\psi}_{\rho}$, the quotient of the universal deformation ring of $\rho:=\cV(\pi^{\vee})$ parameterizing deformations with determinant $\psi\varepsilon$. Since this ring is commutative, 
 we deduce that the centre of $\dualcat(\OO)[\BB]$ is naturally isomorphic to $R^{\psi}_{\rho}$. Moreover, $\cV(P_{\BB})$ is the tautological 
 deformation of $\rho$ to $R^{\psi}_{\rho}$, see Theorem \ref{main}.
 
 If $\BB$ contains a generic principal series representation then $\BB=\{\pi_1, \pi_2\}$, where  
 \begin{equation}\label{split_generic_block}
  \pi_1\cong (\Indu{B}{G}{\chi_1\otimes\chi_2\omega^{-1}})_{\sm}, \quad \pi_2\cong (\Indu{B}{G}{\chi_2\otimes\chi_1\omega^{-1}})_{\sm},
  \end{equation}
and  $\chi_1, \chi_2: \Qp^{\times}\rightarrow k^{\times}$ are continuous characters, such that $\chi_1\chi_2^{-1}\neq \Eins, \omega^{\pm 1}$.
 Then $\pi_{\BB}=\pi_1\oplus \pi_2$ and so $P_{\BB}\cong P_1 \oplus P_2$, where $P_1$ is a projective envelope of $\pi_1^{\vee}$ and 
 $P_2$ is a projective envelope of $\pi_2^{\vee}$ in $\dualcat(\OO)$. Thus 
 \begin{equation}\label{zwischen}
 E_{\BB}\cong \End_{\dualcat(\OO)}(P_1 \oplus P_2)\cong \End_{G_{\Qp}}^{\cont}(\cV(P_1)\oplus \cV(P_2)),
 \end{equation}
 where the last isomorphism follows from \cite[Lemma 8.10]{cmf}. The assumption on the characters $\chi_1$, $\chi_2$ implies that 
 if we consider them as representations of $G_{\Qp}$ via the local class field theory $\Ext^1$-groups between them are $1$-dimensional. 
 This means there are unique up to isomorphism non-split extensions 
 $$ \rho_1=\begin{pmatrix} \chi_1 & \ast \\ 0 & \chi_2 \end{pmatrix}, \quad \rho_2=\begin{pmatrix} \chi_1 & 0\\ \ast & \chi_2 \end{pmatrix}.$$
 Let $R_1$ be the universal deformation ring of $\rho_1$, $R_1^{\psi}$ the quotient of $R_1$ parameterizing deformations of $\rho_1$ with determinant
 $\psi \varepsilon$, and let $\rho_1^{\univ}$ be the tautological deformation of $\rho_1$ to $R_1^{\psi}$. We define $R_2$, $R_2^{\psi}$ and $\rho_2^{\univ}$ 
 in the same way with $\rho_2$ instead of $\rho_1$. It follows from Theorem \ref{main} and \eqref{zwischen} that 
 \begin{equation}\label{zwischen2}
 E_{\BB}\cong \End_{G_{\Qp}}^{\cont}(\rho_1^{\univ} \oplus \rho_2^{\univ}).
 \end{equation}
 We have studied the right hand side of \eqref{zwischen2} in \cite[\S B.1]{cmf}  for $p>2$ and in \cite{versal} in general. To describe the result we need 
 to recall the theory of determinants due to Chenevier \cite{che_det}. 
 
 Let $\rho: G_{\Qp}\rightarrow \GL_2(k)$ be a continuous representation. Let $\mathfrak A$ be the category of local artinian augmented $\OO$-algebra
 with residue field $k$. Let $D^{\ps}:\mathfrak A \rightarrow Sets$ be the functor, which maps  $(A,\mm_A)\in \mathfrak A$ to the set of pairs of  functions
 $(t, d): \gal\rightarrow A$, such 
that the following hold: $d:\gal \rightarrow A^{\times}$ is a continuous group homomorphism, congruent to $\det \rho$ modulo $\mm_A$, 
$t: \gal\rightarrow A$ is a continuous function with $t(1)=2$, and, which satisfy for all $g, h\in \gal$:
\begin{itemize} 
\item[(i)] $t(g)\equiv \tr \rho(g)\pmod{\mm_A}$;
\item[(ii)] $t(gh)=t(hg)$;
\item[(iii)] $d(g) t(g^{-1}h)-t(g)t(h)+ t(gh)=0$.
\end{itemize}
The functor $D^{\ps}$ is pro-represented by a complete local noetherian $\OO$-algebra $R^{\ps}$. Let $R^{\ps, \psi}$ be the quotient of $R^{\ps}$ parameterizing those pairs $(t, d)$, where $d=\psi\varepsilon$. Combining \eqref{zwischen2} with \cite[Prop.3.12, 4.3., Cor.4.4]{versal} we obtain the following: 

\begin{thm}\label{block_split} Let $\BB=\{\pi_1, \pi_2\}$ as above and let $\rho=\chi_1\oplus \chi_2$. 
 The centre of $E_{\BB}$, and hence the centre of the category $\dualcat(\OO)[\BB]$, is naturally isomorphic 
 to $R^{\ps,\psi}$. Moreover, $E_{\BB}$  is a free $R^{\ps, \psi}$-module of rank $4$: 
$$ E_{\BB}\cong \begin{pmatrix} R^{\ps, \psi} e_{\chi_1} & R^{\ps, \psi} \tilde{\Phi}_{12} \\ R^{\ps,\psi} \tilde{\Phi}_{21} & R^{\ps, \psi} e_{\chi_2}\end{pmatrix}.$$
The generators satisfy the following relations 
\begin{equation}\label{relation_0}
 e_{\chi_1}^2=e_{\chi_1}, \quad e_{\chi_2}^2=e_{\chi_2}, \quad e_{\chi_1}e_{\chi_2}=e_{\chi_2}e_{\chi_1}=0,
\end{equation}
\begin{equation}\label{relation_1}
 e_{\chi_1} \tilde{\Phi}_{12}=  \tilde{\Phi}_{12}e_{\chi_2}= \tilde{\Phi}_{12}, \quad 
e_{\chi_2} \tilde{\Phi}_{21}=  \tilde{\Phi}_{21}e_{\chi_1}= \tilde{\Phi}_{21},
\end{equation}
\begin{equation}\label{relation_1_bis}
 e_{\chi_2} \tilde{\Phi}_{12}=  \tilde{\Phi}_{12}e_{\chi_1}= 
e_{\chi_1} \tilde{\Phi}_{21}=  \tilde{\Phi}_{21}e_{\chi_2}=\tilde{\Phi}_{12}^2=\tilde{\Phi}_{21}^2=0,
\end{equation}
\begin{equation}\label{relation_2}
 \tilde{\Phi}_{12} \tilde{\Phi}_{21}= c e_{\chi_1}, \quad  \tilde{\Phi}_{21} \tilde{\Phi}_{12}= c e_{\chi_2}.
 \end{equation}
The element $c$ is regular in $R^{\ps, \psi}$ and generates the reducibility ideal.
 \end{thm}
 
 In order to state the result about the centre of $\dualcat(\OO)[\BB]$ in a uniform way as in Theorem \ref{intro_centre}, we note that 
 if $\rho$ is an irreducible representation then mapping   a deformation $\rho_A$ to $(\tr \rho_A, \det \rho_A)$ induces a homomorphism of $\OO$-algebras $R^{\ps}\rightarrow R_{\rho}$, which is an isomorphism by 
 \cite[Thm.2.22, Ex.3.4]{che_det}.
  
 For a block $\BB$ let $\Ban^{\adm}_{G, \psi}(L)[\BB]$ be the full subcategory of $\Ban^{\adm}_{G, \psi}(L)$
 consisting of those $\Pi$, that for some (equivalently any) open bounded $G$-invariant lattice 
 $\Theta$, all the irreducible subquotients of $\Theta\otimes_{\OO} k$ lie in $\BB$. It is shown in \cite[Prop. 5.36]{cmf} that $\Ban^{\adm}_{G, \psi}(L)$ decomposes into a direct sum of subcategories:
 $$ \Ban^{\adm}_{G, \psi}(L)\cong \bigoplus_{\BB\in \Irr_G^{\adm}/\sim} \Ban^{\adm}_{G, \psi}(L)[\BB].$$
 
 \begin{cor}\label{lan_bij} If $\BB=\{\pi\}$ with $\pi$ supersingular then let $\rho=\cV(\pi^{\vee})$. If $\BB=\{\pi_1, \pi_2\}$ 
 with $\pi_1$, $\pi_2$ given by \eqref{split_generic_block} then let $\rho=\cV(\pi_1^{\vee})\oplus \cV(\pi_2^{\vee})=\chi_1\oplus \chi_2$.
  The map $\Pi\mapsto \cV(\Pi)$ induces a bijection between the isomorphism classes of 
\begin{itemize}
\item absolutely irreducible non-ordinary $\Pi\in \Ban^{\adm}_{G, \psi}(L)[\BB]$;
\item absolutely irreducible $\tilde{\rho}:\gal\rightarrow \GL_2(L)$, such that $\det \tilde{\rho}=\psi\varepsilon$ and the semi-simplification of the 
reduction modulo $\varpi$ of a $\gal$-invariant  $\OO$-lattice in $\tilde{\rho}$ is isomorphic to $\rho$.
\end{itemize}
\end{cor}
\begin{proof}  Given Theorems \ref{intro_V}, \ref{block_split} this is proved in the same way as \cite[Thm.11.4]{cmf}.
\end{proof}

If $\Pi\in \Ban^{\adm}_{G, \psi}(L)[\BB]$ and $\Theta$ is an open bounded $G$-invariant lattice in $\Pi$, then 
$\Theta/\varpi^n$ is an object of $\Mod^{\ladm}_{G, \psi}(\OO)[\BB]$ for all $n\ge 1$. Theorem \ref{intro_centre} gives a natural
action of $R^{\ps, \psi}$ on $\Theta/\varpi^n$ for all $n\ge 1$. Passing to the limit and inverting $p$ we obtain a natural 
homomorphism $R^{\ps, \psi}[1/p]\rightarrow \End_G^{\cont}(\Pi)$. 

\begin{cor}\label{trace_special} Let $\BB$ be as in Corollary \ref{lan_bij} and let $\Pi\in \Ban^{\adm}_{G, \psi}(L)[\BB]$ be absolutely irreducible. 
Then $\tr \cV(\Pi)$ is equal to the specialization of universal pseudocharacter $t^{\univ}: G_{\Qp}\rightarrow R^{\ps, \psi}$ at $x: R^{\ps, \psi}\rightarrow \End^{\cont}_G(\Pi)\cong L$. 
\end{cor}
\begin{proof} This is proved in the same way as \cite[Proposition 11.3]{cmf}. To carry out that proof we need to verify 
that $\cV(P_{\BB})$ is annihilated by $g^2-t^{\univ}(g) g + \psi \varepsilon(g)$ for all $g\in G_{\Qp}$. If $\BB$ contains a supersingular
representation this follows from Cayley--Hamilton since  $\cV(P_{\BB})$ is the universal deformation of $\rho$ with determinant $\psi\varepsilon$, and 
$\tr \cV(P_{\BB})=t^{\univ}$ by \cite[Thm.2.22, Ex.3.4]{che_det}. If $\BB$ contains a generic principal series then $\cV(P_{\BB})\cong \rho_1^{\univ}\oplus \rho_2^{\univ}$ and the assertion follows from \cite[Proposition 3.9]{versal}.
\end{proof}

\begin{cor}\label{dim_ext} For any $\Pi$ as in Corollary \ref{lan_bij}, we have $\dim_L \Ext^1_{G, \psi}(\Pi, \Pi)=3$.
\end{cor}
\begin{proof} 

Let $\Ban^{\adm.\mathrm{fl}}_{G, \psi}(L)[\BB]$ be the full subcategory of $\Ban^{\adm}_{G, \psi}(L)[\BB]$ consisting of objects of finite length. It follows from \cite[Thm. 4.36]{cmf} that this category decomposes into a direct sum of subcategories
$$\Ban^{\adm.\mathrm{fl}}_{G, \psi}(L)[\BB]\cong \bigoplus_{\nn \in \mSpec R^{\ps, \psi}[1/p]} \Ban^{\adm.\mathrm{fl}}_{G, \psi}(L)[\BB]_{\nn},$$
where for a maximal ideal $\nn$ of $R^{\ps, \psi}[1/p]$,  $\Ban^{\adm.\mathrm{fl}}_{G, \psi}(L)[\BB]_{\nn}$ consists of those
finite length representations, which are killed by a power of $\nn$. Moreover, the last part of  \cite[Thm. 4.36]{cmf} implies that, the functor 
$\Pi \mapsto \Hom_{\dualcat(\OO)}(P_{\BB}, \Theta^d)[1/p]$, where $\Theta$ is any open bounded $G$-invariant lattice in $\Pi$, induces an
 an anti-equivalence of categories between  $\Ban^{\adm.\mathrm{fl}}_{G, \psi}(L)[\BB]_{\nn}$ and 
the category of modules of finite length over the $\nn$-adic completion of $E_{\BB}[1/p]$, which we denote by
 $\widehat{E}_{\BB, \nn}$.

Let $\tilde{\rho}=\cV(\Pi)$. Corollary \ref{lan_bij} tells us that $\tilde{\rho}$ is an absolutely irreducible representation with determinant $\psi \varepsilon$. Let $\nn$ be the maximal ideal of $R^{\ps, \psi}[1/p]$ corresponding to 
the pair $(\tr \tilde{\rho}, \det \tilde{\rho})$. It follows from  Corollary \ref{trace_special} that $\Pi$ is annihilated by $\nn$ and hence lies in  $\Ban^{\adm.\mathrm{fl}}_{G, \psi}(L)[\BB]_{\nn}$.  Let $A$ be the completion of $R^{\ps, \psi}[1/p]$ at 
$\nn$. In the supersingular case $E_{\BB}=R^{\ps, \psi}= R^{\psi}$, and so  $\widehat{E}_{\BB, \nn}=A$. In the generic principal series case, 
since $\tilde{\rho}$ is absolutely irreducible, the image of the  generator of  the reducible locus in $R^{\ps, \psi}$ in $\kappa(\nn)$ is non-zero. It follows   
from the description of $E_{\BB}$ in Theorem \ref{block_split} that $\widehat{E}_{\BB, \nn}$ is isomorphic to the algebra of $2\times 2$ matrices with entries in $A$. Thus in both cases we get that $\Ban^{\adm.\mathrm{fl}}_{G, \psi}(L)[\BB]_{\nn}$ is anti-equivalent to the category of $A$-modules
of finite length, and $\Pi$ is identified with the residue field $\kappa(\nn)$ of $A$. Hence, 
$$\Ext^1_{G, \psi}(\Pi, \Pi)\cong \Ext^1_A(\kappa(\nn), \kappa(\nn)).$$

Arguing as in \cite[Lemma 2.3.3]{kisin_moduli} we may identify $A$ with the universal deformation ring, which parametrizes pseudocharacters
with determinant $\psi\varepsilon$ with values in local artinian $L$-algebras, which lift $\tr \tilde{\rho}$. Since $\tilde{\rho}$ is absolutely irreducible 
we may further identify this ring with the quotient of the universal deformation ring of $\tilde{\rho}$ to local
artinian $L$-algebras parameterizing deformations with determinant $\psi\varepsilon$. This ring is formally smooth 
over $L$ of dimension $3$, as $H^2(G_{\Qp}, \ad^0(\tilde{\rho}))\cong H^0( G_{\Qp}, \ad^0(\tilde{\rho})(1))=0$ and so the deformation problem of 
$\tilde{\rho}$ is unobstructed.  In particular, $\dim_L \Ext^1_A(\kappa(\nn), \kappa(\nn))=\dim_L  \nn A/ \nn^2 A =3$. 
\end{proof}
\subsection{The Breuil--M\'ezard conjecture}\label{section_bm}

In this section we apply the formalism developed in \cite{mybm} to prove new cases of the Breuil--M\'ezard conjecture, when $p=2$. We place no restriction
on $p$ in this section. 

Let $\rho:\gal\rightarrow \GL_2(k)$ be a continuous representation, which is either absolutely irreducible, in which case we let $\pi$ be a supersingular 
representation of $G$, such that $\VV(\pi)\cong \rho$, or $\rho\cong \bigl ( \begin{smallmatrix} \chi_1 & \ast \\ 0 & \chi_2 \end{smallmatrix} \bigr )$, 
a non-split extension, such that $\chi_1\chi_2^{-1}\neq \Eins, \omega^{\pm 1}$, in which case we let $\pi=(\Indu{B}{G}{\chi_1\otimes \chi_2\omega^{-1}})_{\sm}$. As before we let $R^{\psi}$ be the quotient of the universal deformation ring 
of $\rho$, parameterizing  deformations with determinant $\psi\varepsilon$ and let $\rho^{\univ}$ be the tautological deformation 
of $\rho$ to $R^{\psi}$.

\begin{prop}\label{hyp_N} $P$ satisfies the hypotheses (N0), (N1) and (N2) of \cite[\S4]{mybm}.
\end{prop}
\begin{proof} (N0) says that $k\wtimes_{R^{\psi}} P$ is of finite length and finitely generated over $\OO\br{K}$. This follows from Proposition \ref{residual}.
To verify (N1) we need to show  that $$\Hom_{\SL_2(\Qp)}(\Eins, P^{\vee})=0.$$ The $\SL_2(\Qp)$-invariants in $P^{\vee}$ is stable under the action of $G$. 
Since $P^{\vee}$ is an injective envelope of $\pi$, if the subspace is non-zero then it  must intersect $\pi$ non-trivially. However, $\pi^{\SL_2(\Qp)}=0$, 
which allows us to conclude. (N2) requires $\cV(P)$ and $\rho^{\univ}$ to be isomorphic as $R^{\psi}\br{\gal}$-modules and this is proved in Theorem \ref{main}.
\end{proof}
Recall,  \cite[\S V.A]{mult},  that the group of $d$-dimensional cycles $\mathcal Z_d(A)$ of  a noetherian ring $A$  is a free abelian group generated by $\pp\in \Spec A$ with $\dim A/\pp = d$. If $\sum_{\pp} n_{\pp} \pp$ and $\sum_{\pp} m_{\pp} \pp$ 
are $d$-dimensional cycles then we write $\sum_{\pp} n_{\pp} \pp \le \sum_{\pp} m_{\pp} \pp$, if $n_{\pp} \le m_{\pp}$ for all
$\pp\in \Spec A$ with $\dim A/\pp = d$.

If $M$ is a finitely generated $A$-module of dimension at  most $d$ then $M_{\pp}$ is an $A_{\pp}$-module of finite length, which we denote by $\ell_{A_{\pp}}(M_{\pp})$, for all $\pp$ 
with $\dim A/\pp=d$. We note that $\ell_{A_{\pp}}(M_{\pp})$ is nonzero only for finitely many $\pp$. Thus
$z_d(M):=\sum_{\pp} \ell_{A_{\pp}}(M_{\pp}) \pp$, where the sum is taken over all $\pp\in \Spec A$ such that $\dim A/\pp=d$, is a 
well defined element of $\mathcal Z_d(A)$. 

If $(A, \mm)$ is a local ring then we define a Hilbert--Samuel multiplicity $e(z)$ of a cycle $z=\sum_{\pp} n_{\pp} \pp\in \mathcal Z_d(A)$ to equal $\sum_{\pp} n_{\pp} e(A/\pp)$, where $e(A/\pp)$ is the Hilbert--Samuel multiplicity of the ring $A/\pp$.  If $M$ is a finitely generated $A$-module of dimension $d$ then the Hilbert--Samuel multiplicity of $M$ is equal to the Hilbert--Samuel multiplicity of its cycle $z_d(M)$, see \cite[\S V.2]{mult}.

If $\Theta$ is a continuous representation of $K$ on a free $\OO$-module of finite rank, we let 
$$M(\Theta):= ( \Hom^{\cont}_{\OO\br{K}}(P, \Theta^d))^d,$$
where $(\ast)^d:= \Hom_{\OO}(\ast, \OO)$. If $\lambda$ is a smooth representation of $K$ on an $\OO$-torsion module of finite length then we let 
$$M(\lambda):=(\Hom^{\cont}_{\OO\br{K}}(P, \lambda^{\vee}))^{\vee},$$
where the superscript $\vee$ denotes the Pontryagin dual. 

\begin{prop}\label{reduce_cycle} Let $\Theta$ be a continuous representation of $K$ on a free $\OO$-module of finite rank with central character $\psi$. Then $M(\Theta)$ is a finitely generated $R^{\psi}$-module. If $M(\Theta)$ is non-zero then it is Cohen--Macaulay and has Krull dimension equal to $2$. We have an equality of $1$-dimensional 
cycles: 
\begin{equation}\label{easy_proj}
z_1(M(\Theta)/\varpi)= \sum_{\sigma} m_{\sigma} z_1( M(\sigma)),
\end{equation}
where the sum is taken over all the irreducible smooth $k$-representations of $K$, and $m_{\sigma}$ denotes the multiplicity with which 
$\sigma$ appears as a subquotient of  $\Theta\otimes_{\OO} k$. 

Moreover, $M(\sigma)\neq 0$ if and only if $\Hom_K(\sigma, \pi)\neq 0$, in which 
case the Hilbert--Samuel multiplicity of $z_1(M(\sigma))$ is equal to $1$.
\end{prop}
\begin{proof} We showed in Proposition \ref{hyp_N} that $k\wtimes_{R^{\psi}} P$ is a finitely generated $\OO\br{K}$-module. 
It follows from Corollary 2.5 in \cite{mybm} that $M(\Theta)$ is a finitely generated $R^{\psi}$-module. The restriction of $P$ 
to $K$ is projective in $\Mod^{\pro}_{K, \psi}(\OO)$ by \cite[Cor.5.3]{mybm}. Proposition 2.24 in \cite{mybm} implies that 
\eqref{easy_proj} holds as an equality of $(d-1)$-dimensional cycles, where $d$ is the Krull dimension of $M(\Theta)$.
Theorem 5.2 in \cite{mybm} shows that there is $x$ in the maximal ideal of $R^{\psi}$, such that we have an exact sequence 
$0\rightarrow P\overset{x}{\rightarrow} P \rightarrow P/xP\rightarrow 0$, such that the restriction of $P/xP$ to $K$ is a projective 
envelope of $(\soc_K \pi)^{\vee}$ in $\Mod^{\pro}_{K, \psi}(\OO)$. Lemma 2.33 in \cite{mybm} implies that $M(\Theta)$ is 
a Cohen--Macaulay module of dimension $2$ and $\varpi$, $x$ is a regular sequence of parameters. If $\sigma$ is an irreducible smooth $k$-representation 
of $K$ with central character $\psi$ then the proof of \cite[Lem.2.33]{mybm} yields an exact sequence $0\rightarrow M(\sigma)\overset{x}{\rightarrow} M(\sigma)\rightarrow
(\Hom^{\cont}_{\OO\br{K}}(P/xP, \sigma^{\vee}))^{\vee}\rightarrow 0$. Since $P/xP$ is a projective envelope of $(\soc_K \pi)^{\vee}$ in 
$\Mod^{\pro}_{K, \psi}(\OO)$, we deduce that $\dim_k M(\sigma)/x M(\sigma)$ is equal to $\dim_k \Hom_K(\sigma, \pi)$. If $ \Hom_K(\sigma, \pi)$
is zero then Nakayama's lemma implies that $M(\sigma)=0$. If $\Hom_K(\sigma, \pi)$
is non-zero then it is a one dimensional $k$-vector space, since the $K$-socle of $\pi$ is multiplicity free. The exact sequence
$0\rightarrow M(\sigma)\overset{x}{\rightarrow} M(\sigma)\rightarrow k\rightarrow 0$ implies that $M(\sigma)$ is a cyclic module, 
and if $\mathfrak a$ denotes its annihilator then $R^{\psi}/\mathfrak a\cong k\br{x}$.
\end{proof}

 \begin{remar}\label{socle} If $\rho$ is absolutely irreducible and $\rho|_{I_{\Qp}}\cong (\omega_2^{r+1}\oplus \omega_2^{p(r+1)})\otimes \omega^m$ then 
 $\soc_K \pi\cong (\Sym^r k^2 \oplus \Sym^{p-1-r} k^2\otimes \det^{r})\otimes \det^m$, where $0\le r\le p-1$, $0\le m\le p-2$ and $\omega_2$ is the fundamental 
 character of Serre of niveau $2$, see \cite{breuil1}, \cite{breuil2}. If $\rho\cong \bigl (\begin{smallmatrix} \chi_1& *\\ 0 & \chi_2 \omega^{r+1} \end{smallmatrix}\bigr ) \otimes \omega^m$, 
 where $\chi_1, \chi_2$ are unramified and $\chi_1\neq\chi_2\omega^{r+1}$ then $\pi\cong  (\Indu{B}{G}{\chi_1\otimes \chi_2\omega^r})_{\sm}\otimes\omega^m \circ \det$. Hence, 
 $\soc_K \pi\cong \Sym^r k^2\otimes \det^m$ if $0<r<p-1$ and $\det^m \oplus \Sym^{p-1}k^2\otimes \det^m$, otherwise. In particular, $\soc_K \pi$ is multiplicity free.
  \end{remar}

  If $\nn\in \mSpec R^{\psi}[1/p]$ then the residue field $\kappa(\nn)$ is a finite extension of $L$. Let $\OO_{\kappa(\nn)}$ be the ring of integers in $\kappa(\nn)$. By 
specializing the universal deformation at $\nn$, we obtain a continuous representation $\rho^{\univ}_{\nn}: G_{\Qp}\rightarrow \GL_2(\OO_{\kappa(\nn)})$, which reduces to $\rho$ modulo 
the maximal ideal of $\OO_{\kappa(\nn)}$.  A $p$-adic Hodge type $(\mathbf w, \tau, \psi)$ consists of the following data: $\mathbf w=(a,b)$ is a pair of 
integers with $b>a$, $\tau: I_{\Qp}\rightarrow \GL_2(L)$ is a representation of the inertia subgroup with an open kernel and $\psi:G_{\Qp}\rightarrow \OO^{\times}$ a continuous character, such 
that $\psi\varepsilon\equiv \det \rho \pmod{\varpi}$, $\psi|_{I_{\Qp}}= \varepsilon^{a+b-1}\det\tau$, where $\varepsilon$ is the $p$-adic cyclotomic character.  If $\rho^{\univ}_{\nn}$ is potentially semi-stable then we say that it is of type $(\mathbf w, \tau, \psi)$ if its Hodge--Tate weights are equal to $\mathbf w$, the determinant equal to $\psi$ and the restriction of the Weil--Deligne 
representation, associated to $\rho^{\mathrm{un}}_{\nn}$ by Fontaine in \cite{fontaine}, to $I_{\Qp}$ is isomorphic to $\tau$.

In \cite{henniart}, Henniart has shown the existence of a smooth irreducible  representation $\sigma(\tau)$ (resp. $\sigma^{\mathrm{cr}}(\tau)$) of $K$ on an $L$-vector space, such that if $\pi$ is a smooth absolutely irreducible infinite dimensional  representation of $G$  and  $\LLL(\pi)$ is the Weil-Deligne representation attached to $\pi$ by the classical local Langlands  correspondence then $\Hom_K(\sigma(\tau), \pi)\neq 0$ (resp. $\Hom_K(\sigma^{\mathrm{cr}}(\tau), \pi)\neq 0$) if and only if $\LLL(\pi)|_{I_{\Qp}}\cong \tau$ (resp. $\LLL(\pi)|_{I_{\Qp}}\cong \tau$ and the monodromy operator $N=0$). The representations $\sigma(\tau)$ and $\sigma^{\cris}(\tau)$ are uniquely determined if $p>2$. If $p=2$ there might be different choices, we choose one. 
 
We let $\sigma(\mathbf w, \tau):=\sigma(\tau)\otimes \Sym^{b-a-1} L^2\otimes \det^a$. Then $\sigma(\mathbf w, \tau)$ is a finite dimensional $L$-vector space. Since $K$ is compact 
and the action of $K$ on $\sigma(\mathbf w, \tau)$ is continuous, there is a $K$-invariant $\OO$-lattice $\Theta$ in $\sigma(\mathbf w, \tau)$. Then $\Theta/(\varpi)$ is a smooth 
finite length $k$-representation of $K$, and we let $\overline{\sigma(\mathbf w, \tau)}$ be its semi-simplification. One may show that $\overline{\sigma(\mathbf w, \tau)}$ does not depend on the choice of a lattice.  For each smooth irreducible  $k$-representation $\sigma$ of $K$ we let $m_{\sigma}(\mathbf w, \tau)$ be the multiplicity with which $\sigma$ occurs in $\overline{\sigma(\mathbf w, \tau)}$. We let $\sigma^{\mathrm{cr}}(\mathbf w, \tau):=\sigma^{\mathrm{cr}}(\tau)\otimes  \Sym^{b-a-1} L^2\otimes \det^a$ and let $m^{\mathrm{cr}}_{\sigma}(\mathbf w, \tau)$ be the multiplicity of $\sigma$ in $\overline{\sigma^{\mathrm{cr}}(\mathbf w, \tau)}$. If $p=2$ then one may show that $\overline{\sigma(\mathbf w, \tau)}$ and 
$\overline{\sigma^{\mathrm{cr}}(\mathbf w, \tau)}$ do not depend on the choice of $\sigma(\tau)$, $\sigma^{\cris}(\tau)$.

\begin{prop}\label{support_global} Let $V=\sigma(\wt, \tau)$ (resp. $V=\sigma^{\cris}(\wt, \tau)$) and let $\Theta$ be a $K$-invariant lattice in $V$. Then 
$\nn\in \mSpec R^{\psi}[1/p]$,  lies in the support of $M(\Theta)$ if and only if $\rho^{\univ}_{\nn}$ is potentially semi-stable (resp. potentially 
crystalline) of type $(\wt, \tau, \psi)$. Moreover, for such $\nn$, $\dim_{\kappa(\nn)} M(\Theta)\otimes_{R^{\psi}} \kappa(\nn)=1$.
\end{prop}
\begin{proof}  Proposition 2.22 of \cite{mybm} implies that 
$$ \dim_{\kappa(\nn)}  M(\Theta)\otimes_{R^{\psi}} \kappa(\nn)= \dim_{\kappa(\nn)} \Hom_K(V, \Pi(\kappa(\nn))).$$
Since $V$ is a locally algebraic representation, 
$$ \Hom_K(V, \Pi(\kappa(\nn)))\cong \Hom_K(V, \Pi(\kappa(\nn))^{\alg}),$$
where the subscript $\alg$ denotes the subspace of locally algebraic vectors. This last subspace is non-zero if and only if 
 $\rho^{\univ}_{\nn}$ is potentially semi-stable (resp. potentially crystalline) of type $(\wt, \tau, \psi)$, in which case it is one dimensional, the 
 argument is identical to the proof  \cite[Prop.4.14]{mybm}, except that, because we assume that $\rho$ is generic, we don't have to consider the nasty 
 cases here. 
\end{proof}

\begin{cor}\label{exist_pst} There exists a reduced, $\OO$-torsion free quotient $R^{\psi}(\wt, \tau)$  of 
$R^{\psi}$, such that a map of $\OO$-algebras $x: R^{\psi}\rightarrow L'$ into a finite field  extension of $L$, factors through  
 $R^{\psi}(\wt, \tau)$ if and only if $\rho^{\univ}_x$ is potentially semi-stable 
 of type $(\wt, \tau, \psi)$. 

Moreover, if $\Theta$ is a $K$-invariant $\OO$-lattice in $\sigma(\wt, \tau)$ and $\mathfrak a$ is the $R^{\psi}$-annihilator of $M(\Theta)$ 
then $R^{\psi}(\wt, \tau)= R^{\psi}/\sqrt{\mathfrak a}$. 
 
 Further, the same result holds if we consider potentially crystalline instead of potentially semi-stable representations with $\sigma^{\cris}(\wt, \tau)$
 instead of $\sigma(\wt, \tau)$.
 \end{cor}
 \begin{proof} Since the support of $M(\Theta)$ is closed in $\Spec R^{\psi}$, the assertion follows from Proposition \ref{support_global}.
 \end{proof}
 
 \begin{cor}\label{cycles_reg} Let $\Theta$ be a $K$-invariant lattice in either $\sigma(\wt, \tau)$ or $\sigma^{\cris}(\wt, \tau)$ and let 
 $\mathfrak a$ be the $R^{\psi}$-annihilator of $M(\Theta)$. Then we have equalities of cycles:
 $$ z_2( R^{\psi}/\mathfrak a)= z_2(M(\Theta)), \quad z_1(R^{\psi}/(\mathfrak a, \varpi))= z_1(M(\Theta)/\varpi).$$
\end{cor}
\begin{proof} The last part of Proposition \ref{support_global} implies that $M(\Theta)$ is generically free of rank $1$. This implies the 
first assertion, see \cite[Lem.2.27]{mybm}. The second follows from the first combined with the fact that $\varpi$ is both $R^{\psi}/\mathfrak a$- and 
$M(\Theta)$-regular, see Proposition 2.2.13 in \cite{eg}.
\end{proof}

\begin{prop}\label{reduced} Let $\mathfrak a$ be the $R^{\psi}$-annihilator 
of $M(\Theta)$, where $\Theta$ is a $K$-invariant $\OO$-lattice in $\sigma(\wt, \tau)$, (resp. $\sigma^{\cris}(\wt, \tau)$). Then 
$R^{\psi}/\mathfrak a$ is reduced. In particular, it is equal to $R^{\psi}(\wt, \tau)$, (resp. $R^{\psi, \cris}(\wt, \tau)$).
\end{prop}
\begin{proof} Proposition 2.30 of \cite{mybm} together with the last part of Proposition \ref{support_global} says that it is enough to show that for almost 
all $\nn \in \mSpec R^{\psi}[1/p]$, lying in the support of $M(\Theta)$, $\dim_{\kappa(\nn)} \Hom_K( V, \Pi( R^{\psi}_{\nn}/ \nn^2 R^{\psi}_{\nn}))\le 2$. 
This amounts to checking that the subspace $\mathcal E$ of $\Ext^1_G(\Pi(\kappa(\nn)), \Pi(\kappa(\nn)))$ generated by the extensions of admissible 
unitary $\kappa(\nn)$-Banach spaces $0\rightarrow \Pi(\kappa(\nn))\rightarrow \mathrm{B}\rightarrow \Pi(\kappa(\nn))\rightarrow 0$, such that the induced map between the subspaces of locally algebraic vectors
$\mathrm{B}^{\alg}\rightarrow \Pi(\kappa(\nn))^{\alg}$ is surjective, is at most one dimensional, see the proof of \cite[Cor.4.21]{mybm}.

If $\tau$ does not extend to an irreducible representation of $W_{\Qp}$ then the proof given \cite[Thm. 4.19]{mybm} carries over: the key input into that proof  is that 
the closure of $\Pi(\kappa(\nn))^{\alg}$ in $\Pi(\kappa(\nn))$ is equal to the universal unitary completion of $\Pi(\kappa(\nn))^{\alg}$ and the only case of this fact not covered
by the references given in the proof of \cite[Thm. 4.19]{mybm} is when $p=2$ and $\Pi(\kappa(\nn))^{\alg} \cong (\Indu{B}{G}{\chi\otimes\chi |\centerdot|^{-1}})_{\sm}\otimes W$, 
where $W$ is an algebraic representation of $G$ and $\chi: \Qp^{\times}\rightarrow \kappa(\nn)^{\times}$ is a smooth character. However, in that case it is explained 
in the second paragraph of the proof of \cite[Proposition 6.13]{blocks} how to deduce from  \cite[Prop.4.2]{except} that any $G$-invariant $\OO$-lattice in $\Pi(\kappa(\nn))^{\alg}$ 
is a finitely generated $\OO[G]$-module, which provides the key input also in this case. We note that the assumption $p>2$ in \cite[\S4]{except} is only used to apply 
the results of Berger--Li--Zhu, in particular the proof of \cite[Prop.4.2]{except} works for all $p$.

If $\tau$ extends
to an irreducible representation of $W_{\Qp}$ then the assertion is proved by Dospinescu in \cite{dospinescu}. Although\footnote{I thank G. Dospinescu for pointing this out to me.} 
the main theorem of \cite{dospinescu} is stated under the assumption $p\ge 5$, the argument only uses that 
 if we let $\Pi=\Pi(\kappa(\nn))$ then $\det \cV(\Pi)=\psi\varepsilon$ and $\dim_L \Ext^1_{G, \psi}(\Pi, \Pi)=3$. This 
 is given by Corollaries \ref{lan_bij}, \ref{dim_ext}.
\end{proof}

\begin{thm}\label{weak_bm} There is a  finite set  $\{\mathcal C_{\sigma}\}_{\sigma} \subset \mathcal Z_{1}(R^{\psi}/\varpi)$, indexed by the irreducible smooth 
$k$-representations $\sigma$ of $K$,  such that for all
 $p$-adic Hodge types $(\wt, \tau)$ we have equalities 
$$ z_{1}( R^{\psi}(\wt, \tau)/\varpi)=\sum_{\sigma} m_{\sigma}(\wt, \tau) \mathcal C_{\sigma},$$
 $$ z_{1}( R^{\psi, \mathrm{cr}}(\wt, \tau)/\varpi)=\sum_{\sigma} m_{\sigma}^{\cris}(\wt, \tau) \mathcal C_{ \sigma}.$$
The cycle  $\mathcal C_{\sigma}$  is non-zero if and only if $\Hom_K(\sigma, \pi)\neq 0$, in which case its Hilbert--Samuel multiplicity is equal to $1$.
 \end{thm}
\begin{proof} Let $\mathfrak a$ be the $R^{\psi}$-annihilator of $M(\Theta)$, where $\Theta$ is a $K$-invariant $\OO$-lattice
in $\sigma(\wt, \tau)$. Corollary \ref{exist_pst} and Proposition \ref{reduced} imply that 
$$ z_1(R^{\psi}(\wt, \tau)/\varpi)= z_1(R^{\psi}/(\sqrt{\mathfrak a}, \varpi))=z_1(R^{\psi}/(\mathfrak a, \varpi)).$$  
Corollary \ref{cycles_reg} and Proposition \ref{reduce_cycle} 
imply that 
$$ z_1(R^{\psi}/(\mathfrak a, \varpi))=\sum_{\sigma} m_{\sigma}(\wt, \tau) z_1(M(\sigma)).$$
We let $\mathcal C_{\sigma}=z_1(M(\sigma))$. The proof in the potentially crystalline case is the same. 
\end{proof}

\begin{remar}\label{remark_reduced_ring} One may use a global argument to prove Proposition \ref{reduced}, 
without using the results of \cite{dospinescu}. However, one needs to assume that the local residual representation 
can be realized as a restriction to $G_{\Qp}$ of a global modular representation. 

Let $\mathfrak b$ be the kernel $R^{\psi}/\mathfrak a\twoheadrightarrow R^{\psi}/\sqrt{\mathfrak a}$. 
Since $M(\Theta)$ is Cohen--Macaulay, $R^{\psi}/\mathfrak a$ is equidimensional. Thus if $\mathfrak b$ is non-zero 
then it  is a $2$-dimensional $R^{\psi}$-module, and the cycle $z_1(\mathfrak  b/\varpi)$ is non-zero. 
Since $$z_1(R^{\psi}/(\mathfrak a, \varpi))=z_1(R^{\psi}/(\sqrt{\mathfrak a}, \varpi))+ z_1(\mathfrak b/\varpi),$$
if $R^{\psi}/\mathfrak a$ is 
not reduced then we would conclude that $e(R^{\psi}/(\mathfrak a, \varpi))> e(R^{\psi}(\wt, \tau)/\varpi)$.
Since $e(R^{\psi}/(\mathfrak a, \varpi))=e(M(\Theta)/\varpi)=\sum_{\sigma} m_{\sigma}(\wt, \tau) e(\mathcal C_{\sigma})$, in this case we would obtain 
a contradiction to the Breuil--M\'ezard conjecture. 

If the residual representation can be suitably globalized (when $p=2$ this means that it is of the form $\rhobar|_{G_{\Qp}}$, where $\rhobar$ satisfies the assumptions made in \S \ref{subsection_residual}) then a global argument gives an inequality in the opposite direction, thus allowing us to conclude that $R^{\psi}/\mathfrak a$ is reduced. If $p>2$ then such an  argument is made in  \cite[(2.3)]{kisinfm}. If $p=2$ then the same argument can be made using 
the inequality  \eqref{fortyone} in the proof of Proposition \ref{equivalent_cond_new} and the proof of Corollary \ref{modular}.
\end{remar} 

\begin{remar}\label{remark_framed_vs_non} If $R^{ \square}$ is the 
framed deformation ring of $\rho$  and $R$ is the universal deformation ring of $\rho$ then 
$R^{\square}\cong R\br{x_1,x_2, x_3}$. Thus we have a  map of cycle groups
$$ f: \mathcal Z_i(  R)\rightarrow \mathcal Z_{i+3}(R^{\square}), \quad \pp \mapsto \pp\br{x_1, x_2, x_3},$$
 which preserves Hilbert--Samuel multiplicities. The extra variables only keep track of a choice of basis. 
This implies that  if $R^{\psi, \square}(\wt, \tau)$ is the quotient of $R^{\square}$ parameterizing potentially semi-stable framed deformation of type  $(\wt, \tau, \psi)$ then $R^{\psi, \square}(\wt, \tau)\cong R^{\psi}(\wt, \tau)\br{x_1, x_2, x_3}$, so that 
the cycle of $R^{\psi, \square}(\wt, \tau)/\varpi$ is the image of the cycle of  $R^{\psi}(\wt, \tau)/\varpi$ under $f$. 
Using this one may deduce a version of Theorem \ref{weak_bm} for framed deformation rings.  
\end{remar} 

Let $\rho=\bigl( \begin{smallmatrix} \chi_1 & 0 \\ 0 & \chi_2 \end{smallmatrix} \bigr )$, and let $R^{\square}$ be the universal 
framed deformation ring of $\rho$. Let $R^{\psi, \square}(\wt, \tau)$ (resp. $R^{\psi, \square, \cris}(\wt, \tau)$) be the reduced, $\OO$-torsion free 
quotient of $R^{\square}$ parameterizing potentially semi-stable (resp. potentially crystalline) lifts of $p$-adic Hodge type $(\wt, \tau, \psi)$.

\begin{thm}\label{split_bm} There is a  subset $\{\mathcal C_{1, \sigma}, \mathcal C_{2, \sigma}\}_{\sigma}$ of $\mathcal Z_4(R^{\psi, \square}/\varpi)$ indexed by the irreducible smooth $k$-representations $\sigma$ of $K$, such that for all $p$-adic Hodge types $(\wt, \tau)$ we have inequalities 
$$ z_{4}( R^{\psi, \square}(\wt, \tau)/\varpi)= \sum_{\sigma} m_{\sigma}(\wt, \tau) (\mathcal C_{1, \sigma}+ \mathcal C_{2, \sigma}),$$
 $$ z_{4}( R^{\psi, \square, \mathrm{cr}}(\wt, \tau)/\varpi)=\sum_{\sigma} m_{\sigma}^{\cris}(\wt, \tau)(\mathcal C_{1, \sigma}+ \mathcal C_{2, \sigma}).$$
The cycle $\mathcal C_{1, \sigma}$ is non-zero if and only if 
 $\Hom_K(\sigma, (\Indu{B}{G}{\chi_1\otimes \chi_2\omega^{-1}})_{\sm})\neq 0$, the cycle $\mathcal C_{2, \sigma}$ is non-zero if and only if 
 $\Hom_K(\sigma, (\Indu{B}{G}{\chi_2\otimes \chi_1\omega^{-1}})_{\sm})\neq 0$, in which case the Hilbert--Samuel multiplicity is equal to $1$.
\end{thm}
\begin{proof} Given Theorem \ref{weak_bm}, the assertion follows from \cite[Thm.7.3, Rem.7.4]{versal}.
\end{proof}

The following Corollary will be used in the global part of the paper. 

\begin{cor}\label{pot_crys} Assume that $p=2$, $\psi$ is unramified  and $\rho$ is either absolutely irreducible or $\rho^{\mathrm{ss}}=\chi_1\oplus \chi_2$, with 
$\chi_1\neq \chi_2$. If $\wt=(0,1)$ and $\tau=\Eins\oplus \Eins$ then $$R^{\psi, \square, \cris}(\wt, \tau)= R^{\psi, \square}(\wt, \tau).$$

In other words, every semi-stable lift of $\rho$ with Hodge--Tate weights $(0,1)$ is crystalline. 
\end{cor}
\begin{proof} It is enough to prove the statement, when  $\rho$ is non-split. Since if the assertion was false in the split case then by
choosing a different lattice in the semi-stable, non-crystalline lift we would also obtain a contradiction in the non-split case. Since
framed deformation rings are formally smooth over the non-framed ones, it is enough to prove that 
$R^{\psi}(\wt, \tau)=R^{\psi, \cris}(\wt, \tau)$. By the same argument as in Remark \ref{remark_reduced_ring} we see that 
it is enough to show that $R^{\psi}(\wt, \tau)/\varpi$ and $R^{\psi, \cris}(\wt, \tau)/\varpi$ have the same cycles (and even the 
equality of Hilbert--Samuel multiplicities will suffice). Since $p=2$ there are only $2$ irreducible smooth $k$-representations of $K$: 
$\Eins$ and $\st$. The $K$-socle of $\pi$ in all the cases is isomorphic to $\Eins \oplus \st$, $\sigma(\wt, \tau)/\varpi \cong \st$ and 
$\sigma^{\cris}(\wt, \tau)/\varpi \cong \Eins$. The assertion follows from Theorem \ref{weak_bm}. 
\end{proof}

 \begin{remar}\label{need_bound} Assume that $p=2$, let $\xi: G_{\Qp}\rightarrow \OO^{\times}$ be unramified and congruent to $\psi$ modulo $\varpi$, and let 
$(\wt, \tau)$ be arbitrary. It follows from Theorem \ref{weak_bm}, Remark \ref{remark_framed_vs_non}, Theorem \ref{split_bm} and the 
proof of Corollary \ref{pot_crys} that 
$$z_4(R^{\psi, \square}(\wt, \tau)/\varpi)=(m_{\Eins}(\wt, \tau)+ m_{\st}(\wt, \tau)) z_4(R^{\xi, \square}((0,1), \Eins \oplus \Eins)/\varpi),$$
where the cycles live in $\mathcal Z_4(R^{\square})$. This equality implies the equality of the respective Hilbert--Samuel multiplicities. 
\end{remar}

\section{Global part}
In the global part of the paper $p=2$, so that $L$ is a finite extension of $\Q_2$ with the ring integers $\OO$ and residue field $k$.
\subsection{Quaternionic modular forms}\label{quaternions}
We follow very closely \cite[(3.1)]{kisin_serre2}. Let $F$ be a totally real field in which $2$ splits completely. Let $D$ be a quaternion algebra with centre $F$, ramified at all the infinite places of $F$ and a set of finite places $\Sigma$, 
which does not contain any primes dividing $2$. We fix a maximal order $\OO_D$ of $D$, and for each finite place $v\not\in \Sigma$ an isomorphism $(\OO_D)_v\cong M_2(\OO_{F_v})$. For each finite place $v$ of $F$ we will denote by 
$\mathbf N(v)$ the order of the residue field at $v$, and by $\varpi_v\in F_v$ a uniformizer. 

Denote by $\AfF\subset \mathbb A_F$ the finite adeles, and let $U=\prod_v U_v$ be a compact open subgroup contained in $\prod_v (\OO_D)_v^{\times}$. We assume that if $v\in \Sigma$ then $U_v=(\OO_D)_v^{\times}$ and if $v\mid 2$ then $U_v=\GL_2(\OO_{F_v})=\GL_2(\ZZ_2)$. Let $A$ be a topological $\ZZ_2$-algebra. For each $v\mid  2$, we fix a continuous representation $\sigma_v : U_v \rightarrow \Aut(W_{\sigma_v})$ on a finite free $A$-module. Write $W_{\sigma} = \otimes_{v\mid 2,A} W_{\sigma_v}$ and denote by $\sigma:\prod_{v\mid 2} U_v
\rightarrow \Aut(W_{\sigma})$ the corresponding representation. We regard $\sigma$ as being a representation of $U$ by letting $U_v$ act trivially if $v\nmid  2$.
Finally, assume there exists a continuous character $\psi: (\AfF )^{\times}/F^{\times}\rightarrow A^{\times}$ such that for any place $v$ of $F$, the action of $U_v \cap \OO_{F_v}^{\times}$ on
$\sigma$  is given by multiplication by $\psi$.  We extend the action of $U$ on $W_{\sigma}$ to $U(\AfF )^{\times}$, by letting $(\AfF )^{\times}$ act via $\psi$.

Let $S_{\sigma, \psi}(U,A)$ denote the set of continuous functions 
$$f : D^{\times} \backslash( D \otimes_F  \AfF )^{\times} \rightarrow W_{\sigma}$$
such that for $g\in(D\otimes_F \AfF)^{\times}$ we have $f(gu)=\sigma (u)^{-1} f(g)$ for $u\in U$  and $f(gz)=
\psi^{-1}(z)f(g)$ for $z\in (\AfF )^{\times}$.  If we write $(D \otimes_F \AfF )^{\times} = \coprod_{i\in I} D^{\times}t_i U(\AfF )^{\times}$ for some $t_i \in  (D \otimes_F \AfF )^{\times}$ and some finite index set $I$, then we have an isomorphism  of $A$-modules: 
\begin{equation}\label{auto_form} 
 S_{\sigma, \psi}(U, A)\overset{\cong}{\longrightarrow} \bigoplus_{i\in I} W_{\sigma}^{(U (\AfF)^{\times}\cap t_i^{-1}D^{\times} t_i)/F^{\times}}, \quad f\mapsto (f(t_i))_{i\in I}.
 \end{equation}
 
  \begin{lem}\label{order} Let $U_{\max}=\prod_v \OO_{D_v}^{\times}$, where the product is taken over all finite places of $F$, and let $t\in (D\otimes_F \AfF)^{\times}$. Then the 
 group $(U_{\max}(\AfF)^{\times} \cap t D^{\times}t^{-1})/F^{\times}$ is finite and there is an integer $N$, independent of $t$, such that its order divides $N$.
 \end{lem}
 \begin{proof} This is explained in \S 7.2 of \cite{KW}, see also \cite[Lemma 1.1]{taylor_deg2}.
 \end{proof} 
  
  I thank Mark Kisin for explaining the proof of the following Lemma to me. 

 \begin{lem}\label{find_prime} Let $v_1$ be  a finite place of $F$, such that $D$ splits at $v_1$, and $v_1$ does not divide $2N$, where $N$ is the integer defined in Lemma \ref{order}. 
 Let $U=\prod_v U_v$ be a subgroup of $(D\otimes_F \AfF)^{\times}$, such  that $U_v= \OO_{D_v}^{\times}$ if $v\neq v_1$ and $U_{v_1}$ is the subgroup of upper triangular, 
 unipotent matrices modulo $\varpi_{v_1}$. Then 
 \begin{equation}\label{no_isotropy}
 (U(\AfF)^{\times} \cap t D^{\times}t^{-1})/F^{\times}=1, \quad \forall t\in (D\otimes_F \AfF)^{\times}.
 \end{equation}
 \end{lem} 
 \begin{proof}  Let $u\in (U(\AfF)^{\times} \cap t D^{\times}t^{-1})$, such that $u\not\in F^{\times}$. Then the $F$-subalgebra  $F[u]$ of $t Dt^{-1}$ is a 
 quadratic field extension of $F$. Let $u'$ be the conjugate of $u$ over $F$. Then $u' = \Nm(u)/u$, where $\Nm$ is  the reduced norm.
Consider $w = u/u' = u^2/\Nm(u)$. Write $u = hg$ with $h \in U$ and $g\in (\AfF)^{\times}$. Then $\Nm(g)=g^2$ and so 
$w = u/u' = h^2/\Nm(h)$. Thus $w$ is in $U$ and also in $tD^{\times}t^{-1}$. 

Since $(U(\AfF)^{\times} \cap t D^{\times}t^{-1})/F^{\times}$ is a subgroup of $(U_{\max}(\AfF)^{\times} \cap t D^{\times}t^{-1})/F^{\times}$, 
$u^N$ is in $F^{\times}$ and hence $w^N=u^N/ (u')^N=1$. Let $\ell$ be the prime dividing $\mathbf N(v_1)$. Since $U_{v_1}$ is a pro-$\ell$ group and $\ell$ 
does not divide $N$, the image of $w$ under the projection $U\rightarrow U_{v_1}$ is equal to $1$. Since for every $v$ the map $D\rightarrow D_v$
is injective, we conclude that $w=1$, which implies that $u\in F$. 
\end{proof} 

If \eqref{no_isotropy} holds then it follows from \eqref{auto_form} that $\sigma\mapsto S_{\sigma, \psi}(U, A)$
defines an exact functor from the category of continuous representations of $U$ on finitely generated $A$-modules, on which $U_v$ for
$v\nmid 2$ acts trivially and
$U\cap (\AfF)^{\times}$ acts by $\psi$,  to the category of finitely generated $A$-modules. 

Let $S$ be  a finite set of places of $F$ containing $\Sigma$, all the places above 
$2$, all the infinite places and all the places $v$, where $U_v$ is not maximal. Let $\TT^{\univ}_{S, A}=A[ T_v, S_v]_{v\not\in S}$ 
be a commutative polynomial ring in the indicated formal variables. We let $( D \otimes_F  \AfF )^{\times}$ act on the space of continuous $W_{\sigma}$-valued function on $( D \otimes_F  \AfF )^{\times}$ by right translations, $(h f)(g):= f(gh)$. Then 
$S_{\sigma, \psi}(U, A)$ becomes a $\TT^{\univ}_{S, A}$-module with $S_v$ acting via the double coset 
$U_v \bigl ( \begin{smallmatrix} \varpi_v & 0 \\ 0 & \varpi_v\end{smallmatrix}\bigr) U_v$ and $T_v$ acting via the double coset 
$U_v \bigl ( \begin{smallmatrix} \varpi_v & 0 \\ 0 & 1\end{smallmatrix}\bigr) U_v$. We write $\TT_{\sigma, \psi}(U, A)$ or 
$\TT_{\sigma, \psi}(U)$ for the image of $\TT^{\univ}_{S, A}$ in the endomorphisms of $S_{\sigma, \psi}(U, A)$.

\subsection{Residual Galois representation}\label{subsection_residual} Keeping the notation of the previous section we fix an algebraic closure $\overline{F}$ of $F$ and let $G_{F, S}$ be the Galois group of the maximal extension of $F$ in $\overline{F}$ which is unramified outside $S$.  We view $\psi$ as a character of $G_{F, S}$ via global class field theory, normalized so that uniformizers are mapped to geometric Frobenii.  Let $\chi_{\cyc}: G_{F, S}\rightarrow \OO^{\times}$ be the global $2$-adic cyclotomic character. We note that $\chi_{\cyc}$ is trivial modulo $\varpi$. For each place $v$ of $F$, including the infinite places,  we fix an embedding $\overline{F}\hookrightarrow \overline{F}_v$. This induces a continuous homomorphism  of  Galois groups $G_{F_v}:=\Gal(\overline{F}_v/F_v)\rightarrow G_{F, S}$. We fix a continuous representation:
$$ \rhobar: G_{F, S}\rightarrow \GL_2(k)$$
and assume that the following conditions hold:
\begin{itemize}
  \item the image of $\rhobar$ is non-solvable;
 \item $\rhobar$ is unramified at all finite places $v\nmid  2$;
 \item if $v\in S$ is a finite place, $v\not \in \Sigma$, $v\nmid  2$ then the eigenvalues of $\rhobar(\Frob_v)$ are distinct;
 \item if $v\in \Sigma$ then the   eigenvalues of $\rhobar(\Frob_v)$ are equal;
 \item  $\det \rhobar \equiv \psi \chi_\cyc \pmod{\varpi}$;
 \item if $v\in S$ is a finite place, $v\not \in \Sigma$, $v\nmid  2$ then 
 $$U_v=\{g\in \GL_2(\OO_{F_v}): g \equiv \bigl (\begin{smallmatrix} 1 & \ast\\ 0 & 1 \end{smallmatrix}\bigr) \pmod{\varpi_v}\}$$
and at least one such $v$ does not divide $2N$, so that the condition of Lemma \ref{find_prime} is satisfied.  
\end{itemize}

\subsubsection{Local deformation rings.}
We fix a basis of the underlying vector space $V_k$ of $\rhobar$. For each $v\in S$ let $R^\square_v$ be the framed deformation ring of $\rhobar|_{G_{F_v}}$ and let $R^{\psi, \square}_v$ be the quotient of $R^\square_v$ parameterizing lifts with determinant $\psi \chi_{\cyc}$. We will now introduce some quotients of $R^{\psi, \square}_v$. 

For $v\mid 2$ let $\tau_v$ be a $2$-dimensional representation of the inertia group $I_v$ with an open kernel, and let $\wt_v=(a_v, b_v)$ be a pair of integers with $b_v> a_v$. Let $\sigma(\tau_v)$ be any absolutely irreducible representation of $U_v=\GL_2(\ZZ_2)$ with the property that for all irreducible  infinite dimensional smooth representations $\pi$ of $\GL_2(\QQ_2)$, $\Hom_{U_v}(\sigma(\tau_v), \pi)\neq 0$
 if and only if the restriction to $I_v$ of the Weil--Deligne representation $\LLL(\pi)$ associated to $\pi$ via the local 
Langlands correspondence is isomorphic to $\tau$.  The existence of such $\sigma(\tau_v)$ is shown in \cite{henniart}, where it is also shown that 
if $\Hom_{U_v}(\sigma(\tau_v), \pi)\neq 0$ then it is one dimensional. We choose a $U_v$-invariant $\OO$-lattice $\sigma(\tau_v)^0$ in $\sigma(\tau_v)$
and let 
\begin{equation}\label{define_sigma_v}
\sigma_v:= \sigma(\tau_v)^0 \otimes_{\OO} \Sym^{b_v-a_v-1} \OO^2 \otimes_{\OO} \dt^{a_v}.
\end{equation} 
We let $R^{\psi, \square}_v(\sigma_v)$ be the reduced, $\OO$-flat quotient of $R^{\psi, \square}_v$ parameterizing potentially semi-stable lifts with Hodge--Tate weights 
$\wt_v$ and inertial type $\tau_v$. This ring is denoted by $R^{\psi, \square}(\wt, \tau)$ in the local part of the paper. 

We similarly define $\sigma^{\cris}(\tau_v)$ by additionally requiring that $\Hom_{U_v}(\sigma^{\cris}(\tau_v), \pi)\neq 0$ if and only if 
the monodromy operator $N$ in $\LLL(\pi)$ is zero and $\LLL(\pi)|_{I_v}\cong \tau_v$. In this case we let 
\begin{equation}\label{define_sigma_cris_v}
\sigma_v:= \sigma^{\cris}(\tau_v)^0 \otimes_{\OO} \Sym^{b_v-a_v-1} \OO^2 \otimes_{\OO} \dt^{a_v}.
\end{equation}
We let $R^{\psi, \square}_v(\sigma_v)$ be the quotient of $R^{\psi, \square}_v$ parameterizing potentially crystalline lifts with Hodge--Tate weights 
$\wt_v$ and inertial type $\tau_v$. This ring is denoted by $R^{\psi,\square,  \cris}(\wt, \tau)$ in the local part of the paper. 

It follows either from the local part of the paper, or from \cite{kisin_pst}, where a more general result is proved, that 
if $R^{\psi, \square}_v(\sigma_v)$ is non-zero then it is equidimensional of Krull dimension $5$. 
Since the residue field of $\ZZ_2$ has $2$ elements, $\sigma(\tau_v)$ need not be unique, see \cite[A.2.6, A.2.7]{henniart}, however the semi--simplification of $\sigma(\tau_v)^0\otimes_\OO k$
is the same in all cases. 
 
 If $v$ is infinite then $R^{\psi, \square}_v$ is a domain of Krull dimension $3$ and $R^{\psi, \square}_v[1/2]$ is regular, 
 \cite[Prop.2.5.6]{kisin_serre2}, \cite[Prop. 3.1]{KW}.
 
 If $v$ is finite, $\rhobar$ is unramified at $v$ and $\rhobar(\Frob_v)$ has distinct Frobenius eigenvalues then  $R^{\psi, \square}_v$
 has Krull dimension $4$ and $R^{\psi, \square}_v[1/2]$ is regular. This follows from \cite[Prop 2.5.4]{kisin_serre2}, where 
 it is shown there that the dimension is $4$ and the irreducible components are regular. Since we assume that  the eigenvalues of  $\rhobar(\Frob_v)$ are distinct, $\rhobar$ cannot have a lift of the form $\gamma \oplus \gamma\chi_{\cyc}$. 
 It follows from the proof of \cite[Prop 2.5.4]{kisin_serre2} that different irreducible components of $R^{\psi, \square}_v[1/2]$ do not intersect. 
 
If $v$ is finite, $\psi$ and $\rhobar$ are  unramified at $v$ and $\rhobar(\Frob_v)$ has equal eigenvalues then 
for an unramified character $\gamma: G_{F_v}\rightarrow \OO^{\times}$ such that $\gamma^2=\psi|_{G_{F_v}}$ we let 
$R^{\psi, \square}_v(\gamma)$ be a reduced $\OO$-torsion free quotient of $R^{\psi, \square}_v$, such that 
if $L'/L$ is a finite extension then a map $x: R^{\psi, \square}_v\rightarrow L'$ factors through $R^{\psi, \square}_v(\gamma)$ if and only if $V_x$ is isomorphic to $\bigl ( \begin{smallmatrix} \gamma\chi_{\cyc} & \ast\\ 0 & \gamma\end{smallmatrix} \bigr )$.   It follows from \cite[Prop. 2.5.2]{kisin_serre2} via  \cite[2.6.6]{kisin_moduli}, \cite[Thm. 3.1]{KW} that $R^{\psi, \square}_v(\gamma)$ is a domain of Krull 
dimension $4$ and $R^{\psi, \square}_v(\gamma)[1/2]$ is regular. If $L$ is large enough then there are precisely two such characters, 
which we denote by $\gamma_1$ and $\gamma_2$. We let 
$\bar{R}^{\psi, \square}_v$ be the image of 
$$ R^{\psi, \square}_v \rightarrow R^{\psi, \square}_v(\gamma_1)[1/2]\times R^{\psi, \square}_{v}(\gamma_2)[1/2].$$
Then $\bar{R}^{\psi, \square}_v$ is a reduced, $\OO$-flat quotient of $R^{\psi, \square}_v$ such that if $L'/L$ is a finite extension then a map $x: R^{\psi, \square}_v\rightarrow L'$ factors through $\bar{R}^{\psi, \square}_v$ if and only if $V_x$ is isomorphic to $\bigl ( \begin{smallmatrix} \gamma\chi_{\cyc} & \ast\\ 0 & \gamma\end{smallmatrix} \bigr )$ for an unramified character $\gamma$. Moreover, 
$$\bar{R}^{\psi, \square}_v[1/2]\cong R^{\psi, \square}_v(\gamma_1)[1/2] \times R^{\psi, \square}_v(\gamma_2)[1/2].$$
Thus $\bar{R}^{\psi, \square}_v[1/2]$ is regular, equidimensional and the Krull dimension of $\bar{R}^{\psi, \square}_v$ is $4$.
 
 We let $R^{\square}_S= \wtimes_{v\in S} R^{\square}_v$, $R^{\psi, \square}_S= \wtimes_{v\in S} R^{\psi, \square}_v$, $\sigma:=\wtimes_{v\mid 2} \sigma_v$ and 
$$ R^{\psi, \square}_S(\sigma):= \wtimes_{v\mid 2} R^{\psi, \square}_v(\sigma_v)\wtimes_{v\in \Sigma} \bar{R}^{\psi, \square}_v
 \wtimes_{v\in S\setminus \Sigma, v\nmid 2\infty} R^{\psi, \square}_v \wtimes_{v\mid \infty} R^{\psi, \square}_v.$$
 It follows from above $R^{\psi, \square}_S(\sigma)$ is equidimensional of Krull dimension equal to 
 \begin{equation}\label{dim_S_loc}
 1+4\sum_{v\mid 2}1 + 3|\Sigma|+ 3( |S|- |\Sigma|- \sum_{v\mid 2} 1- \sum_{v\mid \infty} 1)+ 2 \sum_{v\mid \infty} 1=1+3|S|.
 \end{equation}
 \subsubsection{Global deformation rings.} Since $\rhobar$ is assumed to have non-solvable image, $\rhobar$ is absolutely irreducible.   
 We define $R^{\psi}_{F, S}$ to be the quotient of the universal deformation ring of $\rhobar$, parameterizing deformations with determinant 
 $\psi\chi_{\cyc}$. If $Q$ is a finite set of places of $F$ disjoint from $S$ then we let $S_{Q}=S\cup Q$ and define $R^{\psi}_{F, S_Q}$ in the same way by 
 viewing $\rhobar$ as a representation of $G_{F, S_Q}$.
 
 Denote by $R^{\psi, \square}_{F, S_Q}$ the complete local $\OO$-algebra representing the functor which assigns to an artinian, augmented $\OO$-algebra 
$A$  the set of isomorphism classes of tuples $\{V_A,\beta_w\}_{w\in S}$, where $V_A$ is a deformation of $\rhobar$ to $A$ with determinant $\psi\chi_{\cyc}$ and $\beta_w$ is a lift of a chosen basis of $V_k$ to a basis of $V_A$. The map $\{V_A,\beta_w\}_{w\in S}\mapsto \{V_A, \beta_v\}$
 induces a homomorphism of $\OO$-algebras $R^{\psi, \square}_v\rightarrow R^{\psi, \square}_{F, S_Q}$ for every $v\in S$ and 
 hence a homomorphism of $\OO$-algebras $R^{\psi, \square}_S\rightarrow R^{\psi, \square}_{F, S_Q}$.

 \subsection{Patching}\label{patching}
 
 For each $n\ge 1$ let $Q_n$ be the set of places of $F$ disjoint from $S$, as in \cite[3.2.2]{kisin_serre2} via \cite[Prop. 5.10]{KW}. 
 We let $Q_0=\emptyset$, so that $S_{Q_n}=S$ for $n=0$. 
 Let $U_{Q_n}=\prod_{v} (U_{Q_n})_v$ be a compact  open subgroup of $(D\otimes_F \AfF)^{\times}$, such that $(U_{Q_n})_v=U_v$ for
 $v\not\in Q_n$ and $(U_{Q_n})_v$ is defined as in \cite[3.1.6]{kisin_serre2} for $v\in Q_n$.
  
 Let $\mm$ be a maximal ideal of $\TT^{\univ}_{S, \OO}$, such that the residue field is $k$, $T_v$ is mapped to
 $\tr \rhobar(\Frob_v)$ and $S_v$ is mapped to the image of $\psi(\Frob_v)$ in $k$ for all $v\not\in S$. We define $\mm_{Q_n}$ in 
 $\TT^{\univ}_{S_{Q_n}, \OO}$ in the same manner. Let $\sigma=\otimes_{v\mid 2} \sigma_v$, 
 where each $\sigma_v$ is given by either \eqref{define_sigma_v} or \eqref{define_sigma_cris_v}. 
 We assume that $S_{\sigma, \psi}(U, \OO)_{\mm}\neq 0$. Then  for all $n\ge 0$ there is a surjective homomorphism of $\OO$-algebras 
 $R^{\psi}_{F, S_{Q_n}}\rightarrow \TT_{\sigma, \psi}(U_{Q_n})_{\mm_{Q_n}}$, such that for all $v\not\in S_{Q_n}$ the trace of $\Frob_v$ of 
the tautological $R^{\psi}_{F, S_{Q_n}}$-representation of $G_{F, S_{Q_n}}$ is mapped to $T_v$.  Set 
 $$M_n(\sigma)= R^{\psi, \square}_{F, S_{Q_n}}\otimes_{R^{\psi}_{F, S_{Q_n}}} S_{\sigma, \psi}(U_{Q_n}, \OO)_{\mm_{Q_n},}$$
 with the convention that if $n=0$ then $Q_n=\emptyset$, $S_{Q_n}=S$, $\mm_{Q_n}=\mm$,  so that 
 $$M_0(\sigma)= R^{\psi, \square}_{F, S}\otimes_{R^{\psi}_{F, S}} S_{\sigma, \psi}(U, \OO)_{\mm}.$$
 It follows from the local-global compatibility of Jacquet--Langlands and Langlands correspondences that the action of $R^{\psi, \square}_{F, S_{Q_n}}$ on $M_n(\sigma)$ factors through 
 the quotient $$R^{\psi, \square}_{F, S_{Q_n}}(\sigma):= R^{\psi, \square}_S(\sigma)\otimes_{R^{\psi, \square}_S} R^{\psi, \square}_{F, S_{Q_n}}.$$
  Let  $h=\dim_k H^1(G_{F,S}, \ad \rhobar)-2=|Q_n|$.  Let $\mathfrak a_{\infty}$ denote the ideal of $\OO\br{y_1, \ldots, y_h}$ generated by 
 $(y_1, \ldots, y_h)$. 
Since $ R^{\psi, \square}_{F, S_{Q_n}}$ is formally smooth over $R^{\psi}_{F, S_{Q_n}}$ of relative dimension $j= 4|S|-1$ we may choose an identification 
$$R^{\psi, \square}_{F, S_{Q_n}}=R^{\psi}_{F, S_{Q_n}}\br{y_{h+1}, \ldots, y_{h+j}}$$ and regard $M_n(\sigma)$ as an $\OO\br{y_1, \ldots, y_{h+j}}$-module. This allows us to consider $R^{\psi}_{F, S_{Q_n}}$ as an $R^{\psi, \square}_S$-algebra via the map 
$R^{\psi, \square}_S\rightarrow R^{\psi, \square}_{F, S_{Q_n}}/(y_{h+1}, \ldots, y_{h+j})= R^{\psi}_{F, S_{Q_n}}$. 
We let 
$$R^{\psi}_{F, S_{Q_n}}(\sigma):= R^{\psi, \square}_S(\sigma)\otimes_{R^{\psi, \square}_S} R^{\psi}_{F, S_{Q_n}}.$$
Let $g=2|Q_n|+1$ and $t=2-|S|+|Q_n|$ and let $\widehat{\mathbb G}_m$ be the completion of the $\OO$-group $\mathbb G_m$ along the identity section. 
The patching argument as in \cite[Prop. 9.3]{KW} shows that there exist $\OO\br{y_1, \ldots, y_{h+j}}$-algebras $R'_{\infty}(\sigma)$, $R_{\infty}(\sigma)$
and an $R_{\infty}(\sigma)$-module $M_{\infty}(\sigma)$ with the following properties. 
\begin{itemize}
\item[(P1)] There are surjections of $\OO$-algebras $$R^{\psi, \square}_S(\sigma)\br{x_1, \ldots, x_g}\twoheadrightarrow R'_{\infty}(\sigma)\twoheadrightarrow R_{\infty}(\sigma).$$
\item[(P2)] There is an isomorphism of $R^{\psi, \square}_S(\sigma)$-algebras 
$$R_{\infty}(\sigma)/\mathfrak a_{\infty} R_{\infty}(\sigma)\overset{\cong}{\rightarrow} R^{\psi, \square}_{F, S}(\sigma)$$
and an isomorphism of $R^{\psi, \square}_{F, S}(\sigma)$-modules 
$$M_{\infty}(\sigma)/\mathfrak a_{\infty} M_{\infty}(\sigma)\overset{\cong}{\rightarrow} M_0(\sigma).$$
\item[(P3)] $M_{\infty}(\sigma)$ is finite flat over $\OO\br{y_1, \ldots, y_{h+j}}$.
\item[(P4)] $\Spf R'_{\infty}(\sigma)$ is equipped with a free action of $(\widehat{\mathbb G}_m)^t$, and a $(\widehat{\mathbb G}_m)^t$-equivariant morphism 
$\delta: \Spf R'_{\infty}(\sigma)\rightarrow (\widehat{\mathbb G}_m)^t$, where $(\widehat{\mathbb G}_m)^t$ acts on itself by the square of the identity map.
\item[(P5)] We have $\delta^{-1}(1)= \Spf R_{\infty}(\sigma)\subset \Spf R'_{\infty}(\sigma)$, and the induced action of $(\widehat{\mathbb G}_m[2])^t$ on
$ \Spf R_{\infty}(\sigma)$ lifts to $M_{\infty}(\sigma)$.
\end{itemize}

If $A$ is a local noetherian ring of dimension $d$ and $M$ is a finitely generated $A$-module, we denote by $e(M,A)$  the coefficient of $x^d$ in the 
Hilbert--Samuel polynomial of $M$ with respect to the maximal ideal of $A$, multiplied by $d!$. In particular, $e(M, A)=0$ if $\dim M< \dim A$. If 
$M=A$ we abbreviate  $e(M, A)$ to $e(A)$.

It follows from \cite[Prop. 2.5]{KW} that there is a complete local noetherian $\OO$-algebra $(R^{\mathrm{inv}}_{\infty}(\sigma), \mm^{\inv}_{\sigma})$ with residue field $k$, such that 
$\Spf R^{\mathrm{inv}}_{\infty}(\sigma)= \Spf R'_{\infty}(\sigma)/(\widehat{\mathbb G}_m)^t$. Moreover, 
\begin{equation}\label{r_inv_1}
R'_{\infty}(\sigma)= R^{\inv}_{\infty}(\sigma)\wtimes_{\OO} \OO\br{ \ZZ_2^t}\cong R^{\inv}_{\infty}(\sigma)\br{z_1, \ldots, z_t}.
\end{equation}
This implies  that 
\begin{equation}\label{dimension_fix}
\dim R'_{\infty}(\sigma)= \dim R^{\inv}_{\infty}(\sigma)+t, \quad e(R'_{\infty}(\sigma)/\varpi)=e(R^{\inv}_{\infty}(\sigma)/\varpi).
\end{equation}

\begin{lem}\label{torsor_fix} There are $a_1, \ldots, a_t\in \mm_{\sigma}^{\inv}$, such that 
\begin{equation}\label{eqn_torsor}
 R_{\infty}(\sigma)\cong \frac{R^{\inv}_{\infty}(\sigma)\br{z_1}}{((1+z_1)^2- (1+a_1))}\otimes_{R^{\inv}_{\infty}(\sigma)}\ldots \otimes_{R^{\inv}_{\infty}(\sigma)} \frac{R^{\inv}_{\infty}(\sigma)\br{z_t}}{((1+z_t)^2- (1+a_t))}.
 \end{equation}
In particular, $R_{\infty}(\sigma)$ is a free $R^{\inv}_{\infty}(\sigma)$-module of rank $2^t$. 
\end{lem}
\begin{proof} It follows from \cite[Lem. 9.4]{KW} that $\Spf R_{\infty}(\sigma)$ is a $\gm$-torsor 
over $\Spf R_{\infty}^{\inv}(\sigma)$. The assertion follows from  \cite[Exp. VIII, Prop.4.1] {sga3}. 
\end{proof} 

\begin{lem}\label{transitive_action} Let $\pp\in \Spec R_{\infty}^{\inv}(\sigma)$.
The group $\gm(\OO)$ acts transitively on the set of prime ideals of $R_{\infty}(\sigma)$ lying above $\pp$.
\end{lem}
\begin{proof} Let us write $X$ for $\Spf R_{\infty}(\sigma)$ and $G$ for $\gm$. The action of $G$ on $X$ induces 
an action of $(\pm 1)^t=G(\OO)\hookrightarrow G(R_{\infty}(\sigma))$ on $X(R_{\infty}(\sigma))$. If $g\in G(\OO)$ we let 
$\phi_g\in X(R_{\infty}(\sigma))$ be the image of $(g, \id_{R_{\infty}(\sigma)})$. The map $g\mapsto \phi_g$
induces a homomorphism of groups $G(\OO)\rightarrow \Aut(R_{\infty}(\sigma))$. Explicitly, if $g=(\epsilon_1, \ldots, 
\epsilon_t)$, where $\epsilon_i$ is either $1$ or $-1$ then $\phi_g$ is $R^{\inv}_{\infty}(\sigma)$-linear and 
maps $1+z_i$ to $\epsilon_i (1+z_i)$ for $1\le i\le t$. It follows from \eqref{eqn_torsor} that  $G(\OO)$ acts transitively on 
the set of maximal ideals of $\kappa(\pp)\otimes_{R_{\infty}^{\inv}(\sigma)} R_{\infty}(\sigma)$.   
\end{proof} 	

\begin{lem}\label{support} The support of $M_{\infty}(\sigma)$ in $\Spec R_{\infty}(\sigma)$ is a union of irreducible components. The Krull dimension
of $\Spec R_{\infty}(\sigma)$ is equal to $h+j+1$.
\end{lem}
\begin{proof} It follows from part (P3) above that the support of $M_{\infty}(\sigma)$ is equidimensional of dimension $h+j+1$. To prove the assertion it is enough to 
show that the dimension of $R_{\infty}(\sigma)$ is less or equal to  $h+j+1$. Using Lemma \ref{torsor_fix}, \eqref{dimension_fix}, (P1)  and \eqref{dim_S_loc} we deduce that 
$\dim R_{\infty}(\sigma)\le \dim R^{\psi, \square}_S(\sigma)+ g-t= 3|S|+1 +g-t=h+j+1.$
\end{proof}

\begin{lem}\label{go_local} $e(R'_{\infty}(\sigma)/\varpi)\le e(R^{\psi, \square}_S(\sigma)/\varpi)$.
\end{lem}
\begin{proof} It follows from \eqref{r_inv_1}, Lemmas  \ref{torsor_fix}, \ref{support} that 
$$\dim R'_{\infty}(\sigma)=\dim R_{\infty}(\sigma)+t= t+h+j+1=3|S|+1+g,$$ 
which is also the dimension of 
$R^{\psi, \square}_S(\sigma)\br{x_1,\ldots, x_g}$ by  \eqref{dim_S_loc}.
The surjection in (P1) above implies that 
$$ e(R'_{\infty}(\sigma)/\varpi)\le e(R^{\psi, \square}_S(\sigma)\br{x_1, \ldots, x_g}/\varpi)= 
e( R^{\psi, \square}_S(\sigma)/\varpi).$$
\end{proof}

\begin{lem}\label{regular_point} If $S_{\sigma, \psi}(U, \OO)_{\mm}$ is supported on a closed point 
$\nn\in \Spec R^{\psi, \square}_S(\sigma)[1/2]$ then the localization $R^{\psi, \square}_S(\sigma)_{\nn}$ is a regular ring. 
\end{lem}
\begin{proof} Since the rings $R_v^{\square}[1/2]$ are regular for all $v\nmid  2$ it is enough to show that $\nn$ defines a regular point 
in $\Spec R_v^{\psi, \square}(\sigma)$ for all $v\mid  2$. This follows from the proof of Lemma B.5.1 in \cite{gee_kisin}. The argument is as follows: if the point is not regular, then it must lie on the intersection of two irreducible components of $\Spec R_v^{\psi,\square}(\sigma)$, but this would violate Weight--Monodromy conjecture for $\WD(\rho_{\nn}|_{G_{F_v}})$, 
see the proof of Lemma B.5.1 in \cite{gee_kisin} for details. 
\end{proof}

\begin{lem}\label{regular_point2} If $S_{\sigma, \psi}(U, \OO)_{\mm}$ is supported on a closed point $\nn\in \Spec R_{\infty}(\sigma)[1/2]$ 
then the localization $R_{\infty}(\sigma)_{\nn}$ is a regular ring. 
\end{lem} 
\begin{proof} Let $\nn_S$ be the image of $\nn$ in $\Spec R_S^{\psi, \square}\br{x_1, \ldots, x_g}$, $\nn'$ be the image of $\nn$ in $\Spec R'_{\infty}(\sigma)$, 
via the maps in (P1), and let $\nn^{\inv}$ be the image of $\nn$ in $\Spec R^{\inv}_{\infty}(\sigma)$ via \eqref{eqn_torsor}. 
It follows from Lemma \ref{regular_point}
that $R^{\psi, \square}_S(\sigma)\br{x_1, \ldots, x_g}_{\nn_S}$ is regular ring. If the map 
\begin{equation}\label{localize_p1}
R^{\psi, \square}_S(\sigma)\br{x_1, \ldots, x_g}_{\nn_S}\twoheadrightarrow R'_{\infty}(\sigma)_{\nn'}
\end{equation}
is an isomorphism, then $R'_{\infty}(\sigma)_{\nn'}$ is a regular ring. We may assume that $L$ is sufficiently large, so that using \eqref{r_inv_1} we may write
$\nn'=(\nn^{\inv}, z_1-a_1, \ldots, z_t-a_t)$ with $a_i \in \varpi \OO$ for $1\le i\le t$. The images of $z_1-a_1, \ldots, z_t-a_t$ in 
$\nn'/(\nn')^2$ are linearly independent. Since
$$R^{\inv}_{\infty}(\sigma)_{\nn^{\inv}}\cong R'_{\infty}(\sigma)_{\nn'}/ (z_1-a_1, \ldots, z_t-a_t)R'_{\infty}(\sigma)_{\nn'},$$
we deduce that $R^{\inv}_{\infty}(\sigma)_{\nn^{\inv}}$ is regular. It follows from \eqref{eqn_torsor} that 
the map $$R^{\inv}_{\infty}(\sigma)[1/2]\rightarrow R_{\infty}(\sigma)[1/2]$$ is \'etale. Hence $R_{\infty}(\sigma)_{\nn}$ is a regular ring. 

If \eqref{localize_p1} is not an isomorphism then the dimension of the quotient must decrease. This leads to the  inequality 
$\dim R_{\infty}(\sigma)_{\nn} < \dim R_{\infty}(\sigma)-1$. Since $M_{\infty}(\sigma)$ is a Cohen--Macaulay module, as follows from (P3), its support cannot contain embedded components, hence $\dim M_{\infty}(\sigma)_{\nn}=\dim M_{\infty}(\sigma)-1$. This leads to a contradiction, as
$M_{\infty}(\sigma)_{\nn}$ is a finitely generated $R_{\infty}(\sigma)_{\nn}$-module. 
\end{proof}

\begin{lem}\label{trivial} Let $A$ be a local noetherian ring and let $(x_1, \ldots, x_d)$ be a system of parameters of $A$. If $A$ is equidimensional then every 
irreducible component of $A$ contains a closed point of $(A/(x_2, \ldots, x_d))[1/x_1]$.
\end{lem}
\begin{proof} Let $\pp$ be an irreducible component of $A$. If $A/(\pp, x_2, \ldots, x_d)[1/x_1]$ is zero then $x_1$ is nilpotent 
in $A/(\pp, x_2, \ldots, x_d)$. Since $(x_1, \ldots, x_d)$ is a system of parameters of $A$, we conclude that $A/(\pp, x_2, \ldots, x_d)$
is zero dimensional, which implies that $\dim A/\pp \le d-1$, contradicting equidimensionality of $A$. 
\end{proof}

\begin{lem}\label{still_to_do}  There is an integer $r$ independent of $\sigma$ and the choices made in the patching process such that for all $\pp\in \Spec R_{\infty}(\sigma)$
in the support of $M_{\infty}(\sigma)$ we have 
$$\dim_{\kappa(\pp)} M_{\infty}(\sigma)\otimes_{R_{\infty}(\sigma)} \kappa(\pp)\ge r$$
with equality if $\pp$ is a minimal prime of $R_{\infty}(\sigma)$ in the support of $M_{\infty}(\sigma)$.
\end{lem} 
\begin{proof} Let $\qq$ be a minimal prime of $R_{\infty}(\sigma)$ in the support of $M_{\infty}(\sigma)$. It is enough to 
show that $\dim_{\kappa(\qq)} M_{\infty}(\sigma)\otimes_{R_{\infty}(\sigma)} \kappa(\qq)$ is independent of $\qq$ and $\sigma$. Since 
$$M_{\infty}(\sigma)/(y_1,\ldots, y_{h+j}) M_{\infty}(\sigma)\cong S_{\sigma, \psi}(U, \OO)_{\mm}$$
and $S_{\sigma, \psi}(U, \OO)_{\mm}$ is a finitely generated $\OO$-module, $y_1, \ldots, y_{h+j}, \varpi$ is a system of parameters for 
$R_{\infty}(\sigma)/\qq$ and it follows from Lemma \ref{trivial} that there is a maximal ideal $\nn$ of $R_{\infty}(\sigma)[1/2]$, 
contained in $V(\qq)$ such that $S_{\sigma, \psi}(U, \OO)_{\nn}\neq 0$. It follows from (P3) that $M_{\infty}(\sigma)$ is a
Cohen-Macaulay module. The same holds for the localization at $\nn$. Since $R_{\infty}(\sigma)_{\nn}$ is a regular ring by Lemma 
\ref{regular_point2}, a standard argument with Auslander--Buchsbaum theorem shows that $M_{\infty}(\sigma)_{\nn}$ is a free 
$R_{\infty}(\sigma)_{\nn}$-module. By localizing further at $\qq$ we deduce that 
\begin{equation}
\begin{split}
\dim_{\kappa(\qq)} M_{\infty}(\sigma)\otimes_{R_{\infty}(\sigma)} \kappa(\qq)&= \dim_{\kappa(\nn)} M_{\infty}(\sigma)\otimes_{R_{\infty}(\sigma)} \kappa(\nn)\\ &=\dim_{\kappa(\nn)} S_{\sigma, \psi}(U, \OO)_{\mm}\otimes_{R_{\infty}(\sigma)} \kappa(\nn).
\end{split}
\end{equation}
So it is enough show that $\dim_{\kappa(\nn)} S_{\sigma, \psi}(U, \OO)_{\mm}\otimes_{R_{\infty}(\sigma)} \kappa(\nn)$ is independent of 
$\nn$ and $\sigma$. The action of $R_{\infty}(\sigma)$ on  $S_{\sigma, \psi}(U, \OO)_{\mm}$ factors through the action of the Hecke algebra 
$\mathbb{T}_{\sigma, \psi}(U)$, which 
is reduced. Thus $\mathbb{T}_{\sigma, \psi}(U)[1/2]$ is a product of finite field extensions of $L$ and we have
$$S_{\sigma, \psi}(U, \OO)_{\mm}\otimes_{R_{\infty}(\sigma)} \kappa(\nn)=S_{\sigma, \psi}(U, \OO)_{\nn}= (S_{\sigma, \psi}(U, \OO)_{\mm}\otimes_{\OO} L)[\nn].$$
Let $\pi=\otimes'_v \pi_v$ be the automorphic representation of $(D\otimes_F \AfF)^{\times}$ corresponding to $f^D\in  (S_{\sigma, \psi}(U, \OO)_{\mm}\otimes_{\OO} L)[\nn]$. We assume that $L$ is sufficiently large. It follows from the discussion in \cite[(3.1.14)]{kisin_moduli}, relating $S_{\sigma, \psi}(U, L)$ to the space of classical automorphic forms on $(D\otimes_F \AfF)^{\times}$, that 
$$ \dim_{L} (S_{\sigma, \psi}(U, \OO)_{\mm}\otimes_{\OO} L)[\nn]= \prod_{v\in S, v\nmid 2\infty} \dim_{L} \pi_v^{U_v} \prod_{v\mid 2}\dim_L
\Hom_{U_v}(\sigma(\tau_v), \pi_v).$$
We claim that the right hand side of the above equation is equal to $2^{|S\setminus(\Sigma\cup \{v\mid 2\infty\})|}$. The claim will follow from 
the local-global compatibility of Langlands and Jacquet--Langlands correspondences. Let $\rho_{\nn}$ be the representation 
of $G_{F, S}$ corresponding to $\nn$, considered as a maximal ideal of $R^{\psi}_{F, S}(\sigma)[1/2]$. 
If $v\mid  2$ then the results of Henniart \cite{henniart} imply that  $\dim_L\Hom_{U_v}(\sigma(\tau_v), \pi_v)=1$. If $v\in \Sigma$ 
then $\pi_v$ is unramified character of $D_v^{\times}$, and hence $\dim_{L} \pi_v^{U_v}=1$. 
If $v\in S$, $v\nmid  2\infty$ and $v\not\in \Sigma$ then $D$ is split at $v$, $\rhobar|_{G_{F_v}}$ is unramified and $\rhobar(\Frob_v)$ 
has distinct eigenvalues. This implies that $\rho_{\nn}|_{G_{F_v}}$ is an extension of distinct tamely ramified characters $\psi_1$, $\psi_2$, such that $\psi_1\psi_2^{-1}\neq \chi_{\cyc}^{\pm 1}$. We deduce that   $\pi_v$ is a tamely ramified 
principal series. Since $U_v$ is equal to the subgroup of unipotent upper-triangular matrices modulo $\varpi_v$ in this case, we deduce 
that $\dim_L \pi_v^{U_v}=2$.
\end{proof}

\begin{lem}\label{fibre_inv}   There is an integer $r$ independent of $\sigma$ and the choices made in the patching process such that for all 
minimal primes $\pp$ of $R_{\infty}^{\inv}(\sigma)$ in the support of $M_{\infty}(\sigma)$ we have
$$\dim_{\kappa(\pp)} M_{\infty}(\sigma)\otimes_{R_{\infty}^{\inv}(\sigma)} \kappa(\pp)= 2^tr.$$
\end{lem}
\begin{proof} To ease the notation, let us drop $\sigma$ from it  in this proof. Since $\pp$ is minimal, it is an associated prime and so $M_{\infty}$ will contain $R_{\infty}^{\inv}/\pp$ as a submodule. Since 
$M_{\infty}$ is $\OO$-torsion free, this implies that the quotient field $\kappa(\pp)$ has characteristic $0$. It follows from 
\eqref{eqn_torsor} that $R_{\infty}\otimes_{R^{\inv}_{\infty}} \kappa(\pp)$ is \'etale over $\kappa(\pp)$, and so 
$$R_{\infty}\otimes_{R^{\inv}_{\infty}} \kappa(\pp)\cong \prod_{\qq} \kappa(\qq),$$
where the product is taken over all prime ideals $\qq$ of $R_{\infty}$, such that $\qq \cap R_{\infty}^{\inv}=\pp$.
From this we get 
$$\dim_{\kappa(\pp)} M_{\infty}\otimes_{R_{\infty}^{\inv}} \kappa(\pp)= \sum_{\qq} [\kappa(\qq):\kappa(\pp)]
\dim_{\kappa(\qq)} M_{\infty}\otimes_{R_{\infty}} \kappa(\qq).$$
It follows from Lemma \ref{transitive_action} and (P5) that all $\qq$ appearing in the sum lie in the support of $M_{\infty}$. 
Lemma \ref{still_to_do} implies that $\dim_{\kappa(\qq)} M_{\infty}\otimes_{R_{\infty}} \kappa(\qq)=r$. Thus 
$$\dim_{\kappa(\pp)} M_{\infty}\otimes_{R_{\infty}^{\inv}} \kappa(\pp)=r \dim_{\kappa(\pp)} R_{\infty}\otimes_{R^{\inv}_{\infty}} \kappa(\pp)= r2^t,$$
where the last equality follows from Lemma \ref{torsor_fix}.
\end{proof}
\begin{lem}\label{compare_HS} 
Let $A$ be a local noetherian ring and let $M$, $N$ be finitely generated $A$-modules of  dimension $d$,
and let $x\in A$ be $M$-regular and $N$-regular. If $\ell_{A_{\qq}}(M_{\qq})\le \ell_{A_{\qq}}(N_{\qq})$ for all 
$\qq\in \Spec A$ with $\dim A/\qq =d$ then  $$e(M/xM, A/xA)\le e(N/xN, A/xA).$$ If $\ell_{A_{\qq}}(M_{\qq})= \ell_{A_{\qq}}(N_{\qq})$ for all 
$\qq\in \Spec A$ with $\dim A/\qq =d$ then  $$e(M/xM, A/xA)=e(N/xN, A/xA).$$
\end{lem}
\begin{proof} It follows from Proposition 2.2.13 in \cite{eg} that 
\begin{equation}\label{emerton_gee}
e(M/xM, A/xA)=\sum_{\qq} \ell_{A_{\qq}}(M_{\qq}) e(A/(\qq, x)),
\end{equation}
where the sum is taken over all primes $\qq$ in the support of $M$, such that $\dim A/\qq =d$. The above formula implies both 
assertions. 
\end{proof}

\begin{lem}\label{still_to_2} $e(M_{\infty}(\sigma)/\varpi, R^{\inv}_{\infty}(\sigma)/\varpi)\le 2^tr e(R^{\inv}_{\infty}(\sigma)/\varpi)$.
\end{lem} 
\begin{proof} Let $\mathbb T_{\infty}^{\inv}(\sigma)$ be the image of $R_{\infty}^{\inv}(\sigma)$ in $\End_{\OO}(M_{\infty}(\sigma))$. Then 
$$  e(\mathbb T^{\inv}_{\infty}(\sigma)/\varpi, R^{\inv}_{\infty}(\sigma)/\varpi)\le e(R^{\inv}_{\infty}(\sigma)/\varpi).$$
If $\qq$ is a minimal prime of $R_{\infty}^{\inv}(\sigma)$ in the support of $M_{\infty}(\sigma)$ then it follows from 
Lemma \ref{fibre_inv} that there are surjections $\mathbb T_{\infty}^{\inv}(\sigma)_{\qq}^{\oplus 2^tr} \twoheadrightarrow M_{\infty}(\sigma)_{\qq}$. Thus $\ell(M_{\infty}(\sigma)_{\qq})\le 2^t r \ell(\mathbb T_{\infty}^{\inv}(\sigma)_{\qq})$. The assertion follows from 
Lemma \ref{compare_HS} applied with $x=\varpi$, $M=M_{\infty}(\sigma)$ and $N=\mathbb T_{\infty}^{\inv}(\sigma)^{\oplus 2^t r}$. 
\end{proof}

\begin{lem}\label{nail} If the support of $S_{\sigma, \psi}(U, \OO)_{\mm}$ meets every irreducible component of $R^{\psi, \square}_{S}(\sigma)$ then  
the following hold:
\begin{itemize} 
\item[(i)] $R^{\psi, \square}_S(\sigma)\br{x_1, \ldots, x_g}\twoheadrightarrow R'_{\infty}(\sigma)$ is an isomorphism;
\item[(ii)] $R_{\infty}^{\inv}(\sigma)$ is reduced, equidimensional and $\OO$-flat;
\item[(iii)] $R_{\infty}(\sigma)$ is reduced, equidimensional and $\OO$-flat;
\item[(iv)] the support of $M_{\infty}(\sigma)$ meets every irreducible component of $R_{\infty}(\sigma)$;
\item[(v)]  $2^t r e( R^{\psi, \square}_S(\sigma)/\varpi)=e(M_{\infty}(\sigma)/\varpi, R_{\infty}^{\inv}(\sigma)/\varpi)$.
\end{itemize} 
\end{lem} 
\begin{proof} Since $R^{\psi, \square}_S(\sigma)\br{x_1, \ldots, x_g}$ is reduced, equidimensional and has the same dimension
as $R'_{\infty}(\sigma)$, to prove (i) it is enough to show that $R'_{\infty}(\sigma)_{\qq}\neq 0$ for every irreducible 
component $V(\qq)$ of $\Spec R^{\psi, \square}_S(\sigma)\br{x_1, \ldots, x_g}$. Since 
the diagram 
\begin{displaymath}
\xymatrix@1{R^{\psi,\square}_S(\sigma)\br{x_1,\ldots, x_g}\ar[r] & R_{\infty}(\sigma)\ar[d]\\
R^{\psi,\square}_S(\sigma)\ar[r]\ar[u]&R^{\psi}_{F, S}(\sigma)}
\end{displaymath}
commutes and the support of $S_{\sigma, \psi}(U, \OO)_{\mm}$ meets every irreducible component of $\Spec R^{\psi,\square}_S$,
$V(\qq)$ will contain 
a maximal ideal $\nn_S$ of $R^{\psi, \square}_S(\sigma)\br{x_1, \ldots, x_g}[1/2]$, which lies in the support of $S_{\sigma, \psi}(U, \OO)_{\mm}$. 
It follows from the proof of Lemma \ref{regular_point2} that \eqref{localize_p1} is an isomomorphism in this case. Thus $R'_{\infty}(\sigma)_{\qq}\neq 0$. 
 
 From part (i) we deduce that $R'_{\infty}(\sigma)$ is reduced, equidimensional and $\OO$-flat. It follows from \eqref{r_inv_1} that the same holds for $R^{\inv}_{\infty}(\sigma)$. Since $R_{\infty}(\sigma)$ is a free $R^{\inv}_{\infty}(\sigma)$-module by Lemma \ref{torsor_fix} it is $\OO$-flat. Hence, it is enough to show that $R_{\infty}(\sigma)[1/2]$
 is reduced and equidimensional. It follows from Lemma \ref{torsor_fix} that  $R_{\infty}(\sigma)[1/2]$ is \'etale over $R^{\inv}_{\infty}(\sigma)[1/2]$, which implies the assertion.  We also note that it follows from (i) that the inequality in  Lemma \ref{go_local} is an equality, and  \eqref{r_inv_1} implies that 
\begin{equation}\label{good1}
  e(R_{\infty}^{\inv}(\sigma)/\varpi)=e(R^{\psi, \square}_S/\varpi).
  \end{equation}
 It follows from our assumption that the support of $M_{\infty}(\sigma)$ meets every irreducible component of $R^{\psi, \square}_S(\sigma)\br{x_1, \ldots, x_g}$. Part (i) and \eqref{r_inv_1} imply that  the support of $M_{\infty}(\sigma)$ meets every irreducible 
component of $R^{\inv}_{\infty}(\sigma)$. It follows from Lemma \ref{transitive_action} that the group $\gm(\OO)$ acts transitively on the set of 
of irreducible components of $R_{\infty}(\sigma)$ lying above a given irreducible component of $R^{\inv}_{\infty}(\sigma)$. Thus for part 
(iii) it is enough to show that the support of $M_{\infty}(\sigma)$ in $\Spec R_{\infty}(\sigma)$ is stable under the action of 
 $\gm(\OO)$. This is given by (P5) and can be proved in the same way as \cite[Lem. 9.6]{KW}.
 
 Let $V(\qq)$ be an irreducible component of $\Spec R_{\infty}(\sigma)$. It follows from (iii) that the localization $R_{\infty}(\sigma)_{\qq}$ 
 is a reduced artinian ring, and hence is equal to the quotient field $\kappa(\qq)$. Thus $M_{\infty}(\sigma)_{\qq}\cong 
 M_{\infty}(\sigma)\otimes_{R_{\infty}(\sigma)} \kappa(\qq)$. It follows from  Lemma \ref{still_to_do} that
 $M_{\infty}(\sigma)_{\qq}$ has length $r$ as an $R_{\infty}(\sigma)_{\qq}$-module. By part (iv) $M_{\infty}(\sigma)$ is supported on every irreducible component of $R_{\infty}(\sigma)$, and thus the cycle of $M_{\infty}(\sigma)$ is equal to $r$ times the cycle of 
 $R_{\infty}(\sigma)$. Since both are $\OO$-torsion free, we deduce that the cycle of $M_{\infty}(\sigma)/\varpi$ is equal to 
 $r$ times the cycle of $R_{\infty}(\sigma)/\varpi$, which implies that 
 \begin{equation}\label{good2}
 e(M_{\infty}(\sigma)/\varpi, R_{\infty}^{\inv}(\sigma)/\varpi)= r e(R_{\infty}(\sigma)/\varpi, R_{\infty}^{\inv}(\sigma)/\varpi)= 2^t r e(R_{\infty}^{\inv}(\sigma)/\varpi).
 \end{equation}
 Part (v) follows from \eqref{good1} and \eqref{good2}.
 \end{proof}

 \begin{prop}\label{present_unrestricted} For some $s\ge 0$ there is an isomorphism of $R^{\psi, \square}_{S}$-algebras
 $$ R^{\psi, \square}_{F, S}\cong R^{\psi, \square}_{S}\br{x_1, \ldots, x_{s+|S|-1}}/(f_1, \ldots, f_s).$$
 \end{prop}
 \begin{proof} The assertion follows from the proof of Proposition 4.5 in \cite{KW}, where $s=\dim_k H^1_{\{L_v^{\perp}\}}(S, (\Ad^0)^*(1)))$ in the notation of that paper, 
 see Lemma 4.6 in \cite{KW} and the displayed equation above it. 
 \end{proof}

 \begin{cor}\label{present_restricted} For some $s\ge 0$ there is an isomorphism of $R^{\psi, \square}_{S}(\sigma)$-algebras
 $$ R^{\psi, \square}_{F, S}(\sigma)\cong R^{\psi, \square}_{S}(\sigma)\br{x_1, \ldots, x_{s+|S|-1}}/(f_1, \ldots, f_s).$$
In particular, $\dim R^{\psi, \square}_{F, S}(\sigma)\ge 4|S|$ and $\dim R^{\psi}_{F, S}(\sigma)\ge 1$.
\end{cor}
\begin{proof} Since $$R^{\psi, \square}_{F, S}(\sigma)\cong R^{\psi, \square}_{F, S}\otimes_{ R^{\psi, \square}_{S}}  R^{\psi, \square}_{S}(\sigma)$$
 the assertion follows from Proposition \ref{present_unrestricted}. Since $\dim R^{\psi, \square}_{S}(\sigma)=3|S|+1$ by \eqref{dim_S_loc}, the isomorphism implies that 
$$\dim R^{\psi, \square}_{F, S}(\sigma)\ge 3|S|+1 +s+|S|-1 -s=4|S|.$$
Since $R^{\psi, \square}_{F, S}(\sigma)$ is formally smooth over $R^{\psi}_{F, S}(\sigma)$ of relative dimension $4|S|-1$, we conclude that $\dim R^{\psi}_{F, S}(\sigma)\ge 1$.
\end{proof}

\begin{prop}\label{equivalent_cond_new} If $S_{\sigma, \psi}(U, \OO)_{\mm}\neq 0$ then the following are equivalent:
\begin{itemize} 
\item[(a)]  $2^t r e( R^{\psi, \square}_S(\sigma)/\varpi)=e(M_{\infty}(\sigma)/\varpi, R_{\infty}^{\inv}(\sigma)/\varpi)$;
\item[(b)] $2^t r e( R^{\psi, \square}_S(\sigma)/\varpi)\le  e(M_{\infty}(\sigma)/\varpi, R_{\infty}^{\inv}(\sigma)/\varpi)$;
\item[(c)] the support of $M_{\infty}(\sigma)$ meets every irreducible component of $R_{\infty}(\sigma)$; 
\item[(d)] $R^{\psi}_{F, S}(\sigma)$ is a finitely generated $\OO$-module of rank at least $1$  and 
$$S_{\sigma, \psi}(U, \OO)_{\nn}\neq 0, \quad \forall \nn\in \mSpec R^{\psi}_{F, S}(\sigma)[1/2].$$
\end{itemize} 
In this case any representation $\rho: G_{F, S}\rightarrow \GL_2(\OO)$ corresponding to a maximal ideal of $R^{\psi}_{F, S}(\sigma)[1/2]$ is modular. 
\end{prop}
\begin{proof} Lemmas \ref{still_to_2}, \ref{go_local} and \eqref{r_inv_1} imply that 
\begin{equation}\label{fortyone}
   e(M_{\infty}(\sigma)/\varpi, R^{\inv}_{\infty}(\sigma)/\varpi)\le 2^t r e( R^{\psi, \square}_S(\sigma)/\varpi).
\end{equation}   
Thus (a) is equivalent to (b). Moreover, if (a) holds then the inequalities in the Lemmas cited above have to be equalities. Since $R^{\psi, \square}_S(\sigma)$
is reduced and $\OO$-torsion free, we deduce that $R'_{\infty}(\sigma)\cong  R^{\psi, \square}_S(\sigma)\br{x_1, \ldots, x_g}$. Hence, 
$R_{\infty}'(\sigma)$ is reduced, equidimensional and $\OO$-torsion free. The isomorphism \eqref{r_inv_1} implies that the same holds for $R^{\inv}_{\infty}(\sigma)$, which implies that $R_{\infty}(\sigma)$ is reduced,  equidimensional, $\OO$-torsion free, see the proof of Lemma \ref{nail}.  Since we have assumed (a), we have 
\begin{equation}\label{multiplicity1}
2^tr e(R_{\infty}^{\inv}(\sigma)/\varpi)= e(M_{\infty}(\sigma)/\varpi, R^{\inv}_{\infty}(\sigma)/\varpi).
\end{equation}
Let  $V(\qq_1), \ldots, V(\qq_m)$ be the  irreducible components of the support of $M_{\infty}(\sigma)$ in $\Spec R_{\infty}(\sigma)$. 
Since $R_{\infty}(\sigma)$ is reduced, if $V(\qq)$ is an irreducible component of $\Spec R_{\infty}(\sigma)$ then $\ell(R_{\infty}(\sigma)_{\qq})=1$.
 It follows from Lemma \ref{still_to_do} that if $V(\qq)$ is an irreducible component of $\Spec R_{\infty}(\sigma)$ in the support of $M_{\infty}(\sigma)$ then
  $\ell(M_{\infty}(\sigma)_{\qq})=r$. It follows from \eqref{emerton_gee} that 
\begin{equation}\label{multiplicity2}
e(M_{\infty}(\sigma)/\varpi, R^{\inv}_{\infty}(\sigma)/\varpi)=r \sum_{i=1}^m e(R_\infty(\sigma)/(\varpi,\qq_i), R^{\inv}_{\infty}(\sigma)/\varpi),
\end{equation}
\begin{equation}\label{multiplicity3}
e( R_{\infty}(\sigma)/\varpi, R^{\inv}_{\infty}(\sigma)/\varpi)= \sum_{\qq} e(R_\infty(\sigma)/(\varpi,\qq), R^{\inv}_{\infty}(\sigma)/\varpi),
\end{equation}
where the last sum is taken over all the irreducible components $V(\qq)$. Since 
$e(R_\infty(\sigma)/(\varpi,\qq), R^{\inv}_{\infty}(\sigma)/\varpi)\neq 0$ we deduce from \eqref{multiplicity1}, \eqref{multiplicity2}  and \eqref{multiplicity3}  that (b) implies (c). We have 
$$ R_{\infty}(\sigma)/(y_1, \ldots, y_{h+j})\cong R^{\psi}_{F,S}(\sigma), \quad M_{\infty}(\sigma)/ (y_1, \ldots, y_{h+j})M_{\infty}(\sigma)\cong S_{\sigma, \psi}(U, \OO)_{\mm}.$$
Thus, if $M_{\infty}(\sigma)$ is supported on the whole of $\Spec R_{\infty}(\sigma)$ then $S_{\sigma, \psi}(U, \OO)_{\mm}$ is supported on the 
whole of $\Spec R^{\psi}_{F,S}(\sigma)$. Since  $S_{\sigma, \psi}(U, \OO)_{\mm}$ is a free $\OO$-module of finite rank, we deduce that (c) implies (d). 

If (d) holds then it follows from Corollary \ref{present_restricted} that $f_1, \ldots, f_s, \varpi$ is a part of a system of parameters of $R^{\psi, \square}_{S}(\sigma)\br{x_1, \ldots, x_{s+|S|-1}}$, and Lemma \ref{trivial} implies that every irreducible component 
of that ring contains 
a closed point of $R^{\psi}_{F, S}(\sigma)[1/2]$. Since every such component is of the form $\qq\br{x_1, \ldots, x_{s+|S|-1}}$, we deduce that 
every irreducible component of $R^{\psi, \square}_{S}(\sigma)$ contains a closed point of $R^{\psi}_{F, S}(\sigma)[1/2]$. It follows from the second part of (d) that 
the support of $S_{\sigma, \psi}(U, \OO)_{\mm}$ meets every irreducible component of $R^{\psi, \square}_{S}(\sigma)$. It follows from Lemma \ref{nail} that (d) implies (a). Since $S_{\sigma, \psi}(U, \OO)[1/2]$ is a finite dimensional $L$-vector space, the last assertion is a direct consequence of (d).
\end{proof}

\subsection{Small weights}\label{small_weights} Let $\tilde{\Eins}$ be the trivial representation of $\GL_2(\ZZ_2)$ on a free $\OO$-module of rank $1$. We let $\tilde{\st}$ be the 
space of functions $f: \mathbb P^1(\mathbb{F}_2)\rightarrow \OO$, such that $\sum_{x\in \mathbb P^1(\mathbb{F}_2)} f(x)=0$ equipped of the natural 
action of $\GL_2(\ZZ_2)$. The reduction of $\tilde{\Eins}$ modulo $\varpi$ is the trivial representation, the reduction of 
$\tilde{\st}$ modulo $\varpi$ is isomorphic to $k^2$, which we will also denote by $\st$. These are the only smooth irreducible $k$-representations
of $\GL_2(\ZZ_2)$.

The purpose of this subsection is to verify that the equivalent conditions of Proposition \ref{equivalent_cond_new} hold, when 
for all $v\mid 2$, $\sigma_v$ is either $\tEins$ or $\tilde{\st}$, under the assumption that $\rhobar|_{G_v}$ does not have scalar semi-simplification at any place $v\mid  2$. If $\sigma$ is the trivial representation then the result will follow from the modularity 
lifting theorem of \cite{KW}, \cite{kisin_serre2}. In the general case, our assumption implies that any semi-stable lift of 
$\rhobar|_{G_{F_v}}$ with Hodge--Tate weights $(0,1)$ is crystalline, see Corollary \ref{pot_crys}, this implies that $S_{\tEins, \psi}(U,\OO)_{\mm}$ and $S_{\sigma, \psi}(U, \OO)_{\mm}$ and $R^{\psi}_{F,S}(\tEins)$ and $R^{\psi}_{F, S}(\sigma)$ coincide. 

If $p>2$ then the results of this section are proved by Gee in  \cite{gee_prescribed} by a characteristic $p$ argument. 
  
\begin{prop}\label{main_small} Assume that $\psi$ is trivial on $U\cap (\AfF)^{\times}$, $\sigma_v=\tEins$ for all $v\mid  2$ and 
$\rhobar|_{G_v}$ does not have scalar semi-simplification for any $v\mid  2$. Then $R^{\psi}_{F, S}(\sigma)$ is a finite $\OO$-module of rank at least $1$. 
\end{prop}
\begin{proof} It follows from Lemma 2.2 in \cite{taylor_ico2} that there is a finite solvable, totally real  extension $F'$ of $F$, such that 
for all places $w$ of $F'$ above a place $v\in S$, $F'_w=F_v$, except if $v\mid  2$ and $\rhobar|_{G_v}$ is unramified, in which case 
$F'_w$ is an unramified extension of $\QQ_2$ and $\rhobar|_{G_{F'_{w}}}$ is trivial. Let $S'$ be the places of $F'$ above the places $S$ of $F$. 
By changing $F$ by $F'$ we are in position to apply Proposition 9.3 of \cite{KW}, part (II) of which says that the ring $R^{\psi}_{F', S'}(\sigma)$ is a 
finite $\OO$-module. We now argue as in the last paragraph of the proof of Theorem 10.1 of \cite{KW}. The restriction to $G_{F',S'}$ induces 
a map between the deformation functors and hence a homomorphism $R^{\psi}_{F', S'}(\sigma)\rightarrow R^{\psi}_{F, S}(\sigma)$. Let 
$\rho^{\psi}_{F, S}: G_{F, S}\rightarrow \GL_2(R^{\psi}_{F, S}(\sigma))$ be the universal deformation. Since $R^{\psi}_{F', S'}(\sigma)/\varpi$ is finite, we conclude that 
the image of $G_{F', S'}$ in $\GL_2(R^{\psi}_{F, S}(\sigma)/\varpi)$ under $\rho^{\psi}_{F,S}$ is a finite group. Since $F'/F$ is finite we conclude 
that the image of $G_{F, S}$ in $\GL_2(R^{\psi}_{F, S}(\sigma)/\varpi)$ is a finite group. Lemma 3.6 in \cite{KW_annals} implies that 
$R^{\psi}_{F, S}(\sigma)/\varpi$ is finite. Since $\dim R^{\psi}_{F, S}(\sigma) \ge 1$ by Corollary \ref{present_restricted}, we conclude that 
 $\dim R^{\psi}_{F, S}(\sigma) = 1$ and $\varpi$ is system of parameters for $R^{\psi}_{F, S}(\sigma)$, which implies that $R^{\psi}_{F, S}(\sigma)$ is a finite $\OO$-module 
 of rank at least $1$.
\end{proof}

\begin{cor}\label{D_unram_1} Assume  that $\psi$ is trivial on $U\cap (\AfF)^{\times}$, $\sigma_v=\tEins$ for all $v\mid  2$ and 
$\rhobar|_{G_v}$ does not have scalar semi-simplification for any $v\mid  2$. If $S_{\sigma, \psi}(U, \OO)_{\mm}\neq 0$ then the equivalent conditions of Proposition \ref{equivalent_cond_new} hold.
\end{cor}
\begin{proof} 
Since $S_{\sigma, \psi}(U, \OO)_{\mm}\neq 0$ and is $\OO$-torsion free, there is a maximal ideal $\nn$ of $R^{\psi}_{F, S}[1/2]$ 
such that  $S_{\sigma, \psi}(U, \OO)_{\nn}\neq 0$. This implies that $\rhobar$ satisfies the hypotheses $(\alpha)$ and $(\beta)$ made
in \S 8.2 of \cite{KW}. 

Let $\nn$ be any maximal ideal of $R^{\psi, \square}_{F,S}(\sigma)[1/2]$, and let $\rho_{\nn}$ be the corresponding representation of $G_{F, S}$. It follows from Theorem 9.7 in \cite{KW}, or 
Theorem (3.3.5) \cite{kisin_serre2} that there is a Hilbert eigenform $f$ over $F$, such that $\rho_{\nn}\cong \rho_f$. Let $\pi=\otimes_v' \pi_v$ be the corresponding automorphic representation 
of $\GL_2(\AfF)$. If $v$ is a finite place, where $D$ ramifies, then by the way we have set up 
our deformation problem $\rho_{\nn}|_{G_{F_v}}$ is isomorphic to $\bigl(\begin{smallmatrix} \gamma_v\chi_{\cyc} & \ast\\ 0 & \gamma_v\end{smallmatrix}\bigr)$, where $\gamma_v$ is an unramified character. 
The restriction of the $2$-adic cyclotomic character to $G_{F_v}$ is an unramified character, which sends the arithmetic Frobenius to $q_v\in \ZZ_2^{\times}$. Since $\rho_{\nn}$ arises from a Hilbert modular form,
the representation $\rho_{\nn}|_{G_{F_v}}$ cannot be split as in this case we would obtain a contradiction to purity of $\rho_{\nn}$, see \cite[\S 2.2]{blasius}. Hence, $\rho_{\nn}|_{G_{F_v}}$ is non-split, 
and this implies that $\pi_v$ is a twist of the Steinberg representation by an unramified character, at all $v$, where $D$ is ramified. By Jacquet--Langlands correspondence there is an eigenform $f^D\in S_{\sigma, \psi}(U, \OO)_{\mm}$
with the same Hecke eigenvalues as $f$. This implies that $S_{\sigma, \psi}(U, \OO)_{\mm}$ is supported on $\nn$.  Proposition \ref{main_small} implies that part (d) of Proposition \ref{equivalent_cond_new} holds. 
\end{proof}

\begin{lem}\label{weight_cycling} Fix a place $w$ of $F$ above $2$. Let $\sigma$ and $\sigma'$ be such that for all $v\mid  2$, $v\neq w$, $\sigma_v=\sigma_v'$ and 
is equal to either $\tilde{\Eins}$ or $\tilde{\st}$, and $\sigma_w=\tilde{\Eins}$ and $\sigma'_w= \tilde{\st}$. Assume that $\psi$ is trivial on $U\cap (\AfF)^{\times}$, 
and $\rhobar|_{G_{F_w}}$ does not have scalar semi-simplification.  Then the rings $R^{\psi}_{F,S}(\sigma)$ and 
$R^{\psi}_{F,S}(\sigma')$ are equal. 
Moreover, if $\nn$ is a maximal ideal of  $R^{\psi}_{F,S}(\sigma)[1/2]$ then $S_{\sigma, \psi}(U, \OO)_{\mm}$ is supported on $\nn$ if and only if 
$S_{\sigma', \psi}(U, \OO)_{\mm}$ is supported on $\nn$.
\end{lem}
\begin{proof} The ring $R^{\psi, \square}_{w}(\tilde{\Eins})$ parameterizes crystalline lifts of $\rhobar|_{G_{F_w}}$ with Hodge--Tate weights $(0,1)$. 
The ring $R^{\psi, \square}_{w}(\tilde{\st})$ parameterizes semi-stable lifts of $\rhobar|_{G_{F_w}}$ with Hodge--Tate weights $(0,1)$. Since both rings are
reduced and $\OO$-torsion free, we have a surjection $R^{\psi, \square}_{w}(\tilde{\st})\twoheadrightarrow R^{\psi, \square}_{w}(\tilde{\Eins})$. The assumption 
that $\rhobar|_{G_{F_w}}$ does not have scalar semi-simplification implies that every such semi-stable lift is automatically crystalline, hence the map is an isomorphism. This implies that the global deformation rings are equal, see Corollary \ref{pot_crys}.

We will deduce the second assertion from the Jacquet--Langlands correspondence and the compatibility of local and global Langlands correspondence.
Let $\tau$ be either $\sigma$ or $\sigma'$. We fix an isomorphism $i:\Qpbar\cong \CC$ and let $\tau_{\CC}=\tau\otimes_{\OO} \CC$ and let $\tau_{\CC}^*$ be the $\CC$-linear 
dual of $\tau$. Since $U\cap (\AfF)^{\times}$ acts trivially on $\tau$ by assumption, we may consider $\tau^{\ast}_{\CC}$ as a representation of $U  (\AfF)^{\times}$, on which 
$(\AfF)^{\times}$ acts by $\psi$. Let $U'=\prod_v U'_v$ be an open subgroup of $U$, such that $U'_v=U_v$, if $v\nmid  2$ and $U_v'=\{g\in U_v: g\equiv 1\pmod 2\}$ for all $v\mid  2$. Then $U'$ acts trivially on $\tau$. Let $C^{\infty}( D^{\times}\backslash (D\otimes_F \mathbb A_F)^{\times}/ U')$ be the space of smooth functions $\CC$-valued functions 
on  $D^{\times}\backslash (D\otimes_F \mathbb A_F)^{\times}$, which are invariant under $U'$. Since $U'$ is a normal subgroup of $U$, $U$ acts on this space by right translations
It follows from \cite[(3.1.14)]{kisin_moduli}, \cite[Lem. 1.3]{taylor_deg2} that we 
have an isomorphism
$$S_{\tau, \psi}(U, \OO)\otimes_{\OO} \CC\cong \Hom_{U(\AfF)^{\times}}(\tau, C^{\infty}( D^{\times}\backslash (D\otimes_F \mathbb A_F)^{\times}/ U' D_{\infty}^{\times})).$$
This isomorphism is equivariant for the Hecke operators at $v\not\in S$. The action of $R^{\psi, \square}_{F,S}(\tau)$ on $S_{\tau, \psi}(U, \OO)_{\mm}$ factors through the action of the Hecke algebra $\TT_{\tau, \psi}(U).$ Let $\nn$ be a maximal ideal of $\TT(U)_{\tau, \psi}[1/2]$. The isomorphism above implies that $S_{\tau, \psi}(U, \OO)_{\nn}$ is non-zero if and only if there is an automorphic form 
$$f^D\in C^{\infty}( D^{\times}\backslash (D\otimes_F \mathbb A_F)^{\times}/ U' D_{\infty}^{\times}),$$
 on which the Hecke operators for $v\not \in S$ act by the eigenvalues given by the 
map $\TT_{\tau, \psi}(U)\rightarrow \kappa(n)\overset{i}{\rightarrow} \CC$, and $\Hom_{U(\AfF)^{\times}}(\tau^{\ast}_{\CC}, \pi)\neq 0$, where 
 $\pi=\otimes'_v \pi_v$ is the automorphic representation corresponding to $f^{D}$.

If $S_{\sigma, \psi}(U, \OO)_{\nn}$ is non-zero then the above implies that $\Hom_{U_w}(\Eins, \pi_w)\neq 0$, which implies that $\pi_w$ is an unramified principal series representation, 
which implies that $\Hom_{U_w}(\tilde{\st},\pi_w)\neq 0$. Since $\sigma_v=\sigma'_v$ for all $v\neq w$, we conclude that $S_{\sigma', \psi}(U, \OO)_{\nn}$ is non-zero.

 If $S_{\sigma', \psi}(U, \OO)_{\nn}$ is non-zero then the same argument shows that $\Hom_{U_w}(\tilde{\st}, \pi_w)\neq 0$, which implies that $\pi_w$ is 
either  unramified principal series representation, in which case $\Hom_{U_w}(\Eins, \pi_w)\neq 0$ and thus $S_{\sigma, \psi}(U, \OO)_{\nn}\neq 0$, or 
$\pi_{w}$ is a special series. We would like to rule the last case out. By Jacquet--Langlands correspondence  to $\pi$ we may associate an automorphic representation 
$\pi'=\otimes_v' \pi'_v$ of $\GL_2(\mathbb{A}_F)$, such that $\pi_v=\pi'_v$ for all $v$, where $D$ is split. In particular, $\pi'_w=\pi_w$. Let $\rho_{\nn}$ be the representation of $G_{F,S}$ corresponding to 
the maximal ideal $\nn$ of $R^{\psi}_{F, S}[1/2]$. By the compatibility of local and global Langlands correspondence, if $\pi'_w$ is special then $\rho|_{G_{F_w}}$ is semi-stable
non-crystalline. However, this cannot happen as explained above.
\end{proof} 

\begin{cor}\label{small_OK} Assume that $\psi$ is trivial on $U\cap (\AfF)^{\times}$, for all $v\mid  2$,  
$\sigma_v$ is either $\tEins$ or 
$\tilde{\st}$ and $\rhobar|_{G_v}$ does not have scalar semi-simplification for any $v\mid  2$. If $S_{\sigma, \psi}(U, \OO)_{\mm}\neq 0$ 
then the equivalent conditions of Proposition \ref{equivalent_cond_new} hold.
\end{cor}
\begin{proof} If $\sigma_v=\tEins$ for all $v\mid  2$ then the assertion is proved in Lemma \ref{D_unram_1}. Using this case and Lemma \ref{weight_cycling} we may show that part (d) of Proposition \ref{equivalent_cond_new} is verified for all $\sigma$ as above.  
\end{proof}

\subsection{Computing Hilbert--Samuel multiplicity}

Let $\sigma=\otimes_{v\mid 2} \sigma_v$ be a continuous representation of $U$ on a finitely generated  $\OO$-module $W_{\sigma}$, where 
$\sigma_v$ are of the form \eqref{define_sigma_v} or \eqref{define_sigma_cris_v}.
Let $\psi: (\AfF)^{\times}/F^{\times}\rightarrow \OO^{\times}$ be a continuous character, such that $U\cap  (\AfF)^{\times}$ acts on $W_{\sigma}$ by the 
character $\psi$. Let $\sigmabar$ and $\psibar$ be representations obtained by reducing $\sigma$ and $\psi$ modulo $\varpi$. 
We assume that $U$ satisfies \eqref{no_isotropy}, which implies that  the subgroups $U_{Q_n}$ also satisfy \eqref{no_isotropy}. 
Hence, the functor $\sigma \mapsto S_{\sigma, \psi}(U_{Q_n}, \OO)$ is exact.
We note that since $R^{\psi, \square}_{F, S}$ is formally smooth over $R^{\psi}_{F, S}$, it is a flat $R^{\psi}_{F, S}$-module, so that the functor 
$\otimes_{R^{\psi}_{F, S}} R^{\psi, \square}_{F, S}$ is exact, and so is the localization at $\mm_{Q_n}$. Hence the functor, 
\begin{equation}\label{exact_functor}
 \sigma\mapsto M_n(\sigma)=R^{\psi, \square}_{F, S_{Q_n}}\otimes_{R^{\psi}_{F, S_{Q_n}}} S_{\sigma, \psi}(U_{Q_n}, \OO)_{\mm_{Q_n}}, \end{equation}
is exact. Following \cite[(2.2.5)]{kisinfm}  we fix a $U$-invariant filtration on $\sigmabar$ by $k$-subspaces
$$ 0=L_0\subset L_1\subset\ldots \subset L_s= W_{\sigma}\otimes_{\OO} k,$$
such that for $i=0,1, \ldots, s-1$, $\sigma_i:=L_{i+1}/L_i$ is absolutely irreducible.  Since the functor in \eqref{exact_functor} is exact, this induces 
a filtration on $M_n(\sigma)\otimes_{\OO} k$, which we denote by 
\begin{equation}\label{filtration}
0=M^0_n(\sigma)\subset M^1_n(\sigma)\subset \ldots \subset  M^s_n(\sigma)= M_n(\sigma)\otimes_{\OO} k,
\end{equation}
such that for $i=0,1, \ldots, s-1$ we have 
\begin{equation}\label{graded_piece}
M^{i+1}_n(\sigma)/ M^i_n(\sigma)\cong M_n(\sigma_i).
\end{equation} 
Each representation $\sigma_i$ is of the form $\otimes_{v\mid 2} \sigma_{i, v}$, where $\sigma_{i,v}$ is either the trivial representation, in which case we let 
$\tilde{\sigma}_{i,v}=\tEins$,   or $\st$, in which case we let $\tilde{\sigma}_{i, v}:= \tilde{\st}$. We let $\tsigma_i:=\otimes_{v\mid 2} \tilde{\sigma}_{i,v}$ and
consider it as a representation of $U$ by letting $U_v$ for $v$ not above $2$ act trivially. We note that since both $\tilde{\Eins}$ and $\tilde{\st}$ have trivial central character,  $U\cap (\AfF)^{\times}$ acts trivially on $\tsigma_i$.
We choose a continuous character $\xi: F^{\times} \backslash (\AfF)^{\times}\rightarrow \OO^{\times}$ such that $\psi\equiv \xi\pmod{\varpi}$ and the restriction of 
$\xi$ to $U\cap (\AfF)^{\times}$ is trivial, for example we could choose $\xi$ to be a Teichm\"uller lift of $\bar{\psi}$.  Let   
$$M_n(\tilde{\sigma}_i)=R^{\xi, \square}_{F, S_{Q_n}}\otimes_{R^{\xi}_{F, S_{Q_n}}} S_{\tilde{\sigma}_i, \xi}(U_{Q_n}, \OO)_{\mm_{Q_n}}.$$
The exactness of functor in \eqref{exact_functor} used with $\tsigma_i$ and $\xi$ instead of $\sigma$ and $\psi$ and \eqref{graded_piece}  gives us an isomorphism:
\begin{equation}\label{graded_piece2}
\alpha_{i,n}: M^{i+1}_n(\sigma)/ M^i_n(\sigma)\cong M_n(\sigma_i)\cong M_n(\tsigma_i)\otimes_{\OO} k.
\end{equation} 
The isomorphism $\alpha_{i,n}$ is equivariant for the action of the Hecke operators outside $S_{Q_n}$, since they act by the same formulas on all the modules. Hence \eqref{graded_piece2} 
is an isomorphism of $R^{\square}_{S}\br{x_1, \ldots, x_g}$-modules. We let $\mathfrak a_{i, n}$ 
be the   $R^{\xi, \square}_{F, S_{Q_n}}(\tsigma_i)$-annihilator of $M_n(\tsigma_i)\otimes_{\OO} k$.
Since the action of $R^{\square}_{S}\br{x_1, \ldots, x_g}$ on $M_n(\sigma)$ and $M_n(\tsigma_i)$ factors through $R^{\psi, \square}_{F, S_{Q_n}}(\sigma)$ and 
$R^{\xi, \square}_{F, S_{Q_n}}(\tsigma_i)$ respectively, we obtain a surjection
\begin{equation}\label{graded_piece3}
\varphi_{i,n} : R^{\psi, \square}_{F, S_{Q_n}}(\sigma)\twoheadrightarrow R^{\xi, \square}_{F, S_{Q_n}}(\tsigma_i)/ \mathfrak a_{i,n}.
\end{equation}

\begin{prop}\label{big_patch} We may patch in such a way that there are:
\begin{itemize}
\item an $R_{\infty}(\sigma)$-module  $M_{\infty}(\sigma)$ as in \S \ref{patching};
\item a filtration $$0=M^0_{\infty}(\sigma)\subset M^1_{\infty}(\sigma)\subset \ldots\subset M^s_{\infty}(\sigma)=M_{\infty}(\sigma)\otimes_{\OO} k$$ by $R_{\infty}(\sigma)$-submodules;
\item for each $1\le i\le s$ there is an $R_{\infty}(\tilde{\sigma}_i)$-module $M_{\infty}(\tilde{\sigma}_i)$ as in \S \ref{patching} and a surjection 
$\varphi_i: R_{\infty}(\sigma)\twoheadrightarrow R_{\infty}(\tilde{\sigma}_i) / \mathfrak a_i$, where $\mathfrak a_i$ is the $R_{\infty}(\tilde{\sigma}_i)$-annihilator 
of $M_{\infty}(\tilde{\sigma}_i)\otimes_{\OO} k$, which allows to consider $M_{\infty}(\tilde{\sigma}_i)\otimes_{\OO} k$ as an $R_{\infty}(\sigma)$-module;
\item for each $1\le i\le s$ there is an isomorphism of $R_{\infty}(\sigma)$-modules 
$$\alpha_i: M^i_{\infty}(\sigma)/ M^{i-1}_{\infty}(\sigma)\cong M_{\infty}(\tilde{\sigma}_i)\otimes_{\OO} k.$$
\end{itemize}
\end{prop}
\begin{proof} We modify the proof of \cite[Prop. 9.3]{KW}, which in turn is a modification of the proof of \cite[(3.3.1)]{kisin_moduli}. Let 
$\Delta(\sigma)_m:=(D(\sigma)_m, L(\sigma)_m, D'(\sigma)_m)$ be the patching data of level $m$ as in the proof of \cite[Prop. 9.3]{KW}, where $\sigma$ indicates the fixed weight and inertial type we are working with. In particular, $D(\sigma)_m$ and $D'(\sigma)_m$ are finite
$R_S^{\psi, \square}(\sigma)\br{x_1, \ldots, x_g}$-algebras, where $g=h+j+t-d$,  and $L(\sigma)_m$ is a module over $D(\sigma)_m$ satisfying a number of conditions, listed in the proof of \cite[Prop. 9.3]{KW}. Our patching data of level $m$ consists of tuples: 
$$\Delta_m:=(\Delta(\sigma)_m,  \{ L(\sigma)_m^i\}_{i=0}^s,   \{\Delta(\tsigma_i)_m\}_{i=1}^s,  \{\varphi_{i,m}\}_{i=1}^s, \{\alpha_{i,m}\}_{i=1}^s),$$
where $\{ L(\sigma)_m^i\}_{i=0}^s$ is a filtration of $L(\sigma)_m \otimes_{\OO} k$ by $D(\sigma)_m$-submodules, 
$\varphi_{i, m}: D(\sigma)_m \twoheadrightarrow D(\tsigma_i)_m/\mathfrak a_{i, m}$ is a surjection of $R_S^{\square}\br{x_1, \ldots, x_g}$-algebras, where $\mathfrak a_{i, m}$ is the  $D(\tsigma_i)_m$-annihilator of  $L(\tsigma_i)\otimes_{\OO} k$, $\alpha_{i, m}$ is an isomorphism of
$D(\sigma)_m$-modules between $L(\sigma)^i_m/L(\sigma)^{i-1}_m$ and $L(\tsigma_i)\otimes_{\OO} k$, where the action of 
$D(\sigma)_m$ on this last module is given by $\varphi_{i,m }$. 

An isomorphism of patching data between $\Delta_m$ and $\Delta_m'$ is a tuple $(\beta, \{\beta_i\}_{i=1}^s)$, where 
$\beta: \Delta_m(\sigma)\cong \Delta_m'(\sigma)$, $\beta_i: \Delta_m(\tsigma_i)\cong \Delta_m(\tsigma_i)$ are isomorphism of patching data in the 
sense of \cite[Prop. 9.3]{KW}, which respect the filtration and the maps   $\{\varphi_{i,m}\}_{i=1}^s, \{\alpha_{i,m}\}_{i=1}^s$.
There are only finitely many isomorphism classes of patching data of level  $m$, 
since there are only finitely many isomorphism classes of patching data of level $m$ in the sense of \cite[Prop. 9.3]{KW}, and a finite $\OO$-module can 
admit only finitely many filtrations and there are only finitely many maps between two finite modules.

We then proceed as in the proof of  \cite[Prop. 9.3]{KW}, in particular the integers $a$, $r_m$, $n_0$  and ideals $\mathfrak c_m$ and $\mathfrak b_n$ are
those defined in \textit{loc.cit.}  For an integer $n\ge n_0+1$ and for $m$ with $n\ge m\ge 3$ let $\Delta_{n,m}(\sigma)=
(D(\sigma)_{n,m}, L(\sigma)_{n,m}, D'(\sigma)_{n,m})$ be the patching data of level $m$ as in the proof of  \cite[Prop. 9.3]{KW}. 
In particular, 
$$ D(\sigma)_{n,m}= R_{n+a}(\sigma)/ (\mathfrak c_m R_{n+a}(\sigma)+ \mm_{ R_{n+a}(\sigma)}^{(r_m)}),$$
$$ L(\sigma)_{n,m}= M_{n+a}(\sigma)/ \mathfrak c_m M_{n+a}(\sigma),$$
where $R_n(\sigma):=R^{\psi, \square}_{F, S_{Q_n}}(\sigma)$. We define $\Delta_{n,m}(\tsigma_i)$ analogously with $\tsigma_i$ instead of $\sigma$ and 
with $\xi$ instead of $\psi$.  We let $(L(\sigma)^i_{n, m})_{i=1}^s$ be the filtration obtained reducing \eqref{filtration} modulo $\mathfrak c_m$. Simirlarly, 
we let    $\{\varphi_{i,n,m}\}_{i=1}^s, \{\alpha_{i,n,m}\}_{i=1}^s$ be the maps obtained by reducing \eqref{graded_piece2} and \eqref{graded_piece3} modulo $\mathfrak c_m$. Then
$$\Delta_{n,m}:=(\Delta(\sigma)_{n,m},  \{ L(\sigma)_{n,m}^i\}_{i=0}^s,   \{\Delta(\tsigma_i)_{n,m}\}_{i=1}^s,  \{\varphi_{i,n,m}\}_{i=1}^s, \{\alpha_{i,n,m}\}_{i=1}^s\}),$$
is a patching datum of level $m$ in our sense. Since there are only finitely many isomorphism classes of patching data of level $m$, after 
 replacing the sequence $$((R_{n+a}(\sigma),  M_{n+a}(\sigma)), \{(R_{n+a}(\tsigma_i), M_{n+a}(\tsigma_i))\}_{i=1}^s)_{n\ge n_0+1}$$ by a subsequence, we may assume that for each 
 $m\ge n_0+4$  and all $n\ge m$, $\Delta_{m,n}= \Delta_{m, m}$. The patching data $\Delta_{m,m}$ form a projective system, see  \cite[(3.3.1)]{kisin_moduli}.
 We obtain the desired objects by passing to the limit.  
 \end{proof} 
 
We need to control the image of $R_{\infty}^{\inv}(\sigma)$ under $\varphi_i$. Following \cite{KW} we let 
 $\CNL_{\OO}$ be the category of complete local noetherian $\OO$-algebras with a fixed isomorphism of the residue field with $k$, and whose maps are local $\OO$-algebra homomorphisms. If $A\in \CNL_{\OO}$ then 
 we let $\Sp_A: \CNL_{\OO} \rightarrow Sets$ be the functor $\Sp_A(B)=\Hom_{\CNL_{\OO}}(A, B)$. 
 Let $G$ be a finite abelian group. We let $G^*$ be the group scheme defined over $\OO$, such that 
for every $\OO$-algebra $A$, $G^*(A)=\Hom_{\mathrm{Groups}}(G, A^{\times})$. Assume that 
we are given a free $G^*$ action on $\Sp_A$.  This means that for all $B\in \CNL_{\OO}$, 
$G^*(B)$ acts on $\Sp_A(B)$ without fixed points. By Proposition 2.6 (1) in \cite{KW} the quotient 
$G^*\backslash \Sp_A$ exists in $\CNL_{\OO}$ and is represented by $(A^{\inv}, \mm^{\inv}_A)\in \CNL_{\OO}$. 
Moreover,  $\Sp_A$ is $G^*$-torsor over $\Sp_{A^{\inv}}$. 

\begin{lem}\label{torsor} Let $(A, \mm_A)$ and $(B, \mm_B)$ be in $\CNL_{\OO}$. Assume that $G^*$ 
acts freely on $\Sp_A$ and $\Sp_B$ and we are given a $G^*$-equivariant closed immersion 
$\Sp_B\hookrightarrow \Sp_A$. Then the map induces a closed immersion $\Sp_{B^{\inv}}\hookrightarrow 
\Sp_{A^{\inv}}$. 
\end{lem}
\begin{proof} Since $G^*$ acts trivially on $\Sp_{A^{\inv}}$, by the universal property of the quotient, the map $\Sp_B \rightarrow \Sp_A \rightarrow \Sp_{A^{\inv}}$ factors through $\Sp_{B^{\inv}}\rightarrow \Sp_{A^{\inv}}$. Hence, we obtain a commutative 
diagram in $\CNL_{\OO}$: 
\begin{displaymath}
\xymatrix@1{
A^{\inv} \ar[d]\ar[r] &  B^{\inv} \ar[d]\\
A \ar@{>>}[r] & B  }
\end{displaymath}
Since $\Sp_A$ is $G^*$-torsor over $\Sp_{A^{\inv}}$, it follows from \cite[Exp. VIII, Prop.4.1] {sga3}
that $A$ is a free $A^{\inv}$-module of rank $|G|$. Simirlarly, $B$ is a free $B^{\inv}$-module of rank 
$|G|$. It follows from the commutative diagram that the surjection $A\twoheadrightarrow B$ induces   
a surjection $A/\mm^{\inv}_A A \twoheadrightarrow B/\mm^{\inv}_B B$. Since both $k$-vector 
spaces have dimension $|G|$, the map is an isomorphism and this implies that the image of $\mm^{\inv}_A$ is 
equal to $\mm^{\inv}_B$. Hence, the top horizontal arrow in the diagram is surjective.    
\end{proof} 
  
 Let $\CNL_{\OO}^{[m]}$ be the full subcategory of $\CNL_{\OO}$ consisting of objects $(A, \mm_A)$, such that 
 $\mm_A^m=0$.  We have  an truncation functor $\CNL_{\OO}\rightarrow \CNL_{\OO}^{[m]}$, 
 $A\mapsto A^{[m]}:=A/\mm_A^m$. If $A$ represents the functor $X$ we denote by $X^{[m]}$ the functor represented by $A^{[m]}$. For group chunk actions, we refer the reader to \cite[\S 2.6]{KW}.  
  
\begin{lem}\label{chunk} Let $(A, \mm_A)$ and $(B, \mm_B)$ be in $\CNL_{\OO}$. Assume that $G^*$ 
acts freely on $X:=\Sp_A$ and $Y:=\Sp_B$ and we are given an isomorphism $X^{[m]}\cong Y^{[m]}$ compatible 
with the group chunk $(G^*)^{[m]}$-action. If $m$ is large enough then the image of $\mm^{\inv}_A  A$ in 
$A/\mm_A^m= B/\mm^m_B$ is equal to the image of $\mm^{\inv}_B B$. 
\end{lem}
\begin{proof} Let $X^{\inv}$ and $Y^{\inv}$ denote the quotients of $X$ and $Y$ by $G^*$. Then we have isomorphisms
$$ G^*\times X \cong X\times_{X^{\inv}} X, \quad  G^*\times Y \cong Y\times_{Y^{\inv}} Y,$$
where the map is given $(g, x)\mapsto (x, g x)$. We identify $Z:=X^{[m]}=Y^{[m]}$, $C:= A/\mm_A^m=B/\mm_B^m$. The restriction 
of the  above isomorphism to $\CNL_{\OO}^{[m]}$ gives us an isomorphism: 
$$ (G^*\times Z)^{[m]}\cong (Z\times_{X^{\inv}} Z)^{[m]}, \quad  (G^*\times Z)^{[m]}\cong (Z\times_{Y^{\inv}} Z)^{[m]}.$$
Thus an isomorphism 
$$(Z\times_{X^{\inv}} Z)^{[m]}\cong (Z\times_{Y^{\inv}} Z)^{[m]},$$
where the map is given by $(z_1, z_2)\mapsto (z_1, z_2)$. On rings this isomorphism reads 
$(C\otimes_{A^{\inv}} C)^{[m]}\cong (C\otimes_{B^{\inv}} C)^{[m]}$, $c_1\otimes c_2\mapsto c_1\otimes c_2$.

Both $A/\mm^{\inv}_A A$ and $B/\mm_B^{\inv}B$ are
$k$-vector spaces of dimension $|G|$. In particular, if $m > |G|$ then $\mm_A^m \subset \mm^{\inv}_AA$ and
$\mm_B^m\subset \mm^{\inv}_B B$. So we obtain a map $C\twoheadrightarrow A/\mm^{\inv}_A A$. If $m>2 |G|$ then 
by base changing along this map, we obtain an isomorphism:
$$ A/\mm_A^{\inv} A\otimes_k A/\mm_A^{\inv} A \cong  A/\mm_A^{\inv} A\otimes_{B^{\inv}} A/\mm_A^{\inv} A.$$ 
If the image of $B^{\inv}$ in  $A/\mm^{\inv}_A A$ is not equal to $k$ then for some $b\in B^{\inv}$, 
$1 \otimes b$ and $b \otimes 1$ will be linearly independent over $k$ in the left hand side of the above 
isomorphism and linearly dependent in the right hand side. This implies that the image of $B^{\inv}$ in  $A/\mm^{\inv}_A A$ is equal to $k$. Thus $\mm_B^{\inv} C\subset \mm_A^{\inv} C$ and by symmetry we obtain the other inclusion.
\end{proof}

Let $G_n$ be the Galois group of the maximal abelian extension of $F$ of degree a power of $2$, which is unramified 
outside $Q_n$ and split at primes in $S$. Let $G_{n, 2}=G_n/2 G_n$. It follows from \cite[Lem 5.1 (f)]{KW} that 
$G_{n,2}\cong (\ZZ/2\ZZ)^t$. Let $G^*_{n,2}$ be the group scheme defined over $\OO$, such that 
for every $\OO$-algebra $A$, $G^*_{n,2}(A)=\Hom_{\mathrm{Groups}}(G_{n,2}, A^{\times})$. 
If $A$ is a local artinian 
augmented $\OO$-algebra, $\chi\in G^*_{n,2}(A)$ and $\rho_A$ is a
a $G_{F, S_{Q_n}}$-representation lifting $\rhobar$ to $A$ then so is $\rho_A\otimes \chi$. Moreover, 
since $\chi^2$ is trivial, $\rho_A$ and $\rho_A\otimes \chi$ have the same determinant. This induces an 
action of $G^*_{n,2}$ on $\Spf R^{\square}_{F, S_{Q_n}}$, $\Spf R^{\psi, \square}_{F, S_{Q_n}}(\sigma)$ and 
$\Spf R^{\xi, \square}_{F, S_{Q_n}}(\tsigma_i)$. It follows from \cite[Lem. 5.1]{KW} that this action is free. 
Proposition 2.6 of \cite{KW} implies that the quotient by $G^*_{n,2}$ is represented by a complete local noetherian 
$\OO$-algebra, which we will denote by   $(R^{\square, \inv}_{F, S_{Q_n}}, \mm_n^{\inv})$, $(R^{\psi, \square, \inv}_{F, S_{Q_n}}(\sigma), \mm^{\inv}_{n,\sigma})$ and $(R^{\xi, \square, \inv}_{F, S_{Q_n}}(\tsigma_i), \mm^{\inv}_{n,\tsigma_i})$ respectively.  

\begin{lem}\label{fix1} The map 
$$\Spf R^{\xi, \square}_{F, S_{Q_n}}(\tsigma_i)/ \mathfrak a_{i,n}\rightarrow \Spf R^{\psi, \square}_{F, S_{Q_n}}(\sigma)$$ 
induced by \eqref{graded_piece3} is $G_{n,2}^*$-equivariant. Moreover, $$\varphi_{i, n}(\mm^{\inv}_{n,\sigma} R^{\psi, \square}_{F, S_{Q_n}}(\sigma))= \mm^{\inv}_{n, \tsigma_i}R^{\xi, \square}_{F, S_{Q_n}}(\tsigma_i)/ \mathfrak a_{i,n}.$$
\end{lem}
\begin{proof} The first part follows from \cite[Lem. 9.1]{KW}, see the paragraph after the proof of 
\cite[Prop. 7.6]{KW} and the third paragraph of the proof of \cite[Lem. 9.6]{KW}.  

 Let $q_{\sigma}: R^{\square}_{F, S_{Q_n}}\twoheadrightarrow R^{\psi, \square}_{F, S_{Q_n}}(\sigma)$ and $q_{\tsigma_i}: R^{\square}_{F, S_{Q_n}}\twoheadrightarrow R^{\xi, \square}_{F, S_{Q_n}}(\tsigma_i)$ denote the natural surjections. Since $\varphi_{i,n}\circ q_{\sigma}= q_{\tsigma_i} \pmod{\mathfrak a_{n,i}}$ it is enough to show 
that 
$q_{\sigma}( \mm_n^{\inv}R^{\square}_{F, S_{Q_n}})=\mm^{\inv}_{n,\sigma} R^{\psi, \square}_{F, S_{Q_n}}(\sigma)$
for all $\sigma$ and $\psi$ as above. This follows from Lemma \ref{torsor}.
\end{proof}

Let $\mm^{\inv}_{\sigma}$, $\mm^{\inv}_{\tsigma_i}$ be  the maximal ideals of $R_{\infty}^{\inv}(\sigma)$ and 
$R^{\inv}_{\infty}(\tsigma_i)$ respectively. 

\begin{prop}\label{final_fix} The surjection $\varphi_i: R_{\infty}(\sigma)\twoheadrightarrow R_{\infty}(\tilde{\sigma}_i) / \mathfrak a_i$ maps $\mm^{\inv}_{\sigma}R_{\infty}(\sigma)$ onto the image of 
$\mm^{\inv}_{\tsigma_i}R_{\infty}(\tsigma_i)$. In particular, 
\begin{equation}\label{fix2}
e(M^i_{\infty}(\sigma)/ M^{i-1}_{\infty}(\sigma), R^{\inv}_{\infty}(\sigma)/\varpi)=e(M_{\infty}(\tilde{\sigma}_i)\otimes_{\OO} k, R^{\inv}_{\infty}(\tsigma_i)/\varpi).
\end{equation}
\end{prop} 
\begin{proof}  If $(A,\mm)$ is  a complete local noetherian algebra then by $A^{[r]}$ we denote the ring $A/\mm^r$. We will use the same notation as in the proof of the previous Proposition. It is shown in the course 
of the proof of part (I) of \cite[Prop 9.3]{KW} that $R_{\infty}(\sigma)\cong \varprojlim_m D''_{m, m}(\sigma)$, 
where $D''_{m, n}(\sigma)= R_{n+a}(\sigma)^{[r'_m]}$. Moreover, it is shown that the map is $\gm$-equivariant, by fixing identification of $G_{n+a,2}$ with $(\ZZ/2 \ZZ)^t$. 

For each fixed $r\ge 0$ we have $R_{\infty}(\sigma)^{[r]}\cong \varprojlim_m D''_{m, m}(\sigma)^{[r]}$. Hence, 
by choosing $m$ large enough we may assume that $R_{\infty}(\sigma)^{[r]}=D''_{m, m}(\sigma)^{[r]}$ with $r\le r'_m$. 
Since $\gm$-action on $\Sp_{R_{\infty}(\sigma)}$ and on $\Sp_{R_{n+a}(\sigma)}$ is free by \cite[Lem. 9.4, 5.1]{KW}, 
we are in the situation of Lemma \ref{chunk}. Hence the image of $\mm^{\inv}_{\sigma}R_{\infty}(\sigma)$ in 
  $D''_{m, m}(\sigma)^{[r]}$ is equal to the image of $\mm^{\inv}_{m+a, \sigma} R_{m+a}(\sigma)$.
It follows from Lemma \ref{fix1} that the composition $$R_{\infty}(\sigma)\rightarrow R_{m+a}(\sigma)^{[r]}\overset{\varphi_{i,m}}{\rightarrow}
(R_{m+a}(\tsigma_i)/\mathfrak a_{i,m})^{[r]}$$
maps $\mm^{\inv}_{\sigma} R_{\infty}(\sigma)$ onto the image of 
$\mm^{\inv}_{\tsigma_i} R_{\infty}(\tsigma_i)$. The action of $R_{m+a}(\tsigma_i)$ on 
$L_{m,m}(\tsigma_i)$ factors through $R_{m+a}(\tsigma_i)^{[r'_m]}$. Since by construction 
$$\varphi_i=\varprojlim_m \varphi_{i,m}, \quad R_{\infty}(\tsigma_i)=\varprojlim_m R_{m+a}(\tsigma_i)^{[r'_m]}, \quad M_{\infty}(\tsigma_i)=
 \varprojlim_m L_{m, m}(\tsigma_i),$$ we deduce that $\varphi_i$ maps $\mm^{\inv}_{\sigma} R_{\infty}(\sigma)$ onto the image of $\mm^{\inv}_{\tsigma_i} R_{\infty}(\tsigma_i)$.
  \end{proof}

\begin{cor}\label{modular} Assume that $S_{\sigma, \psi}(U, \OO)_{\mm}\neq 0$ and that   for $v\mid 2$, $\rhobar|_{G_{F_v}}\not\cong \bigl ( \begin{smallmatrix} \chi & \ast\\ 0 & \chi\end{smallmatrix}\bigr )$ for any character $\chi: G_{F_v}\rightarrow k^{\times}$. Then 
the equivalent conditions of Proposition \ref{equivalent_cond_new} hold,  and any $\rho: G_{F, S}\rightarrow \GL_2(\OO)$ 
corresponding to a maximal ideal of $R^{\psi}_{F,S}(\sigma)[1/2]$ is modular.
\end{cor} 
\begin{proof} We will verify that part (b) of Proposition \ref{equivalent_cond_new} holds. We first note that since $S_{\sigma, \psi}(U, \OO)_{\mm}\neq 0$ and 
$U$ satisfies \eqref{no_isotropy}, there  is an $i$, such that $S_{\tsigma_i, \xi}(U, k)_{\mm}\neq 0$. This implies that $S_{\tsigma_i, \xi}(U, \OO)_{\mm}\neq 0$,
and it follows from Lemma \ref{weight_cycling} that  $S_{\tsigma_i, \xi}(U, \OO)_{\mm}\neq 0$ for all $1\le i\le s$ and $S_{\tilde{\Eins}, \xi}(U, \OO)_{\mm}\neq 0$. In particular, the rings $R_S^{\xi, \square}(\tsigma_i)$ are non-zero and equal to $R_S^{\xi, \square}(\tilde{\Eins})$. Corollary \ref{small_OK} implies that
for all $1\le i\le s$ the equality 
\begin{equation} 
 2^t r e( R^{\xi, \square}_S(\tsigma_i)/\varpi)= e(M_{\infty}(\tsigma_i)/\varpi, R^{\inv}_{\infty}(\tsigma_i)/\varpi).
\end{equation}
holds. 
Since the Hilbert--Samuel multiplicity is additive in short exact sequences, we have 
\begin{equation}
e(M_{\infty}(\sigma)/\varpi, R_{\infty}^{\inv}(\sigma)/\varpi)= \sum_{i=1}^s e(M^i_{\infty}(\sigma)/M^{i-1}_{\infty}(\sigma), R_{\infty}^{\inv}(\sigma)/\varpi).
\end{equation}
Proposition \ref{final_fix} implies that for all $1\le i\le s$ we have
\begin{equation}
e(M^i_{\infty}(\sigma)/M^{i-1}_{\infty}(\sigma), R_{\infty}^{\inv}(\sigma)/\varpi)=e(M_{\infty}(\tsigma_i)/\varpi, R_{\infty}^{\inv}(\tsigma_i)/\varpi).
\end{equation}
Thus 
\begin{equation}
e(M_{\infty}(\sigma)/\varpi, R_{\infty}^{\inv}(\sigma)/\varpi)=2^t r \sum_{i=1}^s  e( R^{\xi, \square}_S(\tsigma_i)/\varpi).
\end{equation}
Thus to verify part (b) of Proposition \ref{equivalent_cond_new} it is enough to show that 
\begin{equation}\label{to_show}
e(R^{\psi, \square}_S(\sigma)/\varpi)\le \sum_{i=1}^s  e( R^{\xi, \square}_S(\tsigma_i)/\varpi).
\end{equation}
If $A$ and $B$ are complete local $\kappa$-algebras with residue field $\kappa$ then it is shown in \cite[Prop.1.3.8]{kisinfm} that $e(A\wtimes_{\kappa} B)=e(A)e(B)$. Since $\psi$ is congruent to $\xi$ modulo $\varpi$, the inequality \eqref{to_show} reduces to the inequality on Hilbert--Samuel multiplicities of potentially semi-stable rings at all $v\mid 2$: 
\begin{equation}\label{to_show_local}
e(R^{\psi, \square}_v(\sigma_v)/\varpi) \le \sum_{i=1}^{s_v} e(R^{\xi, \square}_v(\tsigma_{v,i})/\varpi).
\end{equation} 
where $\sigma_{v,i}$ are irreducible $k$-representation of $\GL_2(\mathbb F_2)$, which appear as graded pieces of
a $\GL_2(\ZZ_2)$-invariant filtration on $\sigma_v\otimes_{\OO} k$. The inequality \eqref{to_show_local} is proved in the local part of the paper, see Remark \ref{need_bound}.
\end{proof}
\subsection{Modularity lifting} \label{modularity}

Let $F$ be a totally real field in which $2$ splits completely.

\begin{defi} An allowable base change is a totally real solvable extension $F'$ of $F$ such that $2$ splits completely in $F'$.
\end{defi}

\begin{lem}\label{weight_2} Assume that $[F:\QQ]$ is even. Let $\rhobar: G_F\rightarrow \GL_2(k)$ be a continuous absolutely irreducible representation. 
If there is a Hilbert eigenform $f$ such that $\rhobar\cong \rhobar_f$ then there is a Hilbert eigenform $g$ of parallel weight $2$, such that $\rhobar\cong \rhobar_g$ and 
at $v\mid 2$ the corresponding representation $\pi_v$ of $\GL_2(F_v)$ is either unramified principal series or a twist of Steinberg representation by an unramified character. Moreover, if for all $v\mid 2$, $\rhobar|_{G_{F_v}}\not\cong \bigl ( \begin{smallmatrix} \chi & \ast\\ 0 & \chi\end{smallmatrix}\bigr )$ for any character $\chi: G_{F_v}\rightarrow k^{\times}$ then we may assume that $\pi_v$ is an unramified principal series representation for all $v\mid 2$.  
\end{lem} 
\begin{proof} Let $D$ be the totally definite quaternion algebra with centre $F$ split at all the finite places. Let $f^D\in S_{\tau, \psi}(U, \OO)$ be the eigenform on $D$ associated to $f$ by the  Jacquet--Langlands correspondence, where $U=\prod_{v} U_v$ is a compact  open subgroup of $(D\otimes_F \AfF)^{\times}$, 
such that $U_v=\GL_2(\OO_{F_v})$ for all $v\mid 2$, and $U$ is sufficiently small, so that  \eqref{no_isotropy} holds, and $\tau=\otimes_{v\mid 2} \tau_v$ 
a locally algebraic representation of $U$. Let $\mm$ be the maximal ideal of the Hecke algebra $\mathbb{T}^{\univ}_{S, \OO}$ corresponding to $\rhobar$, then 
$f^D\in S_{\tau, \psi}(U, \OO)_{\mm}$, and hence $S_{\tau, \psi}(U, \OO)_{\mm}$ is non-zero. 

 Let $\bar{\tau}$ denote the reduction
of a $U$-invariant lattice in $\tau$, and let $\bar{\psi}$ denote $\psi$ modulo $\varpi$. Since $U$ satisfies \eqref{no_isotropy} the functor $\sigma\mapsto S_{\sigma, \psi}(U, \OO)$ is exact. The localization functor is also exact. Hence there is an irreducible subquotient $\sigma$ of $\bar{\tau}$ such that $S_{\sigma, \bar{\psi}}(U, k)_{\mm}$ is non-zero. Such $\sigma$ is of the form $\otimes_{v\mid 2} \sigma_v$, where $\sigma_v$ is an representation of $\GL_2(\mathbb{F}_2)$.
Thus $\sigma_v$ is either trivial, in which case we let $\tilde{\sigma}_v=\tilde{\Eins}$, or $k^2$, in which case we let $\tilde{\sigma}_v=\tilde{\st}$. 
Then the reduction of $\tilde{\sigma}_v$ modulo $\varpi_v$ is isomorphic to $\sigma_v$ and $F_v^{\times}\cap U_v$ acts trivially on $\tilde{\sigma}_v$.
 Let 
$\tilde{\sigma}:=\otimes_{v\mid 2} \tilde{\sigma}_v$. Choose a lift $\xi: (\AfF)^{\times}/F^{\times}\rightarrow \OO^{\times}$ of $\bar{\psi}$, which is trivial on 
$U\cap (\AfF)^{\times}$. The exactness of the functor $\sigma\mapsto S_{\sigma, \xi}(U, \OO)$ implies that  $S_{\tilde{\sigma}, \xi}(U, \OO)_{\mm}$ is non-zero, since its reduction modulo $\varpi$ is  equal to $S_{\sigma, \xi}(U, k)_{\mm}$. We may take any eigenform 
$g^D\in S_{\tilde{\sigma}, \tilde{\psi}}(U, \OO)_{\mm}$ and then using Jacquet--Langlands transfer it to a Hilbert modular form, which will have the prescribed properties. The last part follows from Lemma \ref{weight_cycling}.
\end{proof}

\begin{thm} Let $F$ be a totally real field where $2$ is totally split, and let 
$$\rho:G_{F,S}\rightarrow \GL_2(\OO)$$
be a continuous representation. Suppose that 
\begin{itemize}
\item[(i)] $\rhobar: G_{F,S}\overset{\rho}{\rightarrow} \GL_2(\OO)\rightarrow \GL_2(k)$ is modular with non-solvable image;
\item[(ii)] For $v\mid 2$, $\rho|_{G_{F_v}}$ is potentially semi-stable with distinct Hodge--Tate weights;
\item[(iii)] $\det \rho$ is totally odd; 
\item[(iv)] for $v\mid 2$, $\rhobar|_{G_{F_v}}\not\cong \bigl ( \begin{smallmatrix} \chi & \ast\\ 0 & \chi\end{smallmatrix}\bigr )$ for any character $\chi: G_{F_v}\rightarrow k^{\times}$.
\end{itemize}
Then $\rho$ is modular. 
\end{thm}

\begin{proof} Let $\psi=\chi^{-1}_{\cyc} \det \rho$, where $\chi_{\cyc}$ is the $2$-adic cyclotomic character. By solvable base change it is enough to prove the assertion for the restriction of $\rho$ to $G_{F'}$,
where $F'$ is a totally real solvable extension of $F$. Using Lemma 2.2 of \cite{taylor_ico2} we may find an allowable base change $F'$ of $F$ such that 
$[F':\QQ]$ is even and $\rhobar|_{G_{F'}}$ is unramified outside places above $2$. We may further assume that if $\rho$ is ramified at $v\nmid 2$ then 
the image of inertia is unipotent. Let $\Sigma$ be the set of places outside $2$, where $\rho$ is ramified. If $v\in \Sigma$ then 
$$\rho|_{G_{F'_v}}\cong   \begin{pmatrix} \gamma_v \chi_{\cyc}& \ast\\ 0 & \gamma_v\end{pmatrix}$$
where $\gamma_v$ is an unramified character, such that $\gamma_v^2=\psi|_{G_{F'_v}}$.

Since $\rhobar$ is assumed to be modular, Lemma \ref{weight_2} implies that $\rhobar\cong \rhobar_f$, where $f$ is a Hilbert eigenform of parallel weight $2$, and unramified principal series at $v\mid 2$. Using Lemma 3.5.3 of \cite{kisin_moduli}, see also Theorem 8.4 of \cite{KW}, there is an admissible base change $F''/F'$, such that $\rho|_{G_{F''}}$ is ramified at even number of places outside $2$, we still denote this set by $\Sigma$, and there is a
Hilbert eigenform $g$ over $F''$, such that $\rhobar|_{G_{F''}}\cong \rhobar_g$, $g$ has parallel weight $2$, is special of conductor $1$ at $v\in \Sigma$, and is unramified otherwise. 

Let $D$ be the quaternion algebra with centre $F''$ ramified exactly at all infinite places and all $v\in \Sigma$.
Choose a place $v_1$ of $F''$ as in Lemma \ref{find_prime} and such that $\rhobar$ is unramified at $v_1$ and  $\rhobar(\Frob_{v_1})$ has distinct eigenvalues. Let $S$ be the union of infinite places, $\Sigma$, places above $2$ and $v_1$. Let $U=\prod_v U_v$ be an open subgroup of $(D\otimes_{F''} \mathbb A_{F''}^f)^{\times}$, such that $U_v= \OO_{D_v}^{\times}$, if $v\neq v_1$ and $U_{v_1}$ is unipotent upper triangular modulo $\varpi_{v_1}$. 
We note that Lemma \ref{find_prime} implies that $U$ satisfies \eqref{no_isotropy}.
Let $\mm$ be the maximal ideal in the Hecke algebra $\mathbb T^{\univ}_{S, \OO}$ corresponding to $\rhobar$. 

Let $g^D$ be the eigenform on $D$ corresponding to $g$ via the Jacquet--Langlands correspondence. Then $g^D\in S_{\sigma, \psi'}(U, \OO)_{\mm}$, where $\sigma$ is the trivial representation of $U$ and $\psi':(\mathbb A_{F''}^f)^{\times}\rightarrow \OO^{\times}$ is a suitable character congruent to $\psi$ modulo $\varpi$. In particular, $S_{\sigma, \psi'}(U, \OO)_{\mm}\neq 0$. It follows from Lemma \ref{weight_cycling} that $S_{\sigma, \psi'}(U, \OO)_{\mm}\neq 0$, for all $\sigma=\otimes_{v\mid 2} \sigma_v$, 
where $\sigma_v$ is either $\tEins$ or $\tilde{\st}$. Since $U$ satisfies \eqref{no_isotropy}, we deduce that $S_{\sigma, \psi}(U, k)_{\mm}\neq 0$ for 
any irreducible smooth $k$-representation $\sigma$ of $\prod_{v\mid 2} \GL_2(\ZZ_2)$. Since $U$ satisfies \eqref{no_isotropy}, we deduce via Lemma 3.1.4 of 
\cite{kisin_moduli} that $S_{\sigma, \psi}(U, \OO)_{\mm}\neq 0$ for any continuous finite dimensional representation $\sigma$ of $\prod_{v\mid 2} \GL_2(\ZZ_2)$ on which $U\cap (\mathbb A_{F''}^f)^{\times}$ acts by $\psi$.

For $v\mid 2$ suppose that $\rho|_{G_{F''_v}}$ has Hodge--Tate weights  $\mathbf w_v=(a_v, b_v)$ with $b_v > a_v$ and intertial type $\tau_v$. Let 
$\sigma_v$ be defined by \eqref{define_sigma_v} and let $\sigma=\otimes_{v\mid 2} \sigma_v$. The above implies that $S_{\sigma, \psi}(U, \OO)_{\mm}\neq 0$
and, since $\rho|_{G_{F''}}$ defines a maximal ideal of $R^{\psi}_{F'', S}[1/2]$, the assertion follows from Corollary \ref{modular}. 
\end{proof}


\begin{thebibliography}{49}
\bibitem{bl}\textsc{L.\ Barthel and R.\ Livn\'e}, `Irreducible modular representations of $\GL_2$ of a local field',  \textit{Duke Math. J. } 75 (1994), no. 2, 261--292.
\bibitem{blasius} \textsc{D. Blasius}, `Hilbert modular forms and the Ramanujan conjecture', \textit{Noncommutative geometry and number theory}, Aspects Math., E37, Vieweg, Wiesbaden, 2006, pp. 35--56.
\bibitem{bm}\textsc{C. Breuil and A. M\'ezard}, `Multiplicit\'es modulaires et rep\-r\'e\-sen\-ta\-tions de $\GL_2(\Zp)$ et de $\Gal(\Qpbar/\Qp)$ en $l=p$', 
\textit{Duke Math. J.} 115, 2002, 205--310.
\bibitem{bb}\textsc{L.Berger and C. Breuil}, `Sur quelques repr\'esentations potentiellement cristallines de $\GL_2(\Qp)$',
 \textit{Ast\'erisque} 330, 2010, 155--211.
 \bibitem{bockle}\textsc{G.~B\"ockle}, `Deformation rings for some mod $3$ Galois representations of the absolute Galois group of
 $\mathbf{Q}_3$',  \textit{Ast\'erisque} 330 (2010), 529--542.
 \bibitem{breuil1}\textsc{C. Breuil}, `Sur quelques repr\'esentations modulaires et $p$-adiques de $\GL_2(\Qp)$. I', \textit{Compositio} 138, 165--188, 2003.
\bibitem{breuil2}\textsc{C. Breuil}, 'Sur quelques repr\'esentations modulaires et $p$-adiques de $\GL_2(\Qp)$. II', \textit{J. Inst. Math. Jussieu} 
\bibitem{be}\textsc{C. Breuil and M. Emerton}, `Repr\'esentations $p$-adiques ordinaires de $\GL_2(\Qp)$ et compatibilit\'e local-global', \textit{Ast\'erisque} 331 (2010) 255--315.
\bibitem{che2}  \textsc{G.~Chenevier}, `Sur la vari\'et\'e des caract\`eres $p$-adiques du groupe de Galois absolu de $\Qp$', preprint 2009. 
\bibitem{che_det} \textsc{G. Chenevier},  `The $p$-adic analytic space of pseudocharacters of a profinite group and pseudorepresentations over arbitrary rings', \textit{Automorphic forms and Galois representations}, Volume 1, LMS Lecture 
Note Series 414, (Eds. \textsc{F. Diamond, P.L. Kassaei and M. Kim}) (2014), 221--285.
\bibitem{colmez}\textsc{P. Colmez}, `Repr\'esentations de $\GL_2(\Qp)$ et $(\varphi,\Gamma)$-mo\-du\-les', \textit{As\-t\'e\-ris\-que} 330 (2010) 281--509.
\bibitem{PCD}\textsc{P. Colmez, G. Dospinesu and V. Pa\v{s}k\={u}nas}, `The $p$-adic local Langlands correspondence for $\GL_2(\Qp)$', \textit{Camb. J. Math.} 2 (2014), no. 1, 1--47. 
\bibitem{surjectivity}\textsc{P. Colmez, G. Dospinesu and V. Pa\v{s}k\={u}nas}, `Irreducible components of deformation spaces: wild $2$-adic exercises', 
\textit{Int. Math. Res. Not. IMRN}, 2015, no.14, 5333--5356.
\bibitem{ddt} \textsc{H. Darmon, F. Diamond, and R. Taylor}, `FermatÕs last theorem', \textit{Elliptic
curves, modular forms and Fermat's last theorem} (Hong Kong, 1993), Int. Press,
Cambridge, MA, 1997, pp. 2--140.
\bibitem{sga3} \textsc{M. Demazure and  A. Grothendieck}, \textit{ Sch\'emas en groupes. II: Groupes de type multiplicatif, et structure des sch\'emas en groupes g\'en\'eraux.} S\'eminaire de G\'eom\'etrie Alg\'ebrique du Bois Marie 1962/1964 (SGA 3). Dirig\'e par M. Demazure et A. Grothendieck. Lecture Notes in Mathematics, vol. 152. Springer, Berlin (1962/1996).
\bibitem{dospinescu}\textsc{G. Dospinescu}, `Extensions de repr\'esentations de de Rham et vecteurs localement alg\'ebriques', 
\textit{Compositio}, to appear. 
 \bibitem{ord1}\textsc{M. Emerton}, `Ordinary Parts of admissible representations of $p$-adic reductive groups I. Definition 
and first properties', \textit{Ast\'erisque} 331 (2010), 335--381.
\bibitem{ord2}\textsc{M. Emerton}, `Ordinary Parts of admissible representations of $p$-adic reductive groups II. Derived functors',
  \textit{Ast\'erisque} 331 (2010), 383--438.
\bibitem{lg}\textsc{M. Emerton}, `Local-global compatibility in the $p$-adic Langlands programme for $\GL_2/\QQ$', preprint (2011).  
\bibitem{eg} \textsc{M. Emerton and T. Gee}, `A geometric perspective on the Breuil--M\'ezard conjecture', \textit{Journal de l'Institut de Math\'ematiques de Jussieu},  13 (2014), no. 1, 183--223.
\bibitem{eff}\textsc{M. Emerton and V.Pa\v{s}k\={u}nas}, `On the effaceability of certain $\delta$-functors', \textit{Ast\'erisque} 331
(2010) 461--469.
\bibitem{fontaine}\textsc{J.-M. Fontaine}, `Repr\'esentations $p$-adiques semi-stables',  \textit{Ast\'erisque} 223 (1994) 113--184.
\bibitem{gabriel} \textsc{P. Gabriel}, 'Des cat\'egories ab\'eliennes', \textit{Bull. Soc. Math. France} 90 1962 323--448. 
\bibitem{gee_prescribed}\textsc{T. Gee}, `Automorphic lifts of prescribed types', \textit{Mathematische Annalen}, Volume 350, Number 1, 107--144 (2011).
\bibitem{gee_kisin} \textsc{T. Gee and M. Kisin}, `The Breuil--M\'ezard conjecture for potentially Barsotti--Tate representations', 
\textit{Forum Math Pi}, (2014), e1, 56 pp.
\bibitem{henniart}\textsc{G. Henniart}, `Sur l'unicit\'e des types pour $\GL_2$', Appendix to \cite{bm}.
\bibitem{hu_tan}\textsc{Y. Hu and F. Tan}, `The Breuil-M\'ezard conjecture for non-scalar split residual representations', 	\textit{Ann. Scient. de l'E.N.S.}, to appear, \url{http://arxiv.org/abs/1309.1658}.
\bibitem{KM74} \textsc{N. M. Katz and W. Messing}, `Some consequences of the Riemann hypothesis for varieties over finite fields', \textit{Invent. Math.} 23 (1974), 73--77.
\bibitem{KW1}\textsc{C. Khare and J.-P. Wintenberger}, `Serre's modularity conjecture (I)', \textit{Invent. Math.} (2009) 178,  485--504.
\bibitem{KW}\textsc{C. Khare and J.-P. Wintenberger}, `Serre's modularity conjecture (II)', \textit{Invent. Math.} (2009) 178, 505--586.
\bibitem{KW_annals} \textsc{C. Khare and J.-P. Wintenberger}, `On Serre's conjecture for $2$-dimensional mod $p$ representations of $\Gal(\overline{\QQ}/\QQ)$', \textit{Annals of Mathematics}, 169 (2009), 229--253.
\bibitem{kisin_pst} \textsc{M. Kisin}, `Potentially semi-stable deformation rings', \textit{J. Amer. Math. Soc.} 21 (2008), no. 2, 513--546. 
\bibitem{kisin_serre2}\textsc{M. Kisin}, `Modularity of $2$-adic Barsotti--Tate representations', \textit{Invent. Math.}  (2009) 178: 587--634.
\bibitem{kisin_moduli} \textsc{M. Kisin}, `Moduli of finite flat group schemes and modularity', \textit{Annals of Mathematics}, 170 (2009), 1085--1180.
\bibitem{kisinfm}\textsc{M. Kisin}, `The Fontaine-Mazur conjecture for $\GL_2$, \textit{J.A.M.S} 22(3) (2009) 641--690.
\bibitem{kisin}\textsc{M. Kisin}, `Deformations of $G_{\Qp}$ and $\GL_2(\Qp)$  representations', \textit{Ast\'erisque} 330 (2010), 511--528.
\bibitem{except}\textsc{V. Pa\v{s}k\={u}nas} 'On some crystalline representations of ${\rm GL}_2(\mathbb{Q}_p)$', \textit{Algebra Number Theory}  3  (2009),  no. 4, 411--421. 
\bibitem{ext}\textsc{V. Pa\v{s}k\={u}nas}, `Extensions for supersingular representations of $\GL_2(\Qp)$', \textit{Ast\'erisque} 331 (2010) 317--353.
\bibitem{comp}\textsc{V. Pa\v{s}k\={u}nas}, `Admissible unitary completions of locally $\Qp$-rational representations of $\GL_2(F)$', \textit{Represent. Theory} 14  (2010), 324--354. 
\bibitem{mybm}\textsc{V. Pa\v{s}k\={u}nas}, `On the Breuil-M\'ezard conjecture', \textit{Duke Math. J.} 164 (2015), no. 2, 297--359.
\bibitem{cmf}\textsc{V. Pa\v{s}k\={u}nas}, `The image of Colmez's Montreal functor',  \textit{Publ. Math. Inst. Hautes \'Etudes Sci.} 118 (2013), 1--191.
\bibitem{blocks}\textsc{V. Pa\v{s}k\={u}nas}, `Blocks for mod $p$ representations of $\GL_2(\Qp)$, \textit{Automorphic forms and Galois representations}, Volume 2, LMS Lecture Note Series 415, (Eds. \textsc{F. Diamond, P.L. Kassaei and M. Kim}), (2014), 231--247.
\bibitem{versal}\textsc{V. Pa\v{s}k\={u}nas}, `On $2$-adic deformations', preprint 2015, \url{http://arxiv.org/abs/1509.00320}. 
\bibitem{Sai09} \textsc{T. Saito}, `Hilbert modular forms and $p$-adic Hodge theory', \textit{Compos. Math.} 145
(2009), no. 5, 1081--1113.
 \textit{Documenta Mathematica}, The Book Series 4 (J. Coates' Sixtieth Birthday), 631 - 684 (2006).
 \bibitem{iw}\textsc{P. Schneider and J. Teitelbaum}, `Banach space representations and Iwasawa theory', \textit{Israel J. Math.} 127, 359--380 (2002).
\bibitem{mult}\textsc{J.-P. Serre}, \textit{Local algebra}, Springer 2000.
\bibitem{taylor_deg2}\textsc{R. Taylor}, `On the meromorphic continuation of degree two $L$-functions',  
\textit{Documenta Mathematica}, Extra Volume: John Coates' Sixtieth Birthday (2006), 729--779.
\bibitem{taylor_ico2}\textsc{R.Taylor}, ` On icosahedral Artin representations. II',  
 \textit{American Journal of Mathematics} 125 (2003), 549--566.
\bibitem{thorne} \textsc{J. Thorne}, `A $2$-adic automorphy lifting theorem for unitary groups over CM fields', preprint (2015).


\end{thebibliography}
\end{document}